\pdfoutput=1
\RequirePackage{ifpdf}
\ifpdf % We~are running pdfTeX in pdf mode
\documentclass[pdftex]{sigma}
\else
\documentclass{sigma}
\fi

\numberwithin{equation}{section}

\newtheorem{Theorem}{Theorem}[section]
\newtheorem*{Theorem*}{Theorem}
\newtheorem{Corollary}[Theorem]{Corollary}
\newtheorem{Lemma}[Theorem]{Lemma}
\newtheorem{Proposition}[Theorem]{Proposition}
{ \theoremstyle{definition}
	\newtheorem{Definition}[Theorem]{Definition}
	
	\newtheorem{Example}[Theorem]{Example}
	\newtheorem{Remark}[Theorem]{Remark} }

\usepackage[all]{xy}

\newcommand{\ad}{\mathrm{ad}}
\newcommand{\id}{\mathrm{id}}
\newcommand{\End}{\mathrm{End}}
\newcommand{\Hom}{\mathrm{Hom}}

\begin{document}
\allowdisplaybreaks

\newcommand{\arXivNumber}{2207.00390}

\renewcommand{\PaperNumber}{018}

\FirstPageHeading

\ShortArticleName{Differential Antisymmetric Infinitesimal Bialgebras}

\ArticleName{Differential Antisymmetric Infinitesimal Bialgebras,\\ Coherent Derivations and Poisson Bialgebras}

\Author{Yuanchang LIN~$^{\rm a}$, Xuguang LIU~$^{\rm b}$ and Chengming BAI~$^{\rm a}$}

\AuthorNameForHeading{Y.~Lin, X.~Liu and C.~Bai}

\Address{$^{\rm a)}$~Chern Institute of Mathematics \& LPMC, Nankai University, Tianjin 300071, P.R.~China}
\EmailD{\href{mailto:linyuanchang@mail.nankai.edu.cn}{linyuanchang@mail.nankai.edu.cn}, \href{mailto:baicm@nankai.edu.cn}{baicm@nankai.edu.cn}}

\Address{$^{\rm b)}$~Department of Mathematics, University of California, Santa Cruz, CA 95064, USA}
\EmailD{\href{mailto:xliu270@ucsc.edu}{xliu270@ucsc.edu}}

\ArticleDates{Received September 28, 2022, in final form March 16, 2023; Published online April 04, 2023}

\Abstract{We establish a bialgebra theory for differential algebras, called differential antisymmetric infinitesimal (ASI) bialgebras by generalizing the study of ASI bialgebras to the context of differential algebras, in which the derivations play an important role. They are characterized by double constructions of differential Frobenius algebras as well as matched pairs of differential algebras. Antisymmetric solutions of an analogue of associative Yang--Baxter equation in differential algebras provide differential ASI bialgebras, whereas in turn the notions of $\mathcal{O}$-operators of differential algebras and differential dendriform algebras are also introduced to produce the former. On the other hand, the notion of a coherent derivation on an ASI bialgebra is introduced as an equivalent structure of a differential ASI bialgebra. They include derivations on ASI bialgebras and the set of coherent derivations on an ASI bialgebra composes a Lie algebra which is the Lie algebra of the Lie group consisting of coherent automorphisms on this ASI bialgebra. Finally, we apply the study of differential ASI bialgebras to Poisson bialgebras, extending the construction of Poisson algebras from commutative differential algebras with two commuting derivations to the context of bialgebras, which is consistent with the well constructed theory of Poisson bialgebras. In particular, we construct Poisson bialgebras from differential Zinbiel algebras.}


\Classification{16T10; 16T25; 16W99; 17A30; 17B62; 57R56; 81R60}

\section{Introduction}

The aim of this paper is to develop a bialgebra theory for differential algebras and get some applications.
The notion of differential antisymmetric infinitesimal bialgebras is introduced,
giving a new kind of derivations, called coherent derivations, on antisymmetric infinitesimal bialgebras.
As an application, we generalize the typical construction of Poisson algebras from commutative differential algebras to the context of bialgebras.

\subsection{Differential algebras}

The notion of a differential algebra was introduced in \cite{Ritt1950Differential},
which could be regarded as an associative algebra with finitely many commuting derivations.
Such structures sprang from the classical study of algebraic differential equations with meromorphic functions as coefficients~\cite{kolchin1973differential},
and the abstraction of differential operators led to the development of differential algebras,
which in turn have influenced other areas such as Diophantine geometry, computer algebra and model theory~\cite{aschenbrenner2017asymptotic, buium1975differential, pillay1995model}.

Differential algebras also have applications in mathematical physics such as Yang--Mills theory.
A $K$-cycle over the tensor product algebra is used in
\cite{connes1992metric} to derive the full standard model, which
turns out to be called Connes--Lott model, and the differential
algebras become the basic mathematical structure in the
construction. More details about differential algebras used in
such an approach can be found in \cite{connes1994non, iochum1996yang, kalau1996hamilton, kalau1995differential, matthes1996structure}.

The study of differential algebras themselves is also plentiful. For example, the
operad of differential algebras with one derivation was studied in~\cite{loday2010operad} and the free structure of differential
algebras was studied in~\cite{guo2008differential}.
On the other hand, some algebra structures have been found having tight connections with differential algebras.
As a typical example, there is a Poisson algebra obtained from a commutative differential algebra with two commuting derivations.
More examples include the Lie algebras obtained from commutative differential algebras with one derivation~\cite{Xu}, the 3-Lie algebras~\cite{nambu1995generalized} and more general the $n$-Lie algebras~\cite{filippov1985n} obtained from commutative differential algebras.
Furthermore, every Novikov algebra is illustrated to be embedded into a commutative differential algebra with one derivation~\cite{bokut2017grobner}.\looseness=1

\subsection{Differential antisymmetric infinitesimal bialgebras}\label{Sec:1.2}

A bialgebra structure consists of an algebra structure and a coalgebra structure coupled by certain compatibility conditions.
Such structures play important roles in many areas and have connections with other structures arising from mathematics and physics.
For example, Lie bialgebras for Lie algebras are the algebra structures of Poisson--Lie groups, and play an important role in the study of quantized universal enveloping algebras \cite{chari1995guide,drinfegammad1983hamiltonian}.
For associative algebras, there are two bialgebra structures with different compatibility conditions.
They are associative bialgebras in Hopf algebras with the comultiplication being a~homomorphism and antisymmetric infinitesimal (ASI) bialgebras with the comultiplication being a derivation.
For the former, Hopf algebras have their origin in algebraic topology, serve as universal enveloping algebras of Lie algebras and provide a basic algebraic framework for the study of quantum groups~\cite{abe2004hopf}.
The latter is the associative analog of the Lie bialgebras, which can be characterized as double constructions of Frobenius algebras, widely applied in 2d topological field and string theory~\cite{AguiarOn, BaiDouble, joni1979coalgebras, kockfrobenius, lauda2008open}.
Note that the structure of ASI bialgebras first appeared in~\cite{zhelyabian1997jordan} where they were called ``associative D-bialgebras'' (i.e., in the sense of Drinfeld), and later in~\cite{AguiarOn} where they were called ``balanced infinitesimal bialgebras'' in the sense of the opposite algebras.\looseness=1

In this paper, we establish a bialgebra theory for differential
algebras, called differential antisymmetric infinitesimal (ASI)
bialgebras, by extending the study of ASI bialgebras in~\cite{BaiDouble} to the context of differential algebras. The
derivations in a differential algebra play an important role in
such generalizations by introducing an admissibility condition
between the linear operators and the differential algebra, which
gives a reasonable bimodule of the differential algebra on
the dual space. Such an admissibility leads to the compatibility
among the multiplication, the comultiplication and the linear
operators, so that the theory of ASI bialgebras can be extended to
differential ASI bialgebras.

Explicitly, differential ASI bialgebras are characterized
equivalently by matched pairs of differential algebras and double
constructions of differential Frobenius algebras, as the
generalizations of matched pairs of algebras and double
constructions of Frobenius algebras respectively to the context of differential
algebras. The coboundary cases lead to the introduction of the
notion of admissible associative Yang--Baxter equation (AYBE) whose
antisymmetric solutions are used to construct differential ASI
bialgebras. The notions of $\mathcal{O}$-operators of differential
algebras and differential dendriform algebras are introduced to
construct antisymmetric solutions of admissible AYBE in
differential algebras and hence give rise to differential ASI bialgebras.
We summarize these results in the following diagram:
$$
 
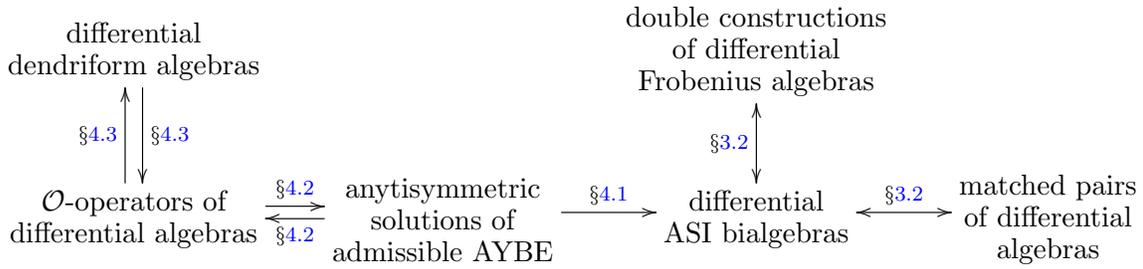
\xymatrix@C=0.75cm{
 \txt{differential \\ dendriform algebras} \ar@{->}[d]<0.7ex>^-{\S \ref{subsec:dendriform}} & &
 \txt{double constructions \\ of differential \\ Frobenius algebras}\ar@{<->}[d]_-{\S \ref{subsec:dasi}}\\
 \txt{$\mathcal{O}$-operators of \\ differential algebras} \ar@{->}[u]<0.7ex>^-{\S \ref{subsec:dendriform}} \ar@{->}[r]<0.5ex>^-{\S \ref{subsec:o_operator}} &
 \txt{anytisymmetric \\ solutions of \\ admissible AYBE} \ar@{->}[l]<0.5ex>^-{\S \ref{subsec:o_operator}} \ar@{->}[r]^-{\S \ref{subsec:cobdasi}} &
 \txt{differential \\ ASI bialgebras} & \txt{matched pairs \\ of differential \\ algebras} \ar@{<->}[l]_-{\S \ref{subsec:dasi}}
 }
$$

\subsection{Coherent derivations on ASI bialgebras}

We have already regarded differential ASI bialgebras as a bialgebra theory for differential algebras.
On the other hand, equivalently, we also regard differential ASI bialgebras
as defining a~kind of ``derivations'' on ASI bialgebras, namely, coherent derivations.
It is an attempt to extend the viewpoint that differential algebras are viewed as ``derivatization'' of algebras by equipping derivations to the context of bialgebras, that is, differential ASI bialgebras may be viewed as ``derivatization'' of ASI bialgebras.
Hence we have the following commutative diagram.
The horizontal arrows are equipping the given algebra structures with suitable derivations and the vertical arrows are equipping the given algebra structures with suitable bialgebra structures.
In particular, the below horizontal arrow is equipping ASI bialgebras with coherent derivations:
$$
 
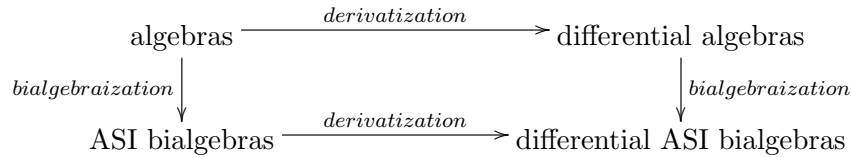
\xymatrix@C=3cm{
 \txt{algebras} \ar[d]_-{bialgebraization} \ar[r]^-{derivatization} &
 \txt{differential algebras} \ar[d]^-{bialgebraization} \\
 \txt{ASI bialgebras} \ar[r]^-{derivatization} & \txt{differential ASI bialgebras}
 }
$$

The notion of a coherent derivation is interesting on its own
right. It is known that the set of derivations on an algebra $(A,
\cdot)$ over the real number field $\mathbb{R}$ forms a Lie algebra, the set of
automorphisms on $(A, \cdot)$ forms a Lie group and the former is
the Lie algebra of the latter. Such features and properties are
still available for ASI bialgebras, in which coherent derivations
play the role as derivations do. Explicitly, the set of coherent
derivations on an ASI bialgebra forms a Lie algebra. Furthermore, we
introduce the notion of a coherent automorphism on an ASI bialgebra
inspired by \cite{BGS} and show that the set of coherent automorphisms
forms a Lie group whose Lie algebra is exactly the Lie algebra
formed by the set of coherent derivations.

On the other hand, motivated by the notion of derivations on Lie bialgebras~\cite{chari1995guide},
we introduce the notion of derivations on ASI bialgebras as an alternative approach.
We show that a~derivation on an ASI bialgebra corresponds to a special coherent derivation and hence derivations on ASI bialgebras can be regarded as special coherent derivations.

\subsection[Poisson bialgebras via commutative and cocommutative differential ASI bialgebras]{Poisson bialgebras via commutative\\ and cocommutative differential ASI bialgebras}

As aforementioned, there are typical examples of Poisson algebras
from commutative differential algebras with two commuting
derivations. It is natural to consider extending such a
relationship to the context of bialgebras. On the other hand,
there is a bialgebra theory for Poisson algebras, namely, Poisson
bialgebras, established in~\cite{ni2013poisson}. Note that there
is a noncommutative version of Poisson bialgebras given in
\cite{liu2020noncommutative}. Therefore we apply the theory of
differential ASI bialgebras to the study of Poisson bialgebras and
thus Poisson bialgebras can be constructed from commutative and
cocommutative differential ASI bialgebras.

Such an approach is consistent with the well constructed theory of
Poisson bialgebras. We establish the explicit relationships
between commutative and cocommutative differential ASI bialgebras
and Poisson bialgebras, as well as the equivalent interpretation
in terms of the corresponding double constructions (Manin triples)
and matched pairs. For the coboundary cases, the relationships
between the involved structures on commutative differential
algebras and Poisson algebras, such as admissible AYBE and Poisson
Yang--Baxter equation (PYBE), $\mathcal{O}$-operators of the two
algebras, and differential Zinbiel (commutative dendriform)
algebras and pre-Poisson algebras, are also given respectively. In particular,
an antisymmetric solution of admissible AYBE in a commutative
differential algebra is naturally a solution of PYBE in the
induced Poisson algebra under certain conditions and thus
differential Zinbiel algebras can be employed to construct Poisson
bialgebras. These notions and structures are illustrated by the
following diagram. The up two layers are notions and structures in
commutative and cocommutative differential ASI bialgebras which
have been illustrated in the diagram in Section~\ref{Sec:1.2} and
the below two layers are the corresponding notions and structures
in Poisson bialgebras given in \cite{ni2013poisson} (also see~\cite{liu2020noncommutative}):
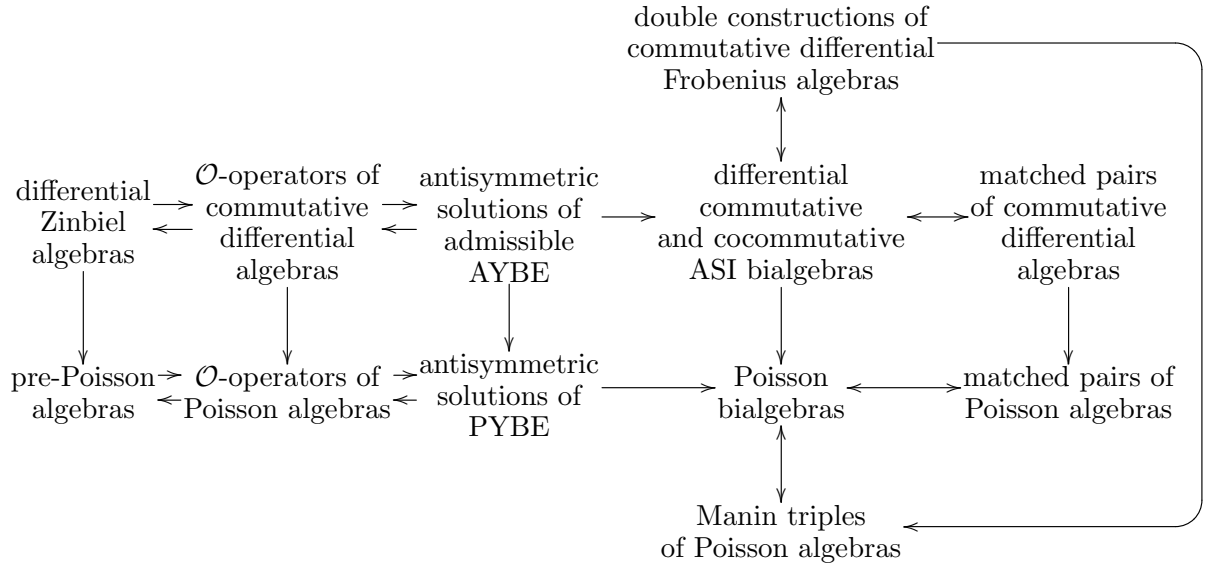
$$
 \xymatrix@C=0.3cm@M=1pt@L=0pt{
 & & & \txt{double constructions of \\ commutative differential \\ Frobenius algebras} \ar@{<->}[d] \ar@{->} `[rr] `[ddd] [ddd] & & \\
 \txt{differential \\ Zinbiel \\ algebras} \ar@{->}[r]<1ex> \ar@{->}[d] &
 \txt{$\mathcal{O}$-operators of \\ commutative \\ differential \\ algebras} \ar@{->}[l]<1ex> \ar@{->}[r]<1ex> \ar@{->}[d] &
 \txt{antisymmetric \\ solutions of \\ admissible \\ AYBE} \ar@{->}[l]<1ex> \ar@{->}[r] \ar@{->}[d] &
 \txt{differential \\ commutative \\ and cocommutative \\ ASI bialgebras} \ar@{->}[d] &
 \txt{matched pairs \\ of commutative \\ differential \\ algebras} \ar@{<->}[l] \ar@{->}[d] & \\
 \txt{pre-Poisson \\ algebras} \ar@{->}[r]<1ex> &
 \txt{$\mathcal{O}$-operators of \\ Poisson algebras} \ar@{->}[l]<1ex> \ar@{->}[r]<1ex> &
 \txt{antisymmetric \\ solutions of \\ PYBE} \ar@{->}[l]<1ex> \ar@{->}[r] &
 \txt{Poisson \\ bialgebras} &
 \txt{matched pairs of \\ Poisson algebras} \ar@{<->}[l] & \\
 & & & \txt{Manin triples \\ of Poisson algebras} \ar@{<->}[u] & &
 }
$$

\subsection{Layout of the paper}

The paper is organized as follows.

In Section~\ref{sec:rep}, we introduce some basic notions about differential algebras and their bimodules.
We introduce the notion of a bimodule of a differential algebra,
while the notion of an admissible quadruple of a differential algebra is introduced to get a bimodule on the dual space.
Moreover, the notion of a matched pair of differential algebras
is introduced to interpret the differential algebra whose underlying vector space is a linear direct sum of two differential subalgebras.

In Section~\ref{sec:dasi}, after the introduction of the notions
of a double construction of differential Frobenius algebra and a
differential ASI bialgebra, we give their equivalence in terms of
matched pairs of differential algebras. Inspired by the notion of
a differential ASI bialgebra, we then introduce the notion of a
coherent derivation on an ASI bialgebra as an equivalent structure
of a differential ASI bialgebra. Furthermore, the set of coherent
derivations on an ASI bialgebra turns out to be a Lie algebra,
which is exactly the Lie algebra of the Lie group consisting of
coherent automorphisms on this ASI bialgebra.

In Section~\ref{sec:cob}, we study the coboundary differential ASI bialgebras.
We first introduce the notion of $\Psi$-admissible AYBE in a differential algebra,
whose antisymmetric solutions are used to construct differential ASI bialgebras.
Then we introduce the notion of $\mathcal{O}$-operators of differential algebras to construct antisymmetric solutions of $\Psi$-admissible AYBE in semi-direct product differential algebras.
Finally, we introduce the notion of a differential dendriform
algebra, which gives an $\mathcal{O}$-operator of the
associated differential algebra. Thus both $\mathcal O$-operators
of differential algebras and differential dendriform algebras give rise to
differential ASI bialgebras.

In Section~\ref{sec:poisson}, proceeding from the typical
construction of Poisson algebras from commutative algebras with a
pair of commuting derivations, we construct Poisson bialgebras
introduced in~\cite{ni2013poisson} from commutative and
cocommutative differential ASI bialgebras.
The explicit relationships between them, as well as the equivalent
interpretation in terms of the corresponding double constructions
(Manin triples) and matched pairs, are established. We also
exhibit the relationships between the involved structures on
commutative differential algebras and Poisson algebras for the
coboundary cases, such as admissible AYBE and PYBE,
$\mathcal{O}$-operators of the two algebras, and differential
Zinbiel algebras and pre-Poisson algebras.
Finally, a construction of Poisson bialgebras from differential Zinbiel algebras is given.

Throughout this paper, all vector spaces and algebras are of finite dimension over a~field~$\mathbb{F}$ of characteristic~$0$, although many results still hold in the infinite-dimensional cases.
Note that all the results might be extended to the graded locally finite-dimensional case, i.e., when all homogeneous components are finite-dimensional: the notion of graded dual offers a duality which is as satisfactory as the algebraic duality in the finite-dimensional case.
The term ``algebra'' always stands for an associative algebra not necessarily having a unit unless otherwise stated.

\section{Differential algebras and their bimodules}\label{sec:rep}

We introduce the notion of a bimodule of a differential algebra.
The notion of an admissible quadruple of a differential algebra is introduced to get a reasonable bimodule on the dual space.
We also introduce the notion of a matched pair of differential algebras to interpret the differential algebra whose underlying vector space is a linear direct sum of two differential subalgebras.

\subsection{Bimodules and admissible quadruples of differential algebras}

\begin{Definition}\label{def:derivation}
 Let $(A, \cdot)$ be an algebra.
 A linear map $\partial\colon A \to A$ is called a \textit{derivation}
 if the Leibniz rule is satisfied, i.e.,
 \begin{equation}
 \partial(a \cdot b) = \partial(a) \cdot b + a \cdot \partial(b), \qquad
 \forall a, b \in A. \label{eq:derivation}
 \end{equation}
\end{Definition}

\begin{Definition}[\cite{Ritt1950Differential}]\label{def:dalg}
 A \textit{differential algebra} is a triple $(A, \cdot, \Phi)$,
 consisting of an algebra $(A, \cdot)$ and a finite set of commuting derivations $\Phi = \{\partial_k\colon A \to A\}^m_{k=1}$.
 A differential algebra $(A, \cdot, \Phi)$ is called \textit{commutative} if $(A, \cdot)$ is commutative.
\end{Definition}

\begin{Definition} Let $(A, \cdot)$ be an algebra, $V$ be a vector space and $l, r\colon A \to \End (V)$ be two linear maps.
 Then $(V, l, r)$ is called an \textit{$A$-bimodule} if the following conditions are satisfied:
 \begin{gather*}
 l(a)l(b)v = l(a \cdot b) v, \qquad\!\!\!
 r(b)r(a)v = r(a \cdot b) v, \qquad\!\!\!
 r(b)l(a)v = l(a)r(b) v, \qquad\!\!\!
 \forall a, b \in A, \, v \in V.\!
 \end{gather*}
\end{Definition}

In particular, for an algebra $(A,\cdot)$, define two linear maps $L_{A}, R_{A}\colon A \to \End(A)$ by $L_A(a)b = a \cdot b = R_A(b)a$, for all $a,b\in A$ respectively.
Then the triple $(A, L_{A}, R_{A})$ is an $A$-bimodule.

Let $(A, \cdot)$ be an algebra and $l, r\colon A \to \End (V)$ be
linear maps. Define a bilinear multiplication still denoted by
$\cdot$ on $A \oplus V$ by
\begin{equation}
 (a + u) \cdot (b + v) := a \cdot b + ( l(a) v + r(b)u ), \qquad
 \forall a, b \in A, u, v \in V. \label{eq:semiprod}
\end{equation}
Then $(V, l, r)$ is an $A$-bimodule if and only
if $(A \oplus V,\cdot)$ is an algebra, which is denoted by $(A
\ltimes_{l,r} V, \cdot)$ and called the \textit{semi-direct
product algebra} by $(A, \cdot)$ and $(V, l, r)$.

\begin{Definition}\label{def:repd}
 A \textit{bimodule of a differential algebra} $(A, \cdot, \Phi\!=\!\{\partial_k\}_{k=1}^m)$ is a quadruple $(V, l, r, \Omega)$, where $(V, l, r)$ is an $A$-bimodule and $\Omega = \{\alpha_k\colon V \to V\}_{k=1}^m$ is a set of commuting linear maps
 such that for all $k=1, \dots, m$,
 \begin{gather}
 \alpha_k( l(a)v ) = l( \partial_k(a) ) v + l(a) \alpha_k(v), \label{eq:repd1} \\
 \alpha_k( r(a)v ) = r( \partial_k(a) ) v + r(a) \alpha_k(v), \qquad
 \forall a \in A, v \in V. \label{eq:repd2}
 \end{gather}
 Two bimodules $(V_1, l_1, r_1, \{\alpha_{1k}\}_{k=1}^m)$ and $(V_2, l_2, r_2, \{\alpha_{2k}\}_{k=1}^m)$ of $(A, \cdot, \Phi)$ are called \textit{equivalent}
 if there exists a linear isomorphism $\varphi\colon V_1 \to V_2$
 such that for all $k=1, \dots, m$,
 \begin{gather*}
 \varphi( l_1(a)v ) = l_2(a) \varphi(v), \qquad
 \varphi( r_1(a)v ) = r_2(a) \varphi(v), \\
 \varphi( \alpha_{1k}(v) ) = \alpha_{2k}( \varphi(v) ), \qquad
 \forall a \in A, \ v \in V.
 \end{gather*}
\end{Definition}

\begin{Example}
 Let $(A, \cdot, \Phi)$ be a differential algebra.
 Then $(A, L_A, R_A, \Phi)$ is a bimodule of $(A, \cdot, \Phi)$.
\end{Example}

For vector spaces $V_1$ and $V_2$,
and linear maps $\phi_1\colon V_1 \to V_1$ and $\phi_2: V_2 \to V_2$,
we abbreviate $\phi_1 + \phi_2$ for the linear map
\begin{equation*}
 \phi_{V_1 \oplus V_2}\colon \ V_1 \oplus V_2 \to V_1 \oplus V_2, \qquad
 v_1 + v_2 \mapsto \phi_1(v_1) + \phi_2(v_2), \qquad
 \forall v_1 \in V_1, \ v_2 \in V_2.
\end{equation*}
Moreover, let $\Pi_1 = \{\alpha_k\colon V_1 \to V_1\}_{k=1}^m$ and $\Pi_2 = \{\beta_k\colon V_2 \to V_2\}_{k=1}^m$ be two sets of commuting linear maps,
then obviously $\{\alpha_k + \beta_k\}_{k=1}^m$ is still a set of commuting linear maps, which is denoted by $\Pi_1 + \Pi_2$.

\begin{Proposition}\label{pro:semidirect}
 Let $(A, \cdot_A, \Phi = \{\partial_k\}_{k=1}^m)$ be a differential algebra.
 Let $l, r\colon A \to \End (V)$ be linear maps
 and $\Omega = \{\alpha_k\colon V \to V\}_{k=1}^m$ be a set of commuting linear maps.
 Then $(V, l, r, \Omega)$ is a bimodule of $(A, \cdot_A, \Phi)$ if and only if $(A \oplus V, \cdot, \Phi + \Omega)$ is a differential algebra, where the multiplication $\cdot$ on $A \oplus V$ is defined by equation~\eqref{eq:semiprod}.
\end{Proposition}
\begin{proof}
 It can be proved directly according to Definitions~\ref{def:dalg} and \ref{def:repd} or as a special case of the matched pair of differential algebras in Theorem~\ref{thm:matda}, where $B = V$ is equipped with the zero multiplication.
\end{proof}

If $(V, l, r, \Omega)$ is a bimodule of a differential algebra $(A, \cdot_A, \Phi)$, then the resulting differential algebra above is denoted by $(A \ltimes_{l,r} V, \cdot, \Phi + \Omega)$ and called the \textit{semi-direct product differential algebra} by $(A, \cdot_A, \Phi)$ and $(V, l, r, \Omega)$.

Denote the standard pairing between the dual space $V^*$ and $V$ by
\begin{equation*}
 \langle\ ,\ \rangle\colon \ V^* \times V \to \mathbb{F}, \qquad
 \langle v^*, v \rangle :=v^*(v), \qquad
 \forall v \in V,\ v^* \in V^*.
\end{equation*}
Let $V$, $W$ be two vector spaces.
For a linear map $\varphi\colon V \to W$,
the transpose map $\varphi^*\colon W^* \to V^*$ is defined by
\begin{equation*}
 \langle \varphi^*(w^*), v \rangle := \langle w^*, \varphi(v) \rangle, \qquad
 \forall v \in V, \ w^* \in W^*.
\end{equation*}
Obviously we have $(\varphi\psi)^*=\psi^*\varphi^*$, where $V$, $W$, $U$ are vector spaces, $\varphi\colon V \to W$ and $\psi\colon U \to V$ are linear maps.
Furthermore, when $\Pi = \{\alpha_k\}$ is a finite set of linear maps,
set $\Pi^* := \{\alpha_k^*\colon \alpha_k \in \Pi\}$ to be the linear dual of $\Pi$.

Let $A$ be an algebra and $V$ be a vector space.
For a linear map $\mu\colon A \to \End (V)$, the linear map $\mu^*\colon A \to \End(V^*)$ is defined by
\begin{equation*}
 \langle \mu^*(a)v^*, v \rangle := \langle v^*, \mu(a)v \rangle, \qquad
 \forall a \in A,\ v \in V, \ v^* \in V^*,
\end{equation*}
that is, $\mu^*(a) = \mu(a)^*$ for all $a \in A$.

Note that for an algebra $(A,\cdot)$ and an $A$-bimodule $(V, l, r)$, $(V^*, r^*, l^*)$ is again an $A$-bimodule.

\begin{Lemma}\label{lem:dualrep}
 Let $(A, \cdot, \Phi = \{\partial_k\}_{k=1}^m)$ be a differential algebra.
 Let $(V, l, r)$ be an $A$-bimodule
 and $\Pi = \{\beta_k\colon V \to V\}_{k=1}^m$ be a set of commuting linear maps.
 Then the quadruple $(V^*, r^*, l^*, \Pi^*)$ is a bimodule of the differential algebra $(A, \cdot, \Phi)$ if and only if for all $k=1, \dots,
 m$, the following equations hold:
 \begin{gather}
 r(a) \beta_k(v) = r(\partial_k(a)) v + \beta_k(r(a)v), \label{eq:admrep1} \\
 l(a) \beta_k(v) = l(\partial_k(a)) v + \beta_k(l(a)v), \qquad
 \forall a \in A, \ v \in V. \label{eq:admrep2}
 \end{gather}
\end{Lemma}
\begin{proof}
	Let $a \in A$. Then we have
	\begin{equation*}
		\beta_k^* r(a)^* - r(a)^* \beta_k^* = (r(a) \beta_k - \beta_k r(a))^*.
	\end{equation*}
	Hence $\beta_k^* r(a)^* - r(a)^* \beta_k^* = r(\partial_k(a))^*$ if and only if $r(a) \beta_k - \beta_k r(a) = r(\partial_k(a))$.
	Therefore equation~\eqref{eq:repd1} holds in which $\alpha_k$ is replaced by $\beta_k^*$ and $l$ is replaced by $r^*$, if and only if equation~\eqref{eq:admrep1} holds.
	Similarly, equation~\eqref{eq:repd2} holds in which $\alpha_k$ is replaced by $\beta_k^*$ and $r$ is replaced by $l^*$, if and only if equation~\eqref{eq:admrep2} holds.
	Thus the conclusion follows.
\end{proof}

We introduce a notion to conceptualize this basic property.
\begin{Definition}\label{def:admissible}
 Let $(A, \cdot, \Phi = \{\partial_k\}_{k=1}^m)$ be a differential algebra.
 Let $(V, l, r)$ be an $A$-bimodule
 and $\Pi = \{\beta_k\colon V \to V\}_{k=1}^m$ be a set of commuting linear maps.
 Then we say \textit{$(V,l,r,\Pi)$ is an admissible quadruple of the differential algebra $(A, \cdot, \Phi)$} or
$\Pi$ is \textit{admissible to $(A, \cdot, \Phi)$ on $(V, l, r)$}
 if $(V^*, r^*, l^*, \Pi^*)$ is a bimodule of $(A, \cdot, \Phi)$ or equivalently equations~\eqref{eq:admrep1}--\eqref{eq:admrep2} hold.
 When $(V, l, r)$ is taken to be $(A, L_A, R_A)$, we say $\Pi$ is \textit{admissible to $(A, \cdot, \Phi)$} or $(A, \cdot, \Phi)$ is \textit{$\Pi$-admissible}.
\end{Definition}

Immediately we have the following conclusion by Lemma~\ref{lem:dualrep}.

\begin{Corollary}\label{cor:qadm}
 Let $(A, \cdot, \Phi = \{\partial_k\}_{k=1}^m)$ be a differential algebra
 and $\Psi = \{\eth_k\colon A \to A \}_{k=1}^m$ be a set of commuting linear maps.
 Then
 $\Psi$ is admissible to $(A, \cdot, \Phi)$ if and only if for all $k=1, \dots, m$,
 the following equations hold:
 \begin{gather}
 \eth_k(a) \cdot b = a \cdot \partial_k(b) + \eth_k(a \cdot b), \label{eq:qadm1} \\
 a \cdot \eth_k(b) = \partial_k(a) \cdot b + \eth_k(a \cdot b), \qquad
 \forall a,b \in A. \label{eq:qadm2}
 \end{gather}
\end{Corollary}

\begin{Example}\label{ex:admissible}
 Let $(V, l, r, \{\alpha_k\}_{k=1}^m)$ be a bimodule of a differential algebra $(A, \cdot, \Phi = \{\partial_k\}_{k=1}^m)$.
 Let $(\theta_1, \dots, \theta_m)$ $\in \mathbb{F}^m$ be given.
 Then by a straightforward checking, we have the following results.
 \begin{enumerate}\itemsep=0pt
 \item Suppose that $\theta_k = \pm 1$ ($1 \leq k \leq m$).
 Then $\{\theta_k \alpha_k\}_{k=1}^m$ is admissible to $(A, \cdot, \Phi)$ on $(V, l, r)$ if and only if
 \begin{gather*}
 (1 + \theta_k)r(\partial_k(a)) v = 0, \qquad
 (1 + \theta_k)l(\partial_k(a)) v = 0, \qquad
 \forall a \in A, \ v \in V, \ 1 \leq k \leq m.
 \end{gather*}
 \item $\{-\alpha_k + \theta_k \id_V\}_{k=1}^m$ is admissible to $(A, \cdot, \Phi)$ on $(V, l, r)$.
 In particular, $\{-\partial_k + \theta_k \id_A \}_{k=1}^m$ is admissible to $(A, \cdot, \Phi)$.
 \item Suppose that $\alpha_k$ is invertible and $\theta_k \neq 0$ for all $k=1, \dots, m$.
 Then $\big\{\theta_k \alpha_k^{-1}\big\}_{k=1}^m$ is admissible to $(A, \cdot, \Phi)$ on $(V, l, r)$ if and only if
 \begin{equation*}
 \theta_k r(\partial_k(a)) \alpha_k^{-1}(v) = \alpha_k(r(\partial_k(a))v), \qquad
 \theta_k l(\partial_k(a)) \alpha_k^{-1}(v) = \alpha_k(l(\partial_k(a))v),
 \end{equation*}
 	for all $a \in A$, $v \in V$, $1 \leq k \leq m$.
 \end{enumerate}
 Hence for all $(\theta_1, \dots, \theta_m) \in \mathbb{F}^m$,
 $(V^*, r^*, l^*, \{-\alpha_k^* + \theta_k \id_{V^*} \}_{k=1}^m)$ is a bimodule of $(A, \cdot, \Phi)$.
 Thus $(A \ltimes_{r^*, l^*} V^*, \cdot, \Phi + \{-\alpha_k^* + \theta_k \id_{V^*}\}_{k=1}^m \})$ is a differential algebra.
\end{Example}

\subsection{Matched pairs of differential algebras}\label{subsec:matched_pair}
We recall the concept of a matched pair of algebras.

\begin{Definition} A \textit{matched pair of algebras} consists of algebras $(A, \cdot_A)$ and $(B, \cdot_B)$,
 together with linear maps $l_A, r_A\colon A \to \End(B)$ and $l_B, r_B: B \to \End(A)$
 such that $(A \oplus B, \star)$ is an algebra, where $\star$ is defined by
 \begin{gather}
 (a + b) \star (a^\prime + b^\prime) := (a \cdot_A a^\prime + r_B(b^\prime)a + l_B(b)a^\prime) + (b \cdot_B b^\prime + l_A(a)b^\prime + r_A(a^\prime)b), \label{eq:multsum}
 \end{gather}
for all $a, a^\prime \in A$, $b, b^\prime \in B$. The matched pair of algebras is denoted by $((A, \cdot_A), (B, \cdot_B), l_A, r_A,\allowbreak l_B, r_B)$ and the resulting algebra $(A \oplus B, \star)$ is denoted by $\big(A \bowtie^{l_B,r_B}_{l_A,r_A} B, \star\big)$ or simply $(A \bowtie B, \star)$.
\end{Definition}

Note that such a notion of a matched pair of algebras is equivalent to the one given in~\cite[Definition~2.1.5]{BaiDouble}.
Moreover, for a matched pair of algebras $((A, \cdot_A), (B, \cdot_B), l_A, r_A, l_B, r_B)$, $(A, l_B, r_B)$ is a $B$-bimodule and $(B, l_A, r_A)$ is an $A$-bimodule.

\begin{Definition}
 Let $(A, \cdot_A, \Phi_A)$ and $(B, \cdot_B, \Phi_B)$ be two differential algebras.
 Suppose that $l_A, r_A\colon A \to \End(B)$ and $l_B, r_B\colon B \to \End(A)$ are linear maps.
 If the following conditions are satisfied:
 \begin{enumerate}\itemsep=0pt
 \item[(1)] $(A, l_B, r_B, \Phi_A)$ is a bimodule of $(B, \cdot_B, \Phi_B)$;
 \item[(2)] $(B, l_A, r_A, \Phi_B)$ is a bimodule of $(A, \cdot_A, \Phi_A)$;
 \item[(3)] $((A, \cdot_A), (B, \cdot_B), l_A, r_A, l_B, r_B)$ is a matched pair of algebras,
 \end{enumerate}
 then $((A, \cdot_A, \Phi_A), (B, \cdot_B, \Phi_B), l_A, r_A, l_B, r_B)$ is called a \textit{matched pair of differential algebras}.
\end{Definition}

\begin{Theorem}\label{thm:matda}
 Let $(A, \cdot_A, \Phi_A)$ and $(B, \cdot_B, \Phi_B)$ be two differential algebras.
 Suppose that $((A, \cdot_A), (B, \cdot_B)$, $l_A$, $r_A$, $l_B, r_B)$ is a matched pair of algebras and
 $(A \bowtie B, \star)$ is the algebra defined by equation~\eqref{eq:multsum}.
 Then $(A \bowtie B, \star, \Phi_A + \Phi_B)$ is a differential algebra if and only if
 $((A, \cdot_A, \Phi_A), (B, \cdot_B, \Phi_B), l_A, r_A$, $l_B$, $r_B)$ is a matched pair of differential algebras.
 Further, any differential algebra whose underlying vector space is the linear direct sum of two differential subalgebras is obtained from a matched pair of these two differential subalgebras.
\end{Theorem}
\begin{proof}
 We only need to prove the first part since the second part follows straightforwardly. Set $\Phi_A = \{\partial_{A,k}\}_{k=1}^m$ and $\Phi_B = \{\partial_{B,k}\}_{k=1}^m$.

 ($\Longleftarrow$)
 Suppose that $((A, \cdot_A, \Phi_A), (B, \cdot_B, \Phi_B), l_A, r_A, l_B, r_B)$ is a matched pair of differential algebras.
 Then for all $a, a^\prime \in A$, $b, b^\prime \in B$ and $k=1,\dots, m$, we have
 \begin{gather*}
(\partial_{A,k} + \partial_{B,k})(a + b) \star (a^\prime + b^\prime) + (a + b) \star (\partial_{A,k} + \partial_{B,k})(a^\prime + b^\prime) \\
 \qquad{} = (\partial_{A,k}(a) \cdot_A a^\prime + r_B(b^\prime)\partial_{A,k}(a) + l_B(\partial_{B,k}(b))a^\prime) \\
 \qquad \quad{}+ (\partial_{B,k}(b) \cdot_B b^\prime + l_A(\partial_{A,k}(a))b^\prime + r_A(a^\prime)\partial_{B,k}(b)) \\
 \qquad \quad{} + (a \cdot_A \partial_{A,k}(a^\prime) + r_B(\partial_{B,k}(b^\prime))a + l_B(b)\partial_{A,k}(a^\prime)) \\
 \qquad \quad{} + (b \cdot_B \partial_{B,k}(b^\prime) + l_A(a)\partial_{B,k}(b^\prime) + r_A(\partial_{A,k}(a^\prime))b) \\
\quad \overset{\eqref{eq:repd1}, \, \eqref{eq:repd2}}{=} \partial_{A,k}(a \cdot_A a^\prime) + \partial_{A,k}(r_B(b^\prime)a) + \partial_{A,k}(l_B(b)a^\prime) \\
 \qquad\quad{} + \partial_{B,k}(b \cdot_B b^\prime) + \partial_{B,k}(l_A(a)b^\prime) + \partial_{B,k}(r_A(a^\prime)b) \\
 \qquad {} = (\partial_{A,k}+\partial_{B,k})((a + b) \star (a^\prime + b^\prime)).
\end{gather*}
 Hence $\Phi_A + \Phi_B$ is a set of commuting derivations on the algebra $(A \bowtie B, \star)$.

 ($\Longrightarrow$)
 Suppose that $\Phi_A + \Phi_B$ is a set of commuting derivations on $(A \bowtie B, \star)$.
 Then
 \begin{gather*}
\begin{split}
& (\partial_{A, k}+\partial_{B,k})((a + b) \star (a^\prime + b^\prime))\\
& \qquad{} = (\partial_{A, k} + \partial_{B,k})(a + b) \star (a^\prime + b^\prime) + (a+b)\star (\partial_{A, k}+\partial_{B,k})(a^\prime + b^\prime),
 \end{split}
 \end{gather*}
 for all $a, a^\prime \in A$, $b, b^\prime \in B$ and $k=1, \dots, m$.
 Taking $a^\prime = b = 0$ and $a = b^\prime = 0$ in the above equation respectively, we obtain
 \begin{gather*}
\partial_{A,k}(r_B(b^\prime)a) = r_B(b^\prime)\partial_{A,k}(a) + r_B(\partial_{B,k}(b^\prime))a, \\
 \partial_{B,k}(l_A(a)b^\prime) = l_A(\partial_{A,k}(a))b^\prime + l_A(a)\partial_{B,k}(b^\prime), \\
 \partial_{A,k}(l_B(b)a^\prime) = l_B(\partial_{B,k}(b))a^\prime + l_B(b)\partial_{A,k}(a^\prime), \\
 \partial_{B,k}(r_A(a^\prime)b) = r_A(a^\prime)\partial_{B,k}(b) + r_A(\partial_{A,k}(a^\prime))b.
 \end{gather*}
 Thus $(A, l_B, r_B, \Phi_A)$ is a bimodule of $(B, \cdot_B, \Phi_B)$ and $(B, l_A, r_A, \Phi_B)$ is a bimodule of $(A, \cdot_A, \Phi_A)$.
 Hence $((A, \cdot_A, \Phi_A), (B, \cdot_B, \Phi_B), l_A, r_A, l_B, r_B)$ is a matched pair of differential algebras.
\end{proof}

\section[Differential ASI bialgebras and coherent derivations on ASI bialgebras]{Differential ASI bialgebras and coherent derivations\\ on ASI bialgebras}\label{sec:dasi}

We introduce the notions of a double construction of differential
Frobenius algebra and a differential antisymmetric infinitesimal
(ASI) bialgebra, and give their equivalence in terms of matched
pairs of differential algebras. The notion of a coherent
derivation on an ASI bialgebra is introduced as an equivalent
structure of a differential ASI bialgebra. The set of coherent
derivations on an ASI bialgebra forms a Lie algebra, which is the
Lie algebra of the Lie group consisting of coherent automorphisms on this ASI bialgebra.

\subsection{Double constructions of differential Frobenius algebras}\label{subsec:double}
We recall the concept of a double construction of Frobenius algebra \cite{BaiDouble}.

\begin{Definition}\label{def:invariant_bilinear_form}
 A bilinear form $\mathfrak{B}(\ ,\ )$ on an algebra $(A, \cdot)$ is called \textit{invariant} if
 \begin{equation*}
 \mathfrak{B}(a \cdot b, c) = \mathfrak{B}(a, b \cdot c), \qquad
 \forall a, b, c \in A.
 \end{equation*}
 A \textit{Frobenius algebra} $(A, \cdot, \mathfrak{B})$ is an algebra $(A, \cdot)$ with a nondegenerate invariant bilinear form $\mathfrak{B}(\ , \ )$.
 A Frobenius algebra $(A, \cdot, \mathfrak{B})$ is called \textit{symmetric} if $\mathfrak{B}(\ , \ )$ is symmetric.
\end{Definition}

Let $(A, \cdot)$ be an algebra.
Suppose that there is an algebra structure $\circ$ on its dual space $A^*$ and an algebra structure $\star$ on the direct sum $A \oplus A^*$ of the underlying vector spaces $A$ and $A^*$, which contains both $(A, \cdot)$ and $(A^*, \circ)$ as subalgebras.
Define a bilinear form on $A \oplus A^*$ by
\begin{equation}
 \mathfrak{B}_d(a + a^*, b + b^*) := \langle a, b^* \rangle + \langle a^*, b \rangle , \qquad
 \forall a, b \in A, \ a^*, b^* \in A^*. \label{eq:bilform}
\end{equation}
If $\mathfrak{B}_d$ is invariant, that is,
$(A \oplus A^*, \star, \mathfrak{B}_d)$ is a symmetric Frobenius algebra,
then the Frobenius algebra $(A \oplus A^*, \star, \mathfrak{B}_d)$ is called a
\textit{double construction of Frobenius algebra associated to $(A, \cdot)$ and $(A^*, \circ)$},
and denoted by $(A \bowtie A^*, \star , \mathfrak{B}_d)$.
The notation $A \bowtie A^*$ is justified since the algebra structure on $A \oplus A^*$ comes from a matched pair of algebras, that is, the product $\star$ on $A \oplus A^*$ is given by equation~\eqref{eq:multsum} for certain matched pair of algebras which in fact is shown to be $((A, \cdot), (A^*, \circ), R^*_A, L^*_A, R^*_{A^*}, L^*_{A^*})$.

We extend these notions to the context of differential algebras.
\begin{Definition}\label{def:differential_Frobenius_algebra}
 A \textit{differential Frobenius algebra} is a quadruple $(A, \cdot, \Phi, \mathfrak{B})$, where $(A, \cdot$, $\Phi = \{\partial_k\}_{k=1}^m)$ is a differential algebra
 and $(A, \cdot, \mathfrak{B})$ is a Frobenius algebra. It is
 called \textit{symmetric} if $\mathfrak B$ is symmetric.
 For all $k=1,\dots, m$, let $\hat{\partial_k}$ be the adjoint linear operator of $\partial_k$ under the nondegenerate bilinear form $\mathfrak{B}$:
 \begin{equation*}
 \mathfrak{B}(\partial_k(a), b) = \mathfrak{B}\big(a, \hat{\partial_k}(b)\big), \qquad
 \forall a, b \in A.
 \end{equation*}
 We call $\hat{\Phi} := \big\{\hat{\partial_k}\big\}_{k=1}^m$ the \textit{adjoint of $\Phi = \{\partial_k\}_{k=1}^m$ with respect to $\mathfrak B$}.
\end{Definition}

\begin{Proposition}\label{pro:FroRep}
 Let $(A, \cdot, \Phi, \mathfrak{B})$ be a symmetric differential Frobenius algebra.
 Let $\hat{\Phi}$ be the adjoint of $\Phi$ with respect to $\mathfrak{B}$.
 Then $\hat{\Phi}$ is admissible to $(A, \cdot, \Phi)$, or equivalently,
 $\big(A^*, R_A^*, L_A^*, \hat{\Phi}^*\big)$~is a~bimodule of the differential algebra $(A, \cdot, \Phi)$.
 Moreover as bimodules of $(A, \cdot, \Phi)$, $(A, L_A, R_A, \Phi)$ and $\big(A^*, R_A^*, L_A^*, \hat{\Phi}^*\big)$ are equivalent.
 Conversely let $(A, \cdot, \Phi)$ be a differential algebra and $\Psi$ be admissible to $(A, \cdot, \Phi)$.
 If the resulting bimodule $(A^*, R_A^*, L_A^*, \Psi^*)$ of $(A, \cdot, \Phi)$ is equivalent to $(A, L_A, R_A, \Phi)$,
 then there exists a nondegenerate invariant bilinear form $\mathfrak{B}$ on $A$ such that $\hat{\Phi} = \Psi$.
\end{Proposition}
\begin{proof}
 Suppose that $(A, \cdot, \Phi= \{\partial_k\}_{k=1}^m, \mathfrak{B})$ is a symmetric differential Frobenius algebra.
 Then for all $a, b, c \in A$ and $k=1, \dots, m$, we have
 \begin{align*}
 0 &\overset{\eqref{eq:derivation}}{=} \mathfrak{B}(\partial_k(a \cdot b), c) - \mathfrak{B}(\partial_k(a) \cdot b, c)-\mathfrak{B}(a \cdot \partial_k(b), c) \\
 &\overset{\hphantom{(2.1)}}{=} \mathfrak{B}(a \cdot b, \hat{\partial_k}(c)) - \mathfrak{B}(\partial_k(a), b \cdot c) - \mathfrak{B}(a, \partial_k(b) \cdot c) \\
 &\overset{\hphantom{(2.1)}}{=} \mathfrak{B}(a, b \cdot \hat{\partial_k}(c)) - \mathfrak{B}(a, \hat{\partial_k}(b \cdot c)) - \mathfrak{B}(a, \partial_k(b) \cdot c).
 \end{align*}
 Thus $b \cdot \hat{\partial_k}(c) = \hat{\partial_k}(b \cdot c) + \partial_k(b) \cdot c$,
 that is, equation~\eqref{eq:qadm2} holds.
 Similarly, equation~\eqref{eq:qadm1} holds.
 Hence $\hat{\Phi}$ is admissible to $(A, \cdot, \Phi)$, that is,
 $\big(A^*, R_A^*, L_A^*, \hat{\Phi}^*\big)$ is a bimodule of $(A, \cdot, \Phi)$.
 Define a~linear map $\varphi\colon A \to A^*$ by
 \begin{equation*}
 \varphi(a)(b) := \langle \varphi(a), b \rangle = \mathfrak{B}(a, b), \qquad
 \forall a, b \in A.
 \end{equation*}
 Obviously, $\varphi$ is a linear isomorphism.
 Moreover for all $a, b, c \in A$ and $k=1, \dots, m$, we have
 \begin{gather*}
 \langle \varphi(L_A(a)b), c \rangle = \langle \varphi(a \cdot b), c \rangle = \mathfrak{B}(a \cdot b, c) = \mathfrak{B}(b, c \cdot a) = \langle \varphi(b),c \cdot a \rangle = \langle R_A^*(a)\varphi(b), c\rangle, \\
 \langle\varphi(R_A(b)a), c \rangle = \langle \varphi(a \cdot b), c \rangle = \mathfrak{B}(a \cdot b,c)=\mathfrak{B}(a,b \cdot c) = \langle \varphi(a), b \cdot c\rangle = \langle L_A^*(b)\varphi(a), c\rangle, \\
 \langle \varphi(\partial_k(a)), b \rangle = \mathfrak{B}(\partial_k(a), b) = \mathfrak{B}(a, \hat{\partial_k}(b)) = \langle \varphi(a), \hat{\partial_k}(b)\rangle = \langle \hat{\partial_k}^*\varphi(a), b \rangle.
 \end{gather*}
 Hence $(A, L_A, R_A, \Phi)$ and $\big(A^*, R_A^*, L_A^*, \hat{\Phi}^*\big)$ are equivalent as bimodules of $(A, \cdot, \Phi)$.

 Conversely, suppose that $\varphi\colon A \to A^*$ is the linear isomorphism giving the equivalence between $(A, L_A$, $R_A$, $\Phi)$ and $(A^*, R_A^*, L_A^*, \Psi^*)$.
 Define a bilinear form $\mathfrak{B}$ on $A$ by
 \begin{equation*}
 \mathfrak{B}(a, b) := \langle \varphi(a),b \rangle , \qquad \forall a, b \in A.
 \end{equation*}
 Then by a similar argument as above,
 we show that $\mathfrak{B}$ is a nondegenerate invariant bilinear form on $(A, \cdot)$
 such that $\hat{\Phi} = \Psi$.
\end{proof}

\begin{Definition}
 Let $(A, \cdot, \Phi)$ be a differential algebra.
 Suppose that there is a differential algebra structure $(A^*, \circ, \Psi^*)$ on the dual space $A^{*}$.
 A \textit{double construction of differential Frobenius algebra associated to $(A, \cdot, \Phi)$ and $(A^*, \circ, \Psi^*)$} is
 a double construction of Frobenius algebra $(A \bowtie A^*,\allowbreak \star, \mathfrak{B}_d)$
 associated to $(A, \cdot)$ and $(A^*, \circ)$
 such that $(A \bowtie A^*, \star, \Phi + \Psi^*)$ is a differential algebra,
 which is denoted by $(A \bowtie A^*, \star, \Phi + \Psi^*, \mathfrak{B}_d)$.
\end{Definition}

\begin{Lemma}\label{lem:douadm}
 Let $(A \bowtie A^*, \star, \Phi + \Psi^*, \mathfrak{B}_d)$ be a double construction of differential Frobenius algebra associated to $(A, \cdot, \Phi)$ and $(A^*, \circ, \Psi^*)$. Then the following conclusions hold:
 \begin{enumerate}\itemsep=0pt
 \item[$(1)$] The adjoint $\widehat{\Phi + \Psi^*}$ of $\Phi + \Psi^*$ with respect to $\mathfrak{B}_d$ is $\Psi + \Phi^*$.
 Further $\Psi + \Phi^*$ is admissible to $(A \bowtie A^*,\star, \Phi + \Psi^*)$.
 \item[$(2)$] $\Psi$ is admissible to $(A, \cdot, \Phi)$.
 \item[$(3)$] $\Phi^*$ is admissible to $(A^*, \circ, \Psi^*)$.
 \end{enumerate}
\end{Lemma}
\begin{proof}
	Set $\Phi=\{\partial_k\}_{k=1}^m$ and $\Psi^* = \{\eth_k^*\}_{k=1}^m$.
 Let $a, b \in A$, $a^*, b^* \in A^*$.

 (1)~By equation~\eqref{eq:bilform}, we have
 \begin{align*}
 	\mathfrak{B}_d((\partial_k + \eth_k^*)(a + a^*), b + b^*) &= \mathfrak{B}_d(\partial_k(a) + \eth_k^*(a^*), b + b^*) = \langle \partial_k(a), b^* \rangle + \langle \eth_k^*(a^*), b \rangle \\
 	& = \langle a, \partial_k^*(b^*) \rangle + \langle a^*, \eth_k(b) \rangle
 	= \mathfrak{B}_d(a + a^*, (\eth_k + \partial_k^*)(b + b^*)).
 \end{align*}
 Hence the adjoint $\widehat{\Phi + \Psi^*}$ of $\Phi + \Psi^*$ with respect to $\mathfrak{B}_d$ is $\Psi + \Phi^*$.
 By Proposition~\ref{pro:FroRep},
 $\Psi + \Phi^*$ is admissible to $(A \bowtie A^*, \star, \Phi + \Psi^*)$.

 (2)~By (1), $\Psi + \Phi^*$ is admissible to $(A \bowtie A^*, \star, \Phi + \Psi^*)$.
 Then by equations~\eqref{eq:qadm1} and \eqref{eq:qadm2}, we have
 \begin{gather*}
 (\eth_k + \partial_k^*)(a + a^*) \star (b + b^*) = (a + a^*) \star (\partial_k + \eth_k^*)(b + b^*) + (\eth_k + \partial_k^*)((a + a^*)\star (b + b^*)), \\
 (a + a^*) \star (\eth_k + \partial_k^*)(b + b^*) =(\partial_k + \eth_k^*)(b + b^*) \star (b + b^*) + (\eth_k + \partial_k^*)((a + a^*) \star (b + b^*)).
 \end{gather*}
 Taking $a^* = b^* = 0$ in the above equations,
 we show that $\Psi$ is admissible to $(A, \cdot, \Phi)$.

 (3) Taking $a = b = 0$ in the above equations,
 we show that $\Phi^*$ is admissible to $(A^*, \circ, \Psi^*)$.
\end{proof}

\begin{Theorem}\label{thm:doumat}
 Let $(A, \cdot, \Phi)$ be a differential algebra.
 Suppose that there is a differential algebra structure $(A^*, \circ, \Psi^*)$ on $A^*$.
 Then there is a double construction of differential Frobenius algebra
 $(A \bowtie A^*, \star, \Phi + \Psi^*, \mathfrak{B}_d)$ associated to $(A, \cdot, \Phi)$ and $(A^*, \circ, \Psi^*)$ if and only if
 $\big((A, \cdot, \Phi), (A^*, \circ, \Psi^*), R^*_A, L^*_A, R^*_{A^*}, L^*_{A^*}\big)$ is a matched pair of differential algebras.
\end{Theorem}
\begin{proof}
 ($\Longrightarrow$)
 By assumption, $(A \oplus A^*, \star, \mathfrak{B}_d)$ is a double construction of Frobenius algebra associated to $(A,\cdot)$ and $(A^*, \circ)$.
 By \cite[Theorem~2.2.1]{BaiDouble}, $\big((A, \cdot), (A^*, \circ), R^*_A, L^*_A, R^*_{A^*}, L^*_{A^*}\big)$ is a~matched pair of algebras.
 On the other hand, by Lemma~\ref{lem:douadm}, $(A^*, R^*_A, L^*_A, \Psi^*)$ is a~bimodule of $(A, \cdot, \Phi)$ and
 $(A, R^*_{A^*}, L^*_{A^*}, \Phi)$ is a bimodule of $(A^*, \circ, \Psi^*)$ respectively.
 Hence $\big((A, \cdot, \Phi),\allowbreak (A^*, \circ, \Psi^*),R^*_A, L^*_A, R^*_{A^*}, L^*_{A^*}\big)$ is a matched pair of differential algebras.

 ($\Longleftarrow$)
 If $\big((A, \cdot, \Phi), (A^*, \circ, \Psi^*), R^*_A, L^*_A, R^*_{A^*}, L^*_{A^*}\big)$ is a matched pair of differential algebras,
 then $\big((A, \cdot), (A^*, \circ), R^*_A, L^*_A, R^*_{A^*}, L^*_{A^*}\big)$ is a matched pair of algebras.
 Hence by \cite[Theorem~2.2.1]{BaiDouble} again,
 $(A \bowtie A^*, \star, \mathfrak{B}_d)$ is a Frobenius algebra, where $\star$ is defined by equation~\eqref{eq:multsum}.
 Moreover by Theorem~\ref{thm:matda},
 $(A \bowtie A^*, \star, \Phi + \Psi^*)$ is a differential algebra.
 Hence $(A \bowtie A^*, \star, \Phi + \Psi^*, \mathfrak{B}_d)$ is a~double construction of differential Frobenius algebra associated to $(A, \cdot, \Phi)$ and $(A^*, \circ, \Psi^*)$.
\end{proof}

\subsection{Differential ASI bialgebras}\label{subsec:dasi}
With the previous preparations, we are ready to introduce the
notion of a differential ASI bialgebra as an enrichment of the
notion of an ASI bialgebra \cite{AguiarOn, BaiDouble, joni1979coalgebras, zhelyabian1997jordan}.

Recall that a \textit{coalgebra} $(A, \Delta)$ is a vector space $A$ with a linear map $\Delta\colon A \to A \otimes A$ satisfying the coassociative law:
\begin{equation*}
 (\Delta \otimes \id)\Delta = (\id \otimes \Delta) \Delta.
\end{equation*}
A coalgebra $(A, \Delta)$ is called \textit{cocommutative}
if $\Delta = \sigma \Delta$,
where $\sigma\colon A \otimes A \to A \otimes A$ is the flip
operator defined by $\sigma(a \otimes b) := b \otimes a$ for all $a, b \in A$.

\begin{Definition}[\cite{BaiDouble}]\label{def:asi_bialgebra}
An \textit{antisymmetric infinitesimal bialgebra} or simply an \textit{ASI bialgebra} is a~triple $(A, \cdot, \Delta)$ consisting of a vector space $A$ and linear maps $\cdot\colon A \otimes A \to A$ and $\Delta\colon A \to A \otimes A$ such that
 \begin{enumerate}\itemsep=0pt
 \item[(1)] $(A, \cdot)$ is an algebra;
 \item[(2)] $(A, \Delta)$ is a coalgebra;
 \item[(3)] the following equations hold:
 \begin{gather}
 \Delta(a \cdot b) = (R_A(b) \otimes \id )\Delta(a) + (\id \otimes L_A(a)) \Delta(b), \label{eq:asiifs} \\
 (L_A(a) \otimes \id - \id \otimes R_A(a))\Delta(b) = \sigma \big( (\id \otimes R_A(b) - L_A(b) \otimes \id )\Delta(a) \big), \label{eq:asisym}
 \end{gather}
 	for all $a, b \in A$.
 \end{enumerate}
\end{Definition}

\begin{Definition}[\cite{doi1981homological}]\label{def:coderivation}
Let $(A, \Delta)$ be a coalgebra.
 A linear map $\eth\colon A \to A$ is called a~\textit{coderivation} on $(A, \Delta)$ if the following equation holds:
 \begin{equation}
 \Delta \eth = (\eth \otimes \id + \id \otimes \eth) \Delta. \label{eq:coderivation}
 \end{equation}
\end{Definition}

\begin{Definition}
 A \textit{differential coalgebra} is a triple $(A, \Delta, \Psi)$,
 consisting of a coalgebra $(A, \Delta)$ and a finite set of commuting coderivations $\Psi = \{\eth_k\colon A \to A\}_{k=1}^m$.
 A differential coalgebra $(A, \Delta, \Psi)$ is called \textit{cocommutative} if $(A, \Delta)$ is cocommutative.
\end{Definition}

\begin{Remark}
 The notion of a differential coalgebra is the dualization of the notion of a~differential algebra, that is, $(A, \Delta, \Psi)$ is a differential coalgebra if and only if $(A^*, \Delta^*, \Psi^*)$ is a~differential algebra.
 Moreover, $(A,\Delta)$ is cocommutative if and only if
 $(A^*,\Delta^*)$ is commutative.
\end{Remark}

\begin{Lemma}\label{lem:coadm}
 Let $(A, \Delta, \Psi = \{\eth_k\}_{k=1}^m)$ be a differential coalgebra and
 $\Phi = \{\partial_k\colon A \to A\}_{k=1}^m$ be a set of commuting linear maps.
 Then $\Phi^*$ is admissible to the differential algebra $(A^*, \Delta^*, \Psi^*)$ if and only
 if the following equations hold:
 \begin{gather}
 (\partial_k \otimes \id) \Delta = (\id \otimes \eth_k) \Delta + \Delta \partial_k, \label{eq:psadm1} \\
 (\id \otimes \partial_k) \Delta = (\eth_k \otimes \id) \Delta + \Delta \partial_k, \qquad \forall k=1, \dots, m. \label{eq:psadm2}
 \end{gather}
\end{Lemma}
\begin{proof}
By Corollary~\ref{cor:qadm}, $\Phi^*$ is admissible to $(A^*, \Delta^*, \Psi^*)$ if and only if the following equations hold:
\begin{gather*}
	\partial_k^*(a^*) \circ b^* = a^* \circ \eth_k^*(b^*) + \partial_k^*(a^* \circ b^*), \\
 a^* \circ \partial_k^*(b^*) = \eth_k^*(a^*) \circ b^* + \partial_k^*(a^* \circ b^*), \qquad \forall k=1, \dots, m,
\end{gather*}
where $a^* \circ b^* = \Delta^*(a^*\otimes b^*)$, for all $a^*, b^* \in A^*$.
Rewriting the above equations in terms of the comultiplication, we
get equations~\eqref{eq:psadm1} and \eqref{eq:psadm2} respectively.
Hence the conclusion holds.
\end{proof}

\begin{Definition}\label{def:dasi}
 A \textit{differential antisymmetric infinitesimal bialgebra} or simply a \textit{differential ASI bialgebra} is a quintuple
 $(A, \cdot, \Delta, \Phi, \Psi)$ satisfying
 \begin{enumerate}\itemsep=0pt
 \item[(1)] $(A, \cdot, \Delta)$ is an ASI bialgebra;
 \item[(2)] $(A, \cdot, \Phi = \{\partial_k\}_{k=1}^m)$ is a differential algebra;
 \item[(3)] $(A, \Delta, \Psi = \{\eth_k\}_{k=1}^m)$ is a differential coalgebra;
 \item[(4)] $\Psi$ and $\Phi^*$ are admissible to the differential algebras $(A, \cdot, \Phi)$ and $(A^*, \Delta^*, \Psi^*)$ respectively, that is,
 equations~\eqref{eq:qadm1}, \eqref{eq:qadm2}, \eqref{eq:psadm1} and \eqref{eq:psadm2} hold.
 \end{enumerate}
 A differential ASI bialgebra $(A, \cdot, \Delta, \Phi, \Psi)$ is called \textit{commutative and cocommutative} if $(A, \cdot)$ is a commutative algebra and $(A, \Delta)$ is a cocommutative coalgebra.
\end{Definition}

\begin{Theorem}\label{thm:asimat}
 Let $(A, \cdot, \Phi)$ be a differential algebra.
 Suppose that there is a differential algebra structure $(A^*, \circ, \Psi^*)$ on $A^*$.
 Let $\Delta\colon A \to A \otimes A$ denote the linear dual of the multiplication $\circ\colon A^* \otimes A^* \to A^*$, that is, $(A, \Delta, \Psi)$ is a differential coalgebra.
 Then the quintuple $(A, \cdot, \Delta, \Phi, \Psi)$ is a differential ASI bialgebra if and only if $((A, \cdot, \Phi), (A^*, \circ, \Psi^*), R^*_A, L^*_A, R^*_{A^*}, L^*_{A^*})$ is a matched pair of differential algebras.
\end{Theorem}
\begin{proof}
 ($\Longrightarrow$)
 If the quintuple $(A, \cdot, \Delta, \Phi, \Psi)$ is a differential ASI bialgebra,
 then $(A, \cdot, \Delta)$ is an ASI bialgebra,
 and $\Psi$ and $\Phi^*$ are admissible to $(A, \cdot, \Phi)$ and $(A^*, \circ, \Psi^*)$ respectively.
 The former means that $((A, \cdot), (A^*, \circ), R^*_A, L^*_A, R^*_{A^*}, L^*_{A^*})$ is a matched pair of algebras by \cite[Theorem~2.2.3]{BaiDouble}.
 The latter means that $(A^*, R^*_A, L^*_A, \Psi^*)$ is a~bimodule of $(A, \cdot, \Phi)$
 and $(A, R^*_{A^*}, L^*_{A^*}, \Phi)$ is a~bimodule of $(A^*, \circ, \Psi^*)$.
 Hence $((A, \cdot, \Phi), (A^*, \circ, \Psi^*), R^*_A, L^*_A, R^*_{A^*}, L^*_{A^*})$ is a matched pair of differential algebras.

 ($\Longleftarrow$)
 If $((A, \cdot, \Phi), (A^*, \circ, \Psi^*), R^*_A, L^*_A, R^*_{A^*}, L^*_{A^*})$ is a matched pair of differential algebras,
 then $((A, \cdot)$, $(A^*, \circ)$, $R^*_A, L^*_A, R^*_{A^*}, L^*_{A^*})$ is a matched pair of algebras,
 and $\Psi$ and $\Phi^*$ are admissible to $(A, \cdot, \Phi)$ and $(A^*, \circ, \Psi^*)$ respectively.
 By \cite[Theorem~2.2.3]{BaiDouble} again, $(A, \cdot, \Delta)$ is an ASI bialgebra and hence $(A, \cdot, \Delta, \Phi, \Psi)$ is a differential ASI bialgebra.
\end{proof}

Combining Theorems~\ref{thm:doumat} and \ref{thm:asimat}, we have the following conclusion.
\begin{Theorem}\label{thm:asieqv}
 Let $(A, \cdot, \Phi)$ be a differential algebra.
 Suppose that there is a differential algebra structure $(A^*, \circ, \Psi^*)$ on $A^*$.
 Let $\Delta\colon A \to A \otimes A$ denote the linear dual of the multiplication $\circ\colon A^* \otimes A^* \to A^*$.
 Then the following conditions are equivalent:
 \begin{enumerate}\itemsep=0pt
 \item[$1.$] There is a double construction of differential Frobenius algebra associated to $(A, \cdot, \Phi)$ and $(A^*, \circ$, $\Psi^*)$.
 \item[$2.$] $((A, \cdot, \Phi), (A^*, \circ, \Psi^*), R^*_A, L^*_A, R^*_{A^*}, L^*_{A^*})$ is a matched pair of differential algebras.
 \item[$3.$] $(A, \cdot, \Delta, \Phi, \Psi)$ is a differential ASI bialgebra.
 \end{enumerate}
\end{Theorem}

\subsection{Coherent derivations on ASI bialgebras}
In this subsection, we consider the case that the set $\Phi$ in a differential algebra $(A, \cdot, \Phi)$ contains exactly one derivation.

Definition~\ref{def:dasi} motivates us to give the following notion.

\begin{Definition}\label{def:chd} A \textit{coherent derivation} on an ASI bialgebra $(A, \cdot, \Delta)$ is a pair $(\partial, \eth)$,
 where $\partial$ is a derivation on the algebra $(A, \cdot)$ and $\eth$ is a coderivation on the coalgebra $(A, \Delta)$
 satisfying equations~\eqref{eq:qadm1}, \eqref{eq:qadm2}, \eqref{eq:psadm1} and \eqref{eq:psadm2}.
\end{Definition}

\begin{Corollary}\label{cor:chd}
Let $(A,\cdot,\Delta)$ be an ASI bialgebra. Then the following conditions are equivalent:
\begin{enumerate}\itemsep=0pt
\item[$1.$]%\label{it:11}
$(A, \cdot, \Delta, \{\partial\}, \{\eth\})$ is a differential ASI bialgebra.
\item[$2.$]%\label{it:12}
$(\partial, \eth)$ is a coherent derivation on the ASI bialgebra $(A, \cdot, \Delta)$.
\item[$3.$]%\label{it:13}
$\partial + \eth^*$ is a derivation on the algebra $\big(A \bowtie_{R_A^*, L_A^*}^{R_{A^*}^*, L_{A^*}^*}
A^*, \star\big)$, where the algebra structure on~$A^*$ is given by~$\Delta^*$.
\end{enumerate}
\end{Corollary}

\begin{proof}
(1) $\Longleftrightarrow$ (2). It follows from
Definitions~\ref{def:dasi} and~\ref{def:chd}.

(1) $\Longleftrightarrow$ (3). It follows from
Theorems~\ref{thm:asimat} and~\ref{thm:matda}.
\end{proof}

Recall \cite{KP,leger2000generalized} that $f \in \End(A)$ is called a \textit{generalized derivation} on an algebra $(A, \cdot)$ if there exist $f', f'' \in \End(A)$ such that
\begin{equation*}
	f(a) \cdot b + a \cdot f'(b) = f''(a \cdot b), \qquad \forall a, b \in
	A.
\end{equation*}
For a coherent derivation $(\partial, \eth)$ on an ASI bialgebra $(A, \cdot,
\Delta)$, we have the following conclusions.
\begin{enumerate}\itemsep=0pt
	\item
	Since $\partial$ is a derivation on the algebra $(A,\cdot)$,
	it is a generalized derivation automatically.
	By equation~\eqref{eq:qadm1}, $\eth$ is a generalized derivation on $(A,\cdot)$.
	By equation~\eqref{eq:qadm2}, we get again that $\partial$ is a generalized derivation on $(A,\cdot)$ from another approach.

	\item
	Since $\eth^*$ is a derivation on the algebra $(A^*,\Delta^*)$, it is a generalized derivation automatically.
	Due to Lemma~\ref{lem:coadm}, $\partial^*$ is a generalized derivation on $(A^*, \Delta^*)$ by equation~\eqref{eq:psadm1}.
	Moreover, by equation~\eqref{eq:psadm2}, we show again that $\eth^*$ is a generalized derivation on $(A^*, \Delta^*)$ from another approach.
\end{enumerate}

Following the notion of a derivation on a Lie bialgebra \cite{chari1995guide}, we give the following notion.

\begin{Definition}A \textit{derivation} on an ASI bialgebra $(A, \cdot, \Delta)$ is a linear map $\partial\colon A \to A$ such that
 $\partial$ is both a derivation on the algebra $(A, \cdot)$ and a coderivation on the coalgebra $(A, \Delta)$.
\end{Definition}
Note that such derivations on ASI bialgebras are called biderivations in~\cite{aguiar2004infinitesimal}.
However, the notion of a biderivation (defined on an algebra $(A,\cdot)$) usually refers to a bilinear map ${f\colon A\otimes A \rightarrow A}$ satisfying certain conditions, which is a different structure~\cite{liu2018biderivations}.
Hence in order to avoid the possible confusion, we adopt the present notion.

\begin{Example}
Every ASI bialgebra has a canonical derivation.
Let $(A, \cdot, \Delta)$ be an ASI bialgebra and $\partial\colon A \to A$ be
the composite of the comultiplication $\Delta\colon A \to A \otimes A$,
the flip operator~$\sigma$, and the multiplication $\cdot\colon A \otimes A \to A$,
i.e., $\partial(a) = \sum_i a_i^2 \cdot a_i^1$ if $\Delta(a) = \sum_i a_i^1 \otimes a_i^2$
for all $a \in A$.
Then $\partial$ is a derivation on the ASI bialgebra $(A, \cdot, \Delta)$.
In fact, let $a, b \in A$ and write
$\Delta(a) = \sum_i a_i^1 \otimes a_i^2$, $\Delta(b) = \sum_i b_i^1 \otimes b_i^2$.
Hence we have
\begin{align*}
 \partial(a \cdot b) & = \cdot \sigma \Delta(a \cdot b) \overset{\eqref{eq:asiifs}}{=} \cdot \sigma \bigg(\sum_i a_i^1 \cdot b \otimes a_i^2 + \sum_i b_i^1 \otimes a \cdot b_i^2\bigg)\\
 & = \sum_i a_i^2 \cdot a_i^1 \cdot b + \sum_i a \cdot b_i^2 \cdot b_i^1 = \partial(a) \cdot b + a \cdot \partial(b),
\end{align*}
that is, $\partial$ is a derivation on $(A, \cdot)$.
Applying the preceding argument to $(A^*, \Delta^*)$,
we obtain $\partial^*$ is a derivation on $(A^*, \Delta^*)$.
Thus $\partial$ is a derivation on the ASI bialgebra $(A, \cdot, \Delta)$.
\end{Example}

\begin{Proposition}
 Let $(A, \cdot, \Delta)$ be an ASI bialgebra.
 Then $\partial$ is a derivation on $(A, \cdot, \Delta)$ if and only if $(\partial, -\partial)$ is a coherent derivation on $(A, \cdot, \Delta)$.
\end{Proposition}
\begin{proof}
 ($\Longrightarrow$)
 Take $\eth = -\partial$.
 Since $\partial$ is a coderivation on $(A, \Delta)$, $\eth$ is a coderivation.
 Moreover, equations~\eqref{eq:qadm1}, \eqref{eq:qadm2}, \eqref{eq:psadm1} and~\eqref{eq:psadm2} hold naturally.
 Thus $(\partial, -\partial)$ is a coherent derivation on the ASI bialgebra $(A, \cdot, \Delta)$.

 ($\Longleftarrow$)
 Since $-\partial$ is a coderivation on $(A, \Delta)$, $\partial$ is a coderivation.
 Thus $\partial$ is a derivation on the ASI bialgebra $(A, \cdot, \Delta)$.
\end{proof}

The notions of coherent homomorphisms and coherent isomorphisms on Lie bialgebras were introduced in \cite{BGS} to interpret the categorial equivalences among Lie bialgebras, Manin triples and matched pairs.
Transferring them to the context of ASI bialgebras, we give the following notion.

\begin{Definition}
 A \textit{coherent endomorphism} on an ASI bialgebra $(A, \cdot, \Delta)$ is a pair $(\phi, \psi)$
 consisting of an algebra endomorphism $\phi\colon A \to A$ on the algebra $(A,\cdot)$ and a coalgebra endomorphism $\psi\colon A \to A$ on the coalgebra $(A,\Delta)$, that is, $\psi$ is a linear map satisfying $(\psi \otimes \psi)\Delta = \Delta \psi$,
 such that
 \begin{gather}
 \psi(\phi(a) \cdot b) = a \cdot \psi(b), \label{eq:coendo1} \\
 \psi(a \cdot \phi(b)) = \psi(a) \cdot b, \qquad \forall a, b \in A, \label{eq:coendo2} \\
 (\id \otimes \phi)\Delta = (\psi \otimes \id )\Delta \phi, \label{eq:coendo3}\\
 (\phi \otimes \id)\Delta = (\id \otimes \psi) \Delta \phi. \label{eq:coendo4}
 \end{gather}
 If in addition both $\phi$ and $\psi$ are bijective,
 then the pair $(\phi, \psi)$ is called a \textit{coherent automorphism} on the ASI bialgebra $(A, \cdot, \Delta)$.
\end{Definition}

\begin{Lemma}\label{lem:endo}
 Let $(A, \cdot, \Delta)$ be an ASI bialgebra and $\phi, \psi\colon A \to A$ be linear maps.
 Then $(\phi, \psi)$ is a~coherent endomorphism on $(A, \cdot, \Delta)$ if and only if $\phi + \psi^*$ is an endomorphism on ${\big(A \bowtie_{R_A^*, L_A^*}^{R_{A^*}^*, L_{A^*}^*} A^*, \star\big)}$, where the algebra structure on $A^*$ is given by $\Delta^*$.
\end{Lemma}
\begin{proof}
It follows from a similar proof as the one for the equivalence between items~(1) and~(3) in Corollary~\ref{cor:chd} or a direct proof as follows.
Note that $\phi + \psi^*$ is an endomorphism on ${\big(A \bowtie_{R_A^*, L_A^*}^{R_{A^*}^*, L_{A^*}^*} A^*, \star\big)}$ if and only if
 \begin{gather*}
 (\phi + \psi^*)(a + a^*) \star (\phi + \psi^*)(b+b^*) = (\phi + \psi^*)((a+a^*) \star (b+b^*)),
 \end{gather*}
for all $a, b \in A$, $a^*, b^* \in A^*$. By equation~\eqref{eq:multsum}, the above equation holds if and only if the following equations hold:
 \begin{alignat*}{3}
 &\psi^*(a^*) \circ \psi^*(b^*) = \psi^*(a^* \circ b^*), \qquad&&R_{A}^*(\phi(a))\psi^*(b^*) = \psi^*(R_{A}^*(a) b^*),& \\
 &L_{A}^*(\phi(b))\psi^*(a^*) = \psi^*(L_{A}^*(b) a^*), \qquad&& \phi(a) \cdot \phi(b) = \phi(a \cdot b),& \\
 &L_{A^*}^*(\psi^*(b^*))\phi(a) = \phi(L_{A^*}^*(b^*) a), \qquad &&R_{A^*}^*(\psi^*(a^*))\phi(b) = \phi(R_{A^*}^*(a^*) b),&
 \end{alignat*}
 for all $a, b \in A$ and $a^*, b^* \in A^*$. Thus these equations hold if and only if $(\phi, \psi)$ is a coherent endomorphism on $(A, \cdot,
 \Delta)$. In fact, as an example, we give an explicit proof for the case that $L_{A^*}^*(\psi^*(b^*))\phi(a) = \phi(L_{A^*}^*(b^*) a)$ for all $a\in A, b^*\in A^*$ if and only if equation~\eqref{eq:coendo3}
 holds and the proofs for the other cases are similar. Note
 that for all $a \in A$ and $a^*, b^* \in A^*$, we have
 \begin{gather*}
 \langle a^*, L_{A^*}^*(\psi^*(b^*))\phi(a)\rangle = \langle \psi^*(b^*) \circ a^*, \phi(a)\rangle = \langle b^* \otimes a^*, (\psi \otimes \id)\Delta(\phi(a))\rangle, \\
 \langle a^*, \phi(L_{A^*}^*(b^*) a)\rangle = \langle b^* \circ \phi^*(a^*), a\rangle = \langle b^* \otimes a^*, (\id \otimes
 \phi)\Delta(a)\rangle.
 \end{gather*}
 Hence the conclusion follows and thus the proof is completed.
\end{proof}

\begin{Proposition}\label{pro:autder}
 Suppose that $(A,\cdot, \Delta)$ is an ASI bialgebra over the real number field $\mathbb{R}$ and $\partial, \eth\colon A \to A$ are linear maps.
 Then $(\partial, \eth)$ is a coherent derivation on $(A, \cdot, \Delta)$ if and only if $(e^{t\partial},e^{t\eth})$ is a coherent automorphism on $(A, \cdot, \Delta)$ for all $t \in \mathbb{R}$.
\end{Proposition}
\begin{proof}
Let $t \in \mathbb{R}$.
Note that $e^{t\partial}, e^{t\eth}$ and $e^{t (\partial + \eth^*)}$ are invertible.
By Lemma~\ref{lem:endo}, $\big(e^{t\partial}, e^{t\eth}\big)$ is a~coherent automorphism if and only if $e^{t\partial} + \big(e^{t\eth}\big)^* = e^{t(\partial + \eth^*)}$ is an automorphism on the algebra $\big(A \bowtie_{R_A^*, L_A^*}^{R_{A^*}^*, L_{A^*}^*} A^*, \star\big)$, where the algebra structure on~$A^*$ is given by $\Delta^*$.
On the other hand, it is known that for all $t \in \mathbb{R}$, $e^{t (\partial + \eth^*)}$ is an automorphism on $\big(A \bowtie_{R_A^*, L_A^*}^{R_{A^*}^*, L_{A^*}^*} A^*, \star\big)$,
if and only if $\partial + \eth^*$ is a derivation on $\big(A \bowtie_{R_A^*, L_A^*}^{R_{A^*}^*, L_{A^*}^*} A^*, \star\big)$.
Hence by Corollary~\ref{cor:chd}, the conclusion holds.
\end{proof}

Given a Lie algebra $(\mathfrak{g}, [\ ,\ ])$, we define the \textit{opposite Lie algebra} $\mathfrak{g}^{\rm op}$ as the vector space $\mathfrak{g}$ with the Lie bracket $[\ ,\ ]^{\rm op}$ given by
\begin{equation*}
 [a, b]^{\rm op} := [b, a] = -[a, b], \qquad \forall a, b \in \mathfrak{g}.
\end{equation*}
In particular, let $V$ be a vector space and $\mathfrak{gl}(V)$ be the general linear Lie algebra with the Lie bracket $[S, T] = S T - T S$ for all $S, T \in \mathfrak{gl}(V)$.
Then $\mathfrak{gl}(V)^{\rm op} \oplus \mathfrak{gl}(V)$ is a Lie algebra with the following Lie bracket
\begin{equation}
 [(S_1, S_2), (T_1, T_2)] := (-[S_1, T_1], [S_2, T_2]) = (T_1 S_1 - S_1 T_1, S_2 T_2 - T_2 S_2), \label{eq:coherent}
\end{equation}
for all $S_1, S_2, T_1, T_2 \in \mathfrak{gl}(V)$.

Given a Lie group $G$, let $G^{\rm op}$ be the opposite group of
$G$ with the same manifold structure as $G$. Then $G^{\rm op}$ is
a Lie group, called the \textit{opposite Lie group} of $G$. Let
${\rm Lie}(G)$ denote the Lie algebra of the Lie group $G$. Then
${\rm Lie}(G^{\rm op}) = ({\rm Lie}(G))^{\rm op}$ \cite{nicolas1989lie}.
In particular, let $V$ be a vector space and ${\rm GL}(V)$ be the general
linear Lie group. Then the direct product ${\rm GL}(V)^{\rm op} \times
{\rm GL}(V)$ is a Lie group with the following multiplication:
\begin{equation}
 (\alpha_1, \alpha_2) (\beta_1, \beta_2) = (\beta_1 \alpha_1, \alpha_2 \beta_2), \qquad \forall \alpha_1, \alpha_2, \beta_1, \beta_2 \in {\rm GL}(V). \label{eq:coherent2}
\end{equation}

Moreover, we have ${\rm Lie}({\rm GL}(V)^{\rm op} \times {\rm GL}(V)) \cong {\rm Lie}({\rm GL}(V)^{\rm op}) \oplus {\rm Lie}({\rm GL}(V)) = \mathfrak{gl}(V)^{\rm op} \oplus \mathfrak{gl}(V)$.

\begin{Theorem}
 Let $(A, \cdot, \Delta)$ be a real ASI bialgebra.
 Let $G$ be the set of coherent automorphisms on $(A, \cdot, \Delta)$ and $\mathfrak{D}$ be the set of coherent derivations on $(A, \cdot, \Delta)$.
 Then $G$ is a closed Lie subgroup of ${\rm GL}(A)^{\rm op} \times {\rm GL}(A)$ and $\mathfrak{D}$ is a Lie subalgebra of $\mathfrak{gl}(A)^{\rm op} \oplus \mathfrak{gl}(A)$.
 Moreover the Lie algebra of~$G$ is exactly~$\mathfrak D$.
\end{Theorem}
\begin{proof}
It is straightforward to verify that $\mathfrak{D}$ is a vector space by
\begin{equation*}
 \lambda (\partial_1, \eth_1) + \mu (\partial_2, \eth_2) = (\lambda \partial_1 + \mu\eth_1, \lambda \partial_2 + \mu \eth_2), \qquad
 \forall (\partial_1, \eth_1), (\partial_2, \eth_2) \in \mathfrak{D},\quad
 \lambda, \mu \in \mathbb{R}.
\end{equation*}
Let $(\partial_1, \eth_2),(\partial_2, \eth_2) \in \mathfrak{D}$ and $a, b \in A$.
It is known that $-[\partial_1,\partial_2]$ is a derivation on $(A, \cdot)$ and $[\eth_1,\eth_2]$ is a coderivation on $(A, \Delta)$.
Moreover, we have
\begin{gather*}
 [\eth_1,\eth_2](a \cdot b) = \eth_1(\eth_2(a \cdot b)) - \eth_2(\eth_1(a \cdot b)) \\
 \hspace*{19mm}{}
 \overset{\eqref{eq:qadm2}}{=} a \cdot \eth_1(\eth_2(b)) - \partial_1(a) \cdot \eth_2(b) - (\partial_2(a) \cdot \eth_1(b) - \partial_1(\partial_2(a)) \cdot b ) \\
 \hphantom{[\eth_1,\eth_2](a \cdot b) =}{}
 - (a \cdot \eth_2(\eth_1(b)) - \partial_2(a) \cdot \eth_1(b)) + (\partial_1(a) \cdot \eth_2(b) - \partial_2(\partial_1(a)) \cdot b ) \\
 \hphantom{[\eth_1,\eth_2](a \cdot b)}{}= a \cdot [\eth_1,\eth_2](b) + [\partial_1,\partial_2](a) \cdot b, \\
 \Delta [\partial_1,\partial_2] = \Delta \partial_1 \partial_2 - \Delta \partial_2 \partial_1 \\
 \hspace*{13mm}{} \overset{\eqref{eq:psadm1}}{=} (\partial_1 \otimes \id - \id \otimes \eth_1) (\partial_2 \otimes \id - \id \otimes \eth_2)\Delta \\
\hphantom{\Delta [\partial_1,\partial_2]=}{}
 - (\partial_2 \otimes \id - \id \otimes \eth_2)(\partial_1 \otimes \id - \id \otimes \eth_1)\Delta \\
 \hphantom{\Delta [\partial_1,\partial_2]}{}
= ([\partial_1,\partial_2] \otimes \id) \Delta + (\id \otimes [\eth_1,\eth_2])\Delta.
\end{gather*}
Similarly we have
\begin{gather*}
 [\eth_1,\eth_2](a \cdot b) = a \cdot [\eth_1,\eth_2](b) + [\partial_1,\partial_2](a) \cdot b, \\
 \Delta [\partial_1, \partial_2] = (\id \otimes [\partial_1, \partial_2])\Delta + ([\eth_1,\eth_2] \otimes \id )\Delta.
\end{gather*}
Thus $(-[\partial_1,\partial_2], [\eth_1,\eth_2])$ is a coherent derivation on $(A, \cdot, \Delta)$.
Then $\mathfrak{D}$ is a Lie algebra with the Lie bracket defined by equation~(\ref{eq:coherent}).
Hence $\mathfrak{D}$ is a Lie subalgebra of $\mathfrak{gl}(A)^{\rm op} \oplus \mathfrak{gl}(A)$.

Let $(\phi_1, \psi_1), (\phi_2,\psi_2) \in G$ and $a, b \in A$.
It is known that $\phi_2 \phi_1$ is an algebra automorphism on $(A, \cdot)$ and $\psi_1 \psi_2$ is a coalgebra automorphism. Moreover, we have
\begin{gather*}
 (\psi_1 \psi_2)((\phi_2 \phi_1)(a) \cdot b) \overset{\eqref{eq:coendo1}}{=} \psi_1(\phi_1(a) \cdot \psi_2(b)) \overset{\eqref{eq:coendo1}}{=} a \cdot (\psi_1\psi_2)(b), \\
 (\id \otimes \phi_2 \phi_1)\Delta \overset{\eqref{eq:coendo3}}{=} (\psi_1 \otimes \phi_2)\Delta \phi_1 \overset{\eqref{eq:coendo3}}{=} (\psi_1 \psi_2 \otimes \id)\Delta(\phi_2 \phi_1).
\end{gather*}
Similarly, we have
\begin{equation*}
 (\psi_1 \psi_2)(a \cdot (\phi_2 \phi_1)(b) ) = (\psi_1\psi_2)(a) \cdot b, \qquad (\phi_2 \phi_1 \otimes \id)\Delta = (\id \otimes \psi_1 \psi_2)\Delta(\phi_2 \phi_1).
\end{equation*}
Thus $(\phi_2 \phi_1, \psi_1 \psi_2)$ is a coherent automorphism.
Meanwhile, it is straightforward to verify $(\id, \id) \in G$ and $\big(\phi_1^{-1}, \psi_1^{-1}\big) \in G$.
Then $G$ is a group with the multiplication defined by equation~\eqref{eq:coherent2}.
Hence $G$ is a subgroup of ${\rm GL}(A)^{\rm op} \times {\rm GL}(A)$.

Moreover, $G$ is closed in ${\rm GL}(A)^{\rm op} \times {\rm GL}(A)$ since $A$ is finite-dimensional and $G$ is a subset of ${\rm GL}(A)^{\rm op} \times {\rm GL}(A)$ determined by equations~\eqref{eq:coendo1}--\eqref{eq:coendo4} and the following equations
\begin{equation*}
 \phi(a \cdot b) = \phi(a) \cdot \phi(b), \qquad (\psi \otimes \psi) \Delta = \Delta \psi, \qquad \forall a, b \in A.
\end{equation*}
Hence by \cite[Theorem~6.9]{SW}, $G$ is a closed Lie subgroup of ${\rm GL}(A)^{\rm op} \times {\rm GL}(A)$ and the Lie algebra of $G$ is
$\big\{(\partial, \eth) \in \mathfrak{gl}(A)^{\rm op} \oplus \mathfrak{gl}(A)\colon \big(e^{t\partial}, e^{t\eth}\big) \in G, \forall t \in \mathbb{R}\big\}$.
By Proposition~\ref{pro:autder}, we have ${\rm Lie}(G) = \mathfrak{D}$.
\end{proof}

%%%%%%%%%%%%%%%%%%%%%%%%%%%%%%%%%%%%%%%%%%%%%%%%%%%%%%%%%%%%%%%%%%%%%%%%%%%%%%%%
%%%%%%%%%%%%%%%%%%%%%%%%%%%%%%%%%%%%%%%%%%%%%%%%%%%%%%%%%%%%%%%%%%%%%%%%%%%%%%%%
%%%%%%%%%%%%%%%%%%%%%%%%%%%%%%%%%%%%%%%%%%%%%%%%%%%%%%%%%%%%%%%%%%%%%%%%%%%%%%%%
\section{Coboundary differential ASI bialgebras}\label{sec:cob}

We study the coboundary differential ASI bialgebras, leading to
the introduction of the notion of $\Psi$-admissible associative
Yang--Baxter equation (AYBE) in a differential algebra. The
antisymmetric solutions of the latter give the former. The notions
of $\mathcal{O}$-operators of differential algebras and
differential dendriform algebras are introduced to provide
antisymmetric solutions of $\Psi$-admissible AYBE in semi-direct
product differential algebras and hence give rise to differential
ASI bialgebras.

%%%%%%%%%%%%%%%%%%%%%%%%%%%%%%%%%%%%%%%%%%%%%%%%%%%%%%%%%%%%%%%%%%%%%%%%%%%%%%%%

\subsection{Coboundary differential ASI bialgebras and admissible AYBE}\label{subsec:cobdasi}

\begin{Definition}\label{def:cobdasi}
 A differential ASI bialgebra $(A, \cdot, \Delta, \Phi, \Psi)$ is called \textit{coboundary}
 if $\Delta$ is defined by
 \begin{equation}
 \Delta(a) := (\id \otimes L_A(a) - R_A(a) \otimes \id)(r), \qquad \forall a \in A, \label{eq:cbd}
 \end{equation}
 for some $r \in A \otimes A$.
\end{Definition}

\begin{Proposition}[{\cite[Theorem~2.3.5]{BaiDouble}}]\label{pro:cobasi}
Let $(A, \cdot)$ be an algebra and $r \in A \otimes A$.
 Define a linear map $\Delta\colon A \to A \otimes A$ by equation~\eqref{eq:cbd}.
 Then $(A, \cdot, \Delta)$ is an ASI bialgebra if and only if the following equations hold:
 \begin{gather}
 (L_A(a) \otimes \id - \id \otimes R_A(a))(\id \otimes L_A(b) - R_A(b) \otimes \id)(r + \sigma(r)) = 0, \label{eq:cobanti} \\
 (\id \otimes \id \otimes L_A(a) - R_A(a) \otimes \id \otimes \id)(r_{12}r_{13}+r_{13}r_{23}-r_{23}r_{12}) = 0, \qquad \forall a,b\in A. \label{eq:cobcoa}
 \end{gather}
 Here for $r = \sum_i a_i \otimes b_i$, we denote
 \begin{gather*}
 r_{12} r_{13} = \sum_{i,j} a_i \cdot a_j \otimes b_i \otimes b_j, \qquad
 r_{13} r_{23} = \sum_{i,j} a_i \otimes a_j \otimes b_i \cdot b_j, \\
 r_{23} r_{12} = \sum_{i,j} a_j \otimes a_i \cdot b_j \otimes b_i.
 \end{gather*}
\end{Proposition}

\begin{Proposition}\label{pro:cobAdmeq}
 Let $(A, \cdot, \Phi = \{\partial_k\}_{k=1}^m)$ be a $\Psi = \{\eth_k\}_{k=1}^m$-admissible differential algebra and $r \in A \otimes A$.
 Define a linear map $\Delta\colon A \to A \otimes A$ by equation~\eqref{eq:cbd}.
 Then the following conclusions hold.
 \begin{enumerate}\itemsep=0pt
 \item[$1.$] For all $k=1, \dots, m$, equation~\eqref{eq:coderivation} holds
 in which $\eth$ is replaced by $\eth_k$ if and only if the following equation holds:
 \begin{gather}
 ( \id \otimes L_A(a) )(\id \otimes \partial_k - \eth_k \otimes \id)(r) \nonumber\\
 \qquad{} + (R_A(a)\otimes \id )(\id \otimes \eth_k - \partial_k \otimes \id)(r) = 0, \qquad \forall a \in A. \label{eq:cobdcod}
 \end{gather}
 \item[$2.$] \label{it:2} For all $k=1, \dots, m$, equation~\eqref{eq:psadm1} holds if and only if the following equation holds:
 \begin{equation}
 (\id \otimes L_A(a) - R_A(a) \otimes \id )(\partial_k \otimes \id - \id \otimes \eth_k)(r) = 0, \qquad \forall a \in A. \label{eq:cobpsadm1}
 \end{equation}
 \item[$3.$] \label{it:3} For all $k=1, \dots, m$, equation~\eqref{eq:psadm2} holds if and only if the following equation holds:
 \begin{equation}
 (\id \otimes L_A(a) - R_A(a) \otimes \id )(\id \otimes \partial_k - \eth_k \otimes \id)(r) = 0, \qquad \forall a \in A. \label{eq:cobpsadm2}
 \end{equation}
 \end{enumerate}
\end{Proposition}
\begin{proof}
 (1)~By equations~\eqref{eq:qadm1}--\eqref{eq:qadm2}, we have
 \begin{gather}
 L_A(\eth_k(a)) = L_A(a) \partial_k + \eth_k L_A(a), \qquad\!
 R_A(\eth_k(a)) = R_A(a) \partial_k + \eth_k R_A(a), \qquad\!
 \forall a \in A.\!\!\! \label{eq:pf1}
 \end{gather}
 Then for all $a \in A$, we have
 \begin{gather*}
 \Delta \eth_k(a) - (\eth_k \otimes \id + \id \otimes \eth_k ) \Delta(a) \\
 \qquad{}\overset{\hphantom{(4.7)}}{=} (\id \otimes L_A(\eth_k(a)) - R_A(\eth_k(a)) \otimes \id )(r) \\
 \qquad\qquad {}- (\eth_k \otimes L_A(a) - \eth_k R_A(a) \otimes \id + \id \otimes \eth_k L_A(a) - R_A(a) \otimes \eth_k )(r) \\
 \qquad{}\overset{\eqref{eq:pf1}}{=} (\id \otimes L_A(a) \partial_k - R_A(a) \partial_k \otimes \id - \eth_k \otimes L_A(a) + R_A(a) \otimes \eth_k )(r) \\
 \qquad{}\overset{\hphantom{(4.7)}}{=} ( \id \otimes L_A(a) )(\id \otimes \partial_k - \eth_k \otimes \id)(r) + (R_A(a)\otimes \id )(\id \otimes \eth_k - \partial_k \otimes \id)(r).
 \end{gather*}
 Hence equation~\eqref{eq:coderivation} holds if and only if equation~\eqref{eq:cobdcod} holds.

 (2)~By equations~\eqref{eq:qadm2} and \eqref{eq:derivation}, we have
 \begin{gather}
 L_A(a) \eth_k = L_A(\partial_k(a)) + \eth_k L_A(a), \qquad\!
 \partial_k R_A(a) = R_A(a) \partial_k + R_A(\partial_k(a)), \qquad\!
 \forall a \in A.\!\!\! \label{eq:pf2}
 \end{gather}
 Then for all $a \in A$, we have
 \begin{gather*}
 (\partial_k \otimes \id)\Delta(a) - (\id \otimes \eth_k )\Delta(a) - \Delta \partial_k(a) \\
 \qquad{}\overset{\hphantom{(4.8)}}{=} ( \partial_k \otimes L_A(a) - \partial_k R_A(a) \otimes \id - \id \otimes \eth_k L_A(a) + R_A(a) \otimes \eth_k )(r) \\
 \qquad\qquad {} - (\id \otimes L_A(\partial_k(a)) - R_A(\partial_k(a)) \otimes \id )(r) \\
 \qquad{} \overset{\eqref{eq:pf2}}{=} ( \partial_k \otimes L_A(a) - \id \otimes L_A(a) \eth_k - R_A(a) \partial_k \otimes \id + R_A(a) \otimes \eth_k )(r) \\
 \qquad{} \overset{\hphantom{(4.8)}}{=} (\id \otimes L_A(a) - R_A(a) \otimes \id )(\partial_k \otimes \id - \id \otimes \eth_k)(r).
 \end{gather*}
 Hence equation~\eqref{eq:psadm1} holds if and only if equation~\eqref{eq:cobpsadm1} holds.

 (3)~It follows from a similar argument as the one of (2).
\end{proof}

Combining Propositions~\ref{pro:cobasi} and \ref{pro:cobAdmeq}, we have the following conclusion.
\begin{Corollary}\label{cor:cobdasiCon}
 Let $(A, \cdot, \Phi = \{\partial_k\}_{k=1}^m)$ be a $\Psi = \{\eth_k\}_{k=1}^m$-admissible differential algebra and $r \in A \otimes A$.
 Define a linear map $\Delta$ by equation~\eqref{eq:cbd}. Then $(A, \cdot, \Delta, \Phi, \Psi)$ is a differential ASI bialgebra if and only if
 equations~\eqref{eq:cobanti}--\eqref{eq:cobpsadm2} are satisfied.
\end{Corollary}

\begin{Theorem}
 Let $(A, \cdot, \Delta, \Phi, \Psi)$ be a differential ASI bialgebra.
 Let $\delta\colon A^* \to A^* \otimes A^*$ be the linear dual of the multiplication $\cdot$ on $A$ and $\circ\colon A^* \otimes A^* \to A^*$ be the linear dual of $\Delta$.
 Then $(A^*, \circ, -\delta, \Psi^*, \Phi^*)$ is a differential ASI bialgebra.
 Further, there is a differential ASI bialgebra structure on the direct sum $A \oplus A^*$ of the underlying vector spaces of $A$ and $A^*$, containing the two differential ASI bialgebras as differential ASI sub-bialgebras.
\end{Theorem}
\begin{proof}
 By \cite[Remark 2.2.4]{BaiDouble}, $(A^*, \circ,-\delta)$ is an ASI bialgebra.
 Moreover $\Psi$ is admissible to the differential algebra $(A, -\delta^*, \Phi)$ if and only if $\Psi$ is admissible to the differential algebra $(A, \delta^*, \Phi)$.
 Therefore with the fact that $\Phi^*$ is admissible to $(A^*, \circ, \Psi^*)$, we know that $(A^*, \circ, -\delta, \Psi^*, \Phi^*)$ is a differential ASI bialgebra.

 Let $r \in A \otimes A^* \subset (A \oplus A^*) \otimes (A \oplus A^*)$ correspond to the identity map $\id: A \to A$.
 Let $\{e_1,e_2,\dots,e_n\}$ be a basis of $A$ and $\{e_1^*,e_2^*,\dots,e_n^*\}$ be the dual basis.
 Then $r = \sum_{i=1}^n e_i \otimes e_i^*$.
 Let $(A \bowtie A^*, \star)$ be the algebra structure on $A \oplus A^*$ obtained from the matched pair of algebras $\big(A, A^*, R^*_A, L^*_A, R^*_{A^*}, L^*_{A^*}\big)$.
 Define
 \begin{equation*}
 \Delta_{A \bowtie A^{*}}(u) = (\id \otimes L_{A \bowtie A^{*}}(u) - R_{A \bowtie A^{*}}(u) \otimes \id)(r), \qquad \forall u \in A \bowtie A^{*}.
 \end{equation*}

	Moreover, $(A \bowtie A^*, \star, \Phi + \Psi^*)$ is a $(\Psi + \Phi^*)$-admissible differential algebra by Theorem~\ref{thm:asieqv} and Lemma~\ref{lem:douadm}.
 Set $\Phi = \{\partial_k\}_{k=1}^m$ and $\Psi = \{\eth_k\}_{k=1}^m$.
	For all $k=1, \dots, m$, we have
 \begin{gather*}
 ((\partial_k + \eth_k^*) \otimes \id - \id \otimes(\eth_k + \partial_k^*))(r) = ((\partial_k + \eth_k^*) \otimes \id - \id \otimes (\eth_k + \partial_k^*)) \left(\sum_{i=1}^n e_i \otimes e_i^*\right) \\
 \qquad{}= \sum_{i=1}^n (\partial_k(e_i) \otimes e_i^* - e_i \otimes \partial_k^*(e_i^*) ) = \sum_{i=1}^n \partial_k(e_i) \otimes e_i^* - \sum_{i=1}^n \sum_{j=1}^n e_i \otimes \langle\partial_k^*(e_i^*),e_j\rangle e_j^* \\
 \qquad{}= \sum_{i=1}^n \partial_k(e_i) \otimes e_i^* - \sum_{j=1}^n \sum_{i=1}^n \langle e_i^*, \partial_k(e_j) \rangle e_i \otimes e_j^* = \sum_{i=1}^n \partial_k(e_i)\otimes e_i^* -
 \sum_{j=1}^n \partial_k(e_j) \otimes e_j^* = 0,
	\end{gather*}
 and similarly $((\eth_k + \partial_k^*) \otimes \id - \id \otimes(\partial_k + \eth_k^*))(r) = 0$.
 Hence equations~\eqref{eq:cobdcod}--\eqref{eq:cobpsadm2} hold.
 By \cite[Theorem~2.3.6]{BaiDouble}, we show that $r$ satisfies equations~\eqref{eq:cobanti} and \eqref{eq:cobcoa}.
 Therefore $(A \bowtie A^*, \star, \Delta_{A \bowtie A^*}, \Phi + \Psi^*, \Psi + \Phi^*)$ is a differential ASI bialgebra by Corollary~\ref{cor:cobdasiCon}.
 Obviously it contains $(A, \cdot, \Delta, \Phi, \Psi)$ and $(A^*, \circ, -\delta, \Psi^*, \Phi^*)$ as differential ASI sub-bialgebras.
\end{proof}

We have the following conclusion as a consequence of Corollary~\ref{cor:cobdasiCon}.
\begin{Corollary}
 Let $(A, \cdot, \Phi = \{\partial_k\}_{k=1}^m)$ be a $\Psi = \{\eth_k\}_{k=1}^m$-admissible differential algebra and $r \in A \otimes A$.
 Then the linear map $\Delta$ defined by equation~\eqref{eq:cbd} makes $(A, \cdot, \Delta, \Phi, \Psi)$ be a differential ASI bialgebra
 if equation~\eqref{eq:cobanti} and the following equations hold:
 \begin{gather}
 r_{12} r_{13} + r_{13} r_{23} - r_{23} r_{12} = 0,\label{eq:aybe} \\
 (\partial_k \otimes \id - \id \otimes \eth_k)(r) = 0, \label{eq:pqadm1} \\
 (\eth_k \otimes \id - \id \otimes \partial_k)(r) = 0 \label{eq:pqadm2}
 \end{gather}
	for all $i=1, \dots, m$.
\end{Corollary}

\begin{Definition}
 Let $(A, \cdot, \Phi = \{\partial_k\}_{k=1}^m)$ be a differential algebra.
 Suppose that $r \in A \otimes A$ and $\Psi = \{\eth_k\colon A \to A\}_{k=1}^m$ is a set of commuting linear maps.
 Then equation~\eqref{eq:aybe} with conditions given by equations~\eqref{eq:pqadm1} and \eqref{eq:pqadm2} is called
 \textit{$\Psi$-admissible associative Yang--Baxter equation} in $(A, \cdot, \Phi)$ or simply \textit{$\Psi$-admissible AYBE} in $(A, \cdot, \Phi)$.
\end{Definition}

\begin{Remark}
 Note that equation~\eqref{eq:aybe} is exactly the associative Yang--Baxter equation (AYBE) in an algebra.
 Also note that if $r$ is antisymmetric (i.e., $r = - \sigma(r)$),
 then equation~\eqref{eq:cobanti} holds naturally, and in this case, equation~\eqref{eq:pqadm1} holds if and only if equation~\eqref{eq:pqadm2} holds.
\end{Remark}

In terms of $\Psi$-admissible AYBE, we have the following conclusion.
\begin{Corollary}\label{cor:admAYBEdasi}
 Let $(A, \cdot, \Phi)$ be a $\Psi$-admissible differential algebra and $r \in A \otimes A$.
 If $r$ is an antisymmetric solution of $\Psi$-admissible AYBE in $(A, \cdot, \Phi)$,
 then $(A, \cdot, \Delta, \Phi ,\Psi)$ is a differential ASI bialgebra,
 where the linear map $\Delta$ is defined by equation~\eqref{eq:cbd}.
\end{Corollary}

For a vector space $A$, through the isomorphism $A \otimes A \cong \Hom(A^*,A)$, any $r = \sum_{i} a_i \otimes b_i \in A \otimes A$ can be identified as a map from $A^*$ to $A$, which we denote by $r^\sharp$, explicitly,
\begin{equation*}
 r^\sharp\colon \ A^* \to A, \qquad a^* \mapsto \sum_{i} \langle a^*, a_i \rangle b_i, \qquad
 \forall a^* \in A^*.
\end{equation*}

\begin{Theorem}\label{thm:AYBEO}
 Let $(A, \cdot, \Phi = \{\partial_k\}_{k=1}^m)$ be a differential algebra and $r \in A \otimes A$ be antisymmetric.
 Let $\Psi = \{\eth_k\colon A \to A\}_{k=1}^m$ be a set of commuting linear maps.
 Then $r$ is a solution of $\Psi$-admissible AYBE in $(A, \cdot, \Phi)$ if and only if $r^\sharp$
 satisfies the following equations:
 \begin{gather}
 r^\sharp(a^*) \cdot r^\sharp(b^*) = r^\sharp\big( R^*_A(r^\sharp(a^*)) b^* + L^*_A(r^\sharp(b^*)) a^* \big), \qquad \forall a^*, b^* \in A^*, \label{eq:AYBEO1} \\
 \partial_k r^\sharp = r^\sharp \eth_k^*, \qquad \forall k=1, \dots, m. \label{eq:AYBEO2}
 \end{gather}
\end{Theorem}
\begin{proof}
 By \cite[Proposition~2.4.7]{BaiDouble},
 $r$ is a solution of AYBE in $(A, \cdot)$ if and only if equation~\eqref{eq:AYBEO1} holds.
 Moreover, writing $r = \sum_i a_i \otimes b_i$, then for all $a^* \in A^*$ and $k=1, \dots,
 m$, we have
 \begin{equation*}
 r^\sharp(\eth_k^*(a^*)) = \sum_{i} \langle \eth_k^*(a^*), a_i \rangle b_i = \sum_{i} \langle a^*, \eth_k(a_i) \rangle b_i, \qquad \partial_k(r^\sharp(a^*)) = \sum_{i}\langle a^*, a_i \rangle \partial_k(b_i).
 \end{equation*}
 So $\partial_k r^\sharp = r^\sharp \eth_k^*$ if and only if equation~\eqref{eq:pqadm2} holds.
\end{proof}

Now let $(A, \cdot, \Phi, \mathfrak{B})$ be a symmetric
differential Frobenius algebra. Then under the natural bijection
$\Hom(A \otimes A, \mathbb{F}) \cong \Hom(A, A^*)$, the bilinear
form $\mathfrak{B}$ corresponds to the linear map (see the proof
of Proposition~\ref{pro:FroRep})
\begin{equation*}
 \varphi\colon \ A \to A^*, \qquad a \mapsto \varphi(a),\qquad
 \text{where} \quad \langle \varphi(a), b \rangle := \mathfrak{B}(a, b), \quad
 \forall a, b \in A.
\end{equation*}
For any $r \in A \otimes A$, define a linear map $P_r\colon A \to A$ by
\begin{equation*}
 P_r\colon \ A \to A, \qquad a \mapsto r^\sharp(\varphi(a)), \qquad \forall a \in A.
\end{equation*}

\begin{Theorem}
 Let $(A, \cdot, \Phi = \{\partial_k\}_{k=1}^m, \mathfrak{B})$ be a symmetric differential Frobenius algebra and $r \in A \otimes A$ be antisymmetric.
 Suppose that $\hat{\Phi}$ is the adjoint of $\Phi$ with respect to $\mathfrak{B}$.
	Then $r$ is a~solution of $\hat{\Phi}$-admissible AYBE in $(A, \cdot, \Phi)$ if and only if
 $P_r$ satisfies the following equations:
 \begin{gather}
 P_r(a) \cdot P_r(b) = P_r(a \cdot P_r(b) + P_r(a) \cdot b), \qquad \forall a, b \in
 A,\label{eq:RB0}\\
		\partial_k P_r = P_r \partial_k,\qquad \forall k=1, \dots, m.
 \end{gather}
Moreover, in this case, $\big(A, \cdot, \Delta, \Phi, \hat{\Phi}\big)$ is
a differential ASI bialgebra, where $\Delta$ is defined by
equation~\eqref{eq:cbd}.
\end{Theorem}
\begin{proof}
 By \cite[Corollary~3.17]{Bai2013O},
 $r$ is a solution of AYBE in the algebra $(A, \cdot)$ if and only if $P_r$ satisfies equation~(\ref{eq:RB0}).
 Moreover, set $r = \sum_i a_i\otimes b_i$.
 For all $a \in A$, we have
 \begin{gather*}
 \partial_k P_r(a) = \partial_k r^\sharp (\varphi(a)), \\
 P_r \partial_k(a) = r^\sharp(\varphi(\partial_k(a))) = \sum_i \mathfrak{B}(\partial_k(a),a_i) b_i = \sum_i \mathfrak{B}(a, \hat{\partial_k}(a_i)) b_i \\
 \hphantom{P_r \partial_k(a)}{}
 = \sum_i \langle \varphi(a), \hat{\partial_k}(a_i) \rangle b_i = r^\sharp \hat{\partial_k}^*(\varphi(a)).
 \end{gather*}
 Note that $\varphi$ is a linear isomorphism. Thus $\partial_k r^\sharp = r^\sharp \hat{\partial_k}^*$ if and only if $\partial_k P_r = P_r \partial_k$.
 Hence the conclusion follows from Theorem~\ref{thm:AYBEO}, Proposition~\ref{pro:FroRep} and Corollary~\ref{cor:admAYBEdasi}.
\end{proof}

\subsection[O-operators of differential algebras]{$\boldsymbol{\mathcal{O}}$-operators of differential algebras}\label{subsec:o_operator}

Theorem~\ref{thm:AYBEO} leads us to give the following notion.
\begin{Definition}
	Let $(A, \cdot, \Phi = \{\partial_k\}_{k=1}^m)$ be a differential algebra and $(V, l, r, \Omega)$ be a bimodule of $(A, \cdot, \Phi)$.
	A linear map $T\colon V \to A$ is called an \textit{$\mathcal{O}$-operator of $(A, \cdot, \Phi)$ associated to $(V, l, r, \Omega)$} if $T$ satisfies
	\begin{gather}
		T(u) \cdot T(v) = T( l(T(u))v + r(T(v)) u ), \qquad \forall u, v \in V, \label{eq:oop1} \\
		\partial_k T = T \alpha_k, \qquad \forall k=1, \dots, m. \label{eq:oop2}
	\end{gather}
\end{Definition}

\begin{Example}\label{example:o_operator}
 Let $(A, \cdot, \Phi)$ be a differential algebra.
 Then the identity map $\id\colon A \to A$ is an $\mathcal{O}$-operator of $(A, \cdot, \Phi)$ associated to $(A, L_A, 0, \Phi)$ or $(A, 0, R_A, \Phi)$.
\end{Example}

Note \cite{BaiDouble} that for an algebra $(A,\cdot)$ and an $A$-bimodule $(V, l, r)$, a linear map $T\colon V \to A$ satisfying equation~\eqref{eq:oop1} is called an \textit{$\mathcal{O}$-operator of $(A, \cdot)$ associated to $(V, l, r)$}.
In particular, when $(V,l,r)$ is taken to be $(A, L_A, R_A)$, an $\mathcal O$-operator $R\colon A \to A$ is called a \textit{Rota--Baxter operator (of weight zero)} on $(A, \cdot)$, that is, $R$ satisfies
\begin{equation*}
	R(a) \cdot R(b) = R(R(a) \cdot b + a \cdot R(b)), \qquad \forall a, b \in A.
\end{equation*}
Hence equation~(\ref{eq:RB0}) means that $P_r$ is a Rota--Baxter operator on the algebra $(A,\cdot)$.

Furthermore we have the following conclusion.
\begin{Lemma}[\cite{BGN0}]\label{lem:o2r}
	Let $(A, \cdot)$ be an algebra and $(V, l, r)$ be an $A$-bimodule.
	A linear map $T\colon V \to A$ is an $\mathcal{O}$-operator of $(A, \cdot)$ associated to $(V, l, r)$ if and only if $\hat T:=0 + T\colon A \oplus V \to A \oplus V$ is a Rota--Baxter operator on the semi-direct product algebra $(A \ltimes_{l, r} V, \cdot)$,
	where the linear map $\hat T$ is defined as
	\begin{equation}
		\hat T(a+u)=T(u),\qquad \forall a\in A,\quad u\in V. \label{eq:hatT}
	\end{equation}
\end{Lemma}

In terms of $\mathcal{O}$-operators, Theorem~\ref{thm:AYBEO} is rewritten as follows.
\begin{Corollary}\label{cor:aybe_o}
 Let $(A, \cdot, \Phi = \{\partial_k\}_{k=1}^m)$ be a differential algebra.
 Let $r \in A \otimes A$ be antisymmetric and
 $\Psi = \{\eth_k\colon A \to A\}_{k=1}^m$ be a set of commuting linear maps.
 Then $r$ is a solution of $\Psi$-admissible AYBE in $(A, \cdot, \Phi)$ if and only if
 $r^\sharp$ is an
	$\mathcal{O}$-operator of $(A, \cdot)$ associated to $(A^*, R_A^*, L_A^*)$ such that $\partial_k r^\sharp = r^\sharp \eth_k^*$ for all $k=1, \dots, m$.
 If in addition, $(A, \cdot, \Phi)$ is $\Psi$-admissible,
 then $r$ is a solution of $\Psi$-admissible AYBE in $(A, \cdot, \Phi)$ if and only if $r^\sharp$ is an $\mathcal{O}$-operator of $(A, \cdot, \Phi)$ associated to the bimodule $(A^*, R_A^*, L_A^*, \Psi^*)$.
\end{Corollary}

Recall that $\mathcal{O}$-operators of algebras give antisymmetric
solutions of AYBE in semi-direct product algebras.

\begin{Proposition}[{\cite[Corollary~3.10]{Bai2013O}}]\label{pro:O-cons}
 Let $(A, \cdot)$ be an algebra and $(V, l, r)$ be an $A$-bimodule.
 Let $T\colon V \to A$ be a linear map which is identified as an element in $(A \ltimes_{r^*, l^*} V^*) \otimes(A \ltimes_{r^*, l^*} V^*)$
 (through $\Hom(V, A) \cong A \otimes V^* \subset (A \ltimes_{r^*, l^*} V^*) \otimes (A \ltimes_{r^*, l^*} V^*)$).
 Then $r = T - \sigma(T)$ is an antisymmetric solution of AYBE in the algebra $(A \ltimes_{r^*, l^*}V^*, \cdot)$ if and only if $T$ is an $\mathcal{O}$-operator of $(A, \cdot)$ associated to $(V, l, r)$.
\end{Proposition}

We next generalize the above construction to the context of differential algebras,
showing that $\mathcal{O}$-operators of differential algebras give antisymmetric solutions of admissible AYBE in semi-direct product differential algebras and hence give rise to differential ASI bialgebras.

\begin{Proposition}\label{pro:admsemi}
 Let $(A, \cdot, \Phi = \{\partial_k\}_{k=1}^m)$ be a differential algebra and
 $(V, l, r)$ be an $A$-bimodule.
 Let $\Psi = \{\eth_k\colon A \to A \}_{k=1}^m$, $\Omega = \{\alpha_k\colon V \to V\}_{k=1}^m$ and
 $\Pi = \{\beta_k\colon V \to V\}_{k=1}^m$ be sets of commuting linear maps.
 Then the following conditions are equivalent.
 \begin{enumerate}\itemsep=0pt
 \item[$1.$] There is a differential algebra $(A \ltimes_{l,r} V, \cdot, \Phi + \Omega)$ such that $\Psi + \Pi$ is admissible to $(A \ltimes_{l,r} V, \cdot, \Phi + \Omega)$.

 \item[$2.$] There is a differential algebra $(A \ltimes_{r^*, l^*} V^*, \cdot, \Phi + \Pi^*)$ such that $\Psi + \Omega^*$ is admissible to $(A \ltimes_{r^*,l^*} V^*, \Phi + \Pi^*)$.

 \item[$3.$] The following conditions are satisfied:
 \begin{enumerate}\itemsep=0pt
 \item[$(i)$] $(V, l, r, \Omega)$ is a bimodule of $(A, \cdot, \Phi)$;
 \item[$(ii)$] $(A, \cdot, \Phi)$ is $\Psi$-admissible;
 \item[$(iii)$] $\Pi$ is admissible to $(A, \cdot, \Phi)$ on $(V, l, r)$;
 \item[$(iv)$] for all $a \in A, v \in V$ and $k=1, \dots, m$, the following equations hold:
 \begin{gather}
 l(\eth_k(a))v = l(a) \alpha_k(v) + \beta_k( l(a)v ), \label{eq:admsemi1} \\
 r(\eth_k(a))v = r(a) \alpha_k(v) + \beta_k( r(a)v ). \label{eq:admsemi2}
 \end{gather}
 \end{enumerate}
 \end{enumerate}
\end{Proposition}
\begin{proof}
 (1) $\iff$ (3).
 By Proposition~\ref{pro:semidirect}, $(A \ltimes_{l,r} V, \cdot, \Phi + \Omega)$ is a differential algebra if and only if
 $(V, l, r, \Omega)$ is a bimodule of the differential algebra $(A, \cdot, \Phi)$.
 Let $a, b \in A$, $u, v \in V$ and $k=1, \dots, m$. Then
 \begin{gather*}
 ((\eth_k + \beta_k)(a + u)) \cdot (b + v) = \eth_k(a) \cdot b + r(b)\beta_k(u) + l(\eth_k(a)) v, \\
 (a + u) \cdot ((\partial_k + \alpha_k)(b + v)) = a \cdot \partial_k(b) + l(a)\alpha_k(v) + r(\partial_k(b)) u, \\
 (\eth_k + \beta_k)((a + u)\cdot (b + v)) = \eth_k(a \cdot b) + \beta_k(l(a)v) + \beta_k(r(b)u).
 \end{gather*}
 Therefore equation~\eqref{eq:qadm1} holds (where $\eth_k$ is replaced by $\eth_k + \beta_k$, $\partial_k$ by $\partial_k + \alpha_k$, $a$ by $a + u$ and $b$ by $b + v$)
 if and only if equation~\eqref{eq:qadm1} (corresponding to $u = v = 0$), equation~\eqref{eq:admrep1}, where $a$ is replaced by $b$ and $v$ by $u$, (corresponding to $a = v = 0$) and equation~\eqref{eq:admsemi1} (corresponding to $b = u =0$) hold.
 Similarly, equation~\eqref{eq:qadm2} holds (where $\eth_k$ is replaced by $\eth_k + \beta_k$, $\partial_k$ by $\partial_k + \alpha_k$, $a$ by $a + u$ and $b$ by $b + v$)
 if and only if equation~\eqref{eq:qadm2}, equation~\eqref{eq:admrep2} and equation~\eqref{eq:admsemi2} hold.

 (2) $\iff$ (3).
 In item~(1), take $V = V^*$, $l = r^*$, $r = l^*$, $\Pi = \Omega^*$, $\Omega = \Pi^*$.
 Then from the above equivalence between item~(1) and item~(3),
 we show that item~(2) holds if and only if the conditions (i)--(iii) in item~(3)
 as well as the following equations hold (for all $a \in A$, $v^* \in V^*$, $k=1, \dots, m$):
 \begin{gather}
 r^*(\eth_k(a)) v^* = r^*(a) \beta_k^*(v^*) + \alpha_k^* (r^*(a)v^*), \label{eq:pf_admsemi1} \\
 l^*(\eth_k(a)) v^* = l^*(a) \beta_k^*(v^*) + \alpha_k^* (l^*(a)v^*). \label{eq:pf_admsemi2}
 \end{gather}
	Let $a \in A$. Then we have
	\begin{equation*}
		r^*(a) \beta_k^* + \alpha_k^* r^*(a) = ( \beta_k r(a) + r(a) \alpha_k)^*.
	\end{equation*}
	Hence equation~\eqref{eq:pf_admsemi1} holds if and only if equation~\eqref{eq:admsemi2} holds.
	Similarly, equation~\eqref{eq:pf_admsemi2} holds if and only if equation~\eqref{eq:admsemi1} holds.
	Therefore item~(2) holds if and only if item~(3) holds.
\end{proof}

\begin{Theorem}\label{thm:aybesemi}
 Let $(A, \cdot, \Phi = \{\partial_k\}_{k=1}^m)$ be a differential algebra.
 Suppose that $\Pi = \{\beta_k\colon V \to V\}_{k=1}^m$ is admissible to $(A, \cdot, \Phi)$ on $(V, l, r)$ and
 hence $(V^*, r^*, l^*, \Pi^*)$ is a bimodule of $(A, \cdot, \Phi)$.
 Let $\Psi = \{\eth_k\colon A \to A\}_{k=1}^m$ and $\Omega = \{\alpha_k\colon V \to V\}_{k=1}^m$ be sets of commuting linear maps.
 Let $T\colon V \to A$ be a linear map.
 \begin{enumerate}\itemsep=0pt
 \item[$1.$] The element $r = T - \sigma(T)$ is an antisymmetric solution of $(\Psi + \Omega^*)$-admissible AYBE in the differential algebra $(A \ltimes_{r^*, l^*}V^*, \cdot, \Phi + \Pi^*)$
 if and only if $T$ is an $\mathcal{O}$-operator of $(A, \cdot)$ associated to $(V, l, r)$ such that $\partial_k T = T \alpha_k$ and $T \beta_k = \eth_k T$ for all $k=1,\dots,m$.

 \item[$2.$] Assume that $(V, l, r, \Omega)$ is a bimodule of $(A, \cdot, \Phi)$.
 If $T$ is an $\mathcal{O}$-operator of $(A, \cdot, \Phi)$ associated to $(V, l, r,\Omega)$ and $T \beta_k = \eth_k T$ for all $k=1, \dots, m$,
 then $r = T - \sigma(T)$ is an antisymmetric solution of $(\Psi + \Omega^*)$-admissible AYBE in the differential algebra $(A \ltimes_{r^*, l^*} V^*, \cdot, \Phi + \Pi^*)$.
 If in addition, $(A, \cdot, \Phi)$ is $\Psi$-admissible and equations~\eqref{eq:admsemi1}--\eqref{eq:admsemi2} are satisfied,
 then the differential algebra $(A \ltimes_{r^*, l^*} V^*, \cdot, \Phi + \Pi^*)$ is $(\Psi + \Omega^*)$-admissible.
 Therefore in this case,
 there is a differential ASI bialgebra $(A \ltimes_{r^*, l^*} V^*, \cdot, \Delta, \Phi + \Pi^*, \Psi + \Omega^*)$,
 where the linear map $\Delta$ is defined by
 equation~\eqref{eq:cbd} with $r=T-\sigma(T)$.
\end{enumerate}
\end{Theorem}
\begin{proof}
 (1)
 By Proposition~\ref{pro:O-cons}, $r$ satisfies equation~\eqref{eq:aybe} if and only if equation~\eqref{eq:oop1} holds.
 Let $\{v_1, v_2, \dots, v_n\}$ be a basis of $V$ and $\{v_1^*, v_2^*, \dots, v_n^*\}$ be the dual basis.
 Then $T = \sum_{i=1}^n T(v_i) \otimes v_i^* \in (A \ltimes_{r^*, l^*} V^*) \otimes(A \ltimes_{r^*, l^*} V^*)$ and
 $r = T - \sigma(T) = \sum_{i=1}^n (T(v_i) \otimes v_i^* - v_i^* \otimes T(v_i))$.
 Note that for all $k=1, \dots, m$, we have
 \begin{gather*}
 ((\partial_k + \beta_k^*) \otimes \id)(r) = \sum_{i=1}^n (\partial_k T(v_i) \otimes v_i^* - \beta_k^*(v_i^*) \otimes T(v_i)),\\
 (\id \otimes(\eth_k + \alpha_k^*))(r) = \sum_{i=1}^n (T(v_i)\otimes \alpha_k^*(v_i^*) - v_i^* \otimes \eth_k T(v_i)).
 \end{gather*}
 Moreover, we have
 \begin{gather*}
 \sum_{i=1}^n \beta_k^*(v_i^*) \otimes T(v_i) = \sum_{i=1}^n\left(\sum_{j=1}^n \langle \beta_k^*(v_i^*), v_j \rangle v_j^* \otimes T(v_i)\right) = \sum_{j=1}^n v_j^* \otimes \sum_{i=1}^n
 \langle v_i^*, \beta_k(v_j) \rangle T(v_i) \\
 \hphantom{\sum_{i=1}^n \beta_k^*(v_i^*) \otimes T(v_i}{}
 =\sum_{i=1}^n v_i^* \otimes T\left(\sum_{j=1}^n \langle \beta_k(v_i), v_j^* \rangle v_j\right) = \sum_{i=1}^n v_i^* \otimes T \beta_k(v_i).
 \end{gather*}
 And similarly, $\sum_{i=1}^n T(v_i) \otimes \alpha_k^*(v_i^*) = \sum_{i=1}^n T \alpha_k(v_i)\otimes v_i^*$.
 Therefore $((\partial_k + \beta_k^*)\otimes \id)(r) = (\id \otimes (\eth_k + \alpha_k^*))(r)$ if and only if
 $\partial_k T = T \alpha_k$ and $\eth_k T = T \beta_k$ for all $k=1, \dots, m$.
 Hence item~(1) holds.

(2) It follows from item~(1), Proposition~\ref{pro:admsemi} and Corollary~\ref{cor:admAYBEdasi}.
\end{proof}

\begin{Corollary}\label{cor:qminp}
 Let $(A, \cdot, \Phi = \{\partial_k\}_{k=1}^m)$ be a differential algebra and
 $(V, l, r, \Omega = \{\alpha_k\}_{k=1}^m)$ be a bimodule of $(A, \cdot, \Phi)$.
 Let $(\theta_1, \dots, \theta_m) \in \mathbb{F}^m$ be given.
 Let $T\colon V \to A$ be an $\mathcal{O}$-operator of $(A, \cdot, \Phi)$ associated to $(V, l, r,
 \Omega)$.
 Then $r = T - \sigma(T)$ is an antisymmetric solution of $(\{-\partial_k + \theta_k \id_A\}_{k=1}^m + \Omega^*)$-admissible AYBE
 in the
 $(\{-\partial_k + \theta_k \id_A\}_{k=1}^m + \Omega^*)$-admissible
 differential algebra $(A \ltimes_{r^*, l^*}V^*, \cdot ,\Phi + \{-\alpha_k^* + \theta_k \id_{V^*}\}_{k=1}^m)$.
 Therefore there is a differential ASI bialgebra
 $(A \ltimes_{r^*, l^*}V^*,\allowbreak \cdot, \Delta, \Phi + \{-\alpha_k^* + \theta_k \id_{V^*}\}_{k=1}^m, \{-\partial_k + \theta_k \id_A\}_{k=1}^m + \Omega^*)$,
 where the linear map $\Delta$ is defined by equation~\eqref{eq:cbd} with $r = T - \sigma(T)$.
\end{Corollary}
\begin{proof}
 Take $\Psi = \{-\partial_k + \theta_k \id_A\}_{k=1}^m$ and $\Pi = \{-\alpha_k + \theta_k \id_V\}_{k=1}^m$ in Theorem~\ref{thm:aybesemi}\,(2).
 By Example~\ref{ex:admissible},
 $\Pi$ is admissible to the differential algebra $(A,\cdot,\Phi)$ on $(V,l,r)$ and $(A,\cdot, \Phi)$ is $\Psi$-admissible.
 Obviously, we have
 \begin{equation*}
 T (-\alpha_k + \theta_k \id_V) = (-\partial_k + \theta_k \id_A) T, \qquad \forall k=1, \dots, m,
 \end{equation*}
 for $T$ being an $\mathcal{O}$-operator of $(A, \cdot, \Phi)$ associated to $(V, l, r, \Omega)$.
 Moreover, equations~\eqref{eq:admsemi1}--\eqref{eq:admsemi2} hold, where $\eth_k=-\partial_k + \theta_k \id_A, \beta_k=-\alpha_k + \theta_k \id_V$ for all $k=1,\dots, m$.
 Hence the conclusion follows from Theorem~\ref{thm:aybesemi}\,(2).
\end{proof}

\begin{Corollary}\label{cor:idoop}
 Let $(A, \cdot, \Phi = \{\partial_k\}_{k=1}^m)$ be a differential algebra. Let $\{e_1, e_2, \dots, e_n\}$ be a~basis of $A$ and $\{e_1^*,e_2^*, \dots, e_n^*\}$ be the dual basis.
 Let $(\theta_1, \dots, \theta_m) \in \mathbb{F}^m$ be given.
 Then $r = \sum_{i=1}^n (e_i \otimes e_i^* - e_i^* \otimes e_i)$ is an antisymmetric solution of
 $(\{-\partial_k + \theta_k \id_A\}_{k=1}^m + \Phi^*)$-admissible AYBE in the
 $(\{-\partial_k + \theta_k \id_A\}_{k=1}^m +
 \Phi^*)$-admissible
 differential algebra $(A \ltimes_{R_A^*, 0}A^*, \cdot ,\Phi + \{-\partial_k^* + \theta_k \id_{A^*}\}_{k=1}^m)$ or $(A \ltimes_{0, L_A^*}A^*, \cdot ,\Phi + \{-\partial_k^* + \theta_k \id_{A^*}\}_{k=1}^m)$.
 Hence there are differential ASI bialgebras
 $(A \ltimes_{R_A^*, 0}A^*, \cdot, \Delta, \Phi + \{-\partial_k^* + \theta_k \id_{A^*}\}_{k=1}^m, \{-\partial_k + \theta_k \id_A\}_{k=1}^m + \Phi^*)$ and
 $(A \ltimes_{0, L_A^*}A^*, \cdot, \Delta, \Phi + \{-\partial_k^* + \theta_k \id_{A^*}\}_{k=1}^m, \{-\partial_k + \theta_k \id_A\}_{k=1}^m + \Phi^*)$,
 where the linear map $\Delta$ is defined by equation~\eqref{eq:cbd}
 with the above $r$.
\end{Corollary}
\begin{proof}
 By Example~\ref{example:o_operator},
 the identity map $\id\colon A \to A$ is an $\mathcal{O}$-operator of $(A, \cdot, \Phi)$ associated to $(A, L_A, 0,\Phi)$ or $(A, 0, R_A, \Phi)$.
 Note that $\id = \sum_{i=1}^n e_i \otimes e_i^*$.
 Hence the conclusion follows from Corollary~\ref{cor:qminp}.
\end{proof}

\subsection{Differential dendriform algebras}\label{subsec:dendriform}

We recall the notion of dendriform algebras.

\begin{Definition}[\cite{Loday2001Dialgebras}]
 Let $A$ be a vector space with two bilinear multiplications denoted by~$\succ$ and~$\prec$ respectively.
 The triple $(A, \succ, \prec)$ is called a \textit{dendriform algebra} if
 \begin{gather*}
 (a \prec b) \prec c = a \prec (b \prec c +b \succ c), \qquad
 (a \succ b) \prec c = a \succ (b \prec c), \\
 (a \prec b + a \succ b) \succ c = a \succ (b \succ c)
 \end{gather*}
	for all $a, b, c \in A$.
\end{Definition}

For a dendriform algebra $(A, \succ, \prec)$, define two linear
maps $L_\succ, R_\prec\colon A \to \End(A)$ respectively by
\begin{equation*}
	L_\succ (a)b = a \succ b, \qquad R_\prec(a)b = b\prec a,\qquad \forall a, b\in A.
\end{equation*}

\begin{Proposition}[\cite{BaiDouble,Loday2001Dialgebras}]
 Let $(A, \succ, \prec)$ be a dendriform algebra.
 Then the bilinear multiplication
 \begin{equation}
 a \cdot b := a \succ b + a \prec b, \qquad \forall a, b \in A, \label{eq:muldend}
 \end{equation}
 defines an algebra $(A, \cdot)$, called the \textit{associated algebra} of $(A, \succ, \prec)$.
 Moreover $(A, L_\succ, R_\prec)$ is an $A$-bimodule and
 the identity map $\id$ is an $\mathcal{O}$-operator of $(A, \cdot)$ associated to $(A, L_\succ, R_\prec)$.
\end{Proposition}

\begin{Definition}
	A \textit{derivation} on a dendriform algebra $(A, \succ, \prec)$ is a linear map $\partial\colon A\rightarrow A$ satisfying the following equations:
 \begin{gather}
 \partial(a \succ b) = \partial(a) \succ b + a \succ \partial(b), \label{eq:derdend1} \\
 \partial(a \prec b) = \partial(a) \prec b + a \prec \partial(b), \qquad \forall a, b \in A. \label{eq:derdend2}
 \end{gather}
 $(A, \succ, \prec, \Phi)$ is called a \textit{differential dendriform algebra} if
 $(A, \succ, \prec)$ is a dendriform algebra and $\Phi = \{\partial_k\colon A \to A\}_{k=1}^m$ is a finite set of commuting derivations on $(A, \succ, \prec)$.
\end{Definition}

\begin{Proposition}\label{pro:assda}
 Let $(A, \succ, \prec, \Phi)$ be a differential dendriform algebra.
 Then with the bilinear multiplication $\cdot$ defined by equation~\eqref{eq:muldend},
 $(A, \cdot, \Phi)$ is a differential algebra,
 called the \textit{associated differential algebra of $(A, \succ, \prec, \Phi)$}.
 Moreover, $(A, L_\succ, R_\prec, \Phi)$ is a bimodule of the associated differential algebra $(A, \cdot, \Phi)$
 and the identity map $\id\colon A \to A$ is an $\mathcal{O}$-operator of $(A, \cdot, \Phi)$ associated to $(A, L_\succ, R_\prec, \Phi)$.
\end{Proposition}
\begin{proof}
	Set $\Phi = \{\partial_k\}_{k=1}^m$.
 It is straightforward that $\partial_k$ is a derivation on $(A, \cdot)$ for all $k=1, \dots, m$.
 Moreover, equations~\eqref{eq:repd1} and \eqref{eq:repd2} hold,
 where $l = L_\succ$, $r = R_\prec$,
 if and only if equations~\eqref{eq:derdend1} and \eqref{eq:derdend2} hold respectively.
 That is, $(A, L_\succ, R_\prec, \Phi)$ is a bimodule of the differential algebra $(A, \cdot, \Phi)$
 and the last conclusion follows immediately.
\end{proof}

Recall \cite{aguiar2000pre,bai2013splitting} that a Rota--Baxter operator $R$ on an algebra $(A,\cdot)$ gives a dendriform algebra $(A,\succ,\prec)$, where
\begin{equation}
	a \succ b = R(a) \cdot b, \qquad a\prec b = a\cdot R(b), \qquad \forall a, b \in A. \label{eq:dd}
\end{equation}

It is straightforward to get the following conclusion.
\begin{Lemma}\label{lem:ff}
Let $(A, \cdot, \Phi = \{\partial_k\}_{k=1}^m)$ be a differential
algebra and $R$ be a Rota--Baxter operator on the algebra
$(A,\cdot)$. If $R\partial_k=\partial_kR$ for all $k=1,\dots, m$, then
$(A,\succ,\prec, \Phi)$ is a differential dendriform algebra,
where $\succ$ and $\prec$ are respectively defined by equation~\eqref{eq:dd}.
\end{Lemma}

\begin{Proposition}\label{pro:opd} Let $T\colon V \to A$ be an $\mathcal{O}$-operator of a differential algebra $(A, \cdot, \Phi)$ associated~to a~bimodule $(V, l, r, \Omega)$.
 Then there exists a differential dendriform algebra structure $(V, \succ, \prec$, $\Omega)$ on $V$,
 where $\succ$ and $\prec$ are respectively defined by
 \begin{equation}
 u \succ v := l(T(u)) v \qquad \text{and} \qquad
 u \prec v := r(T(v)) u, \qquad
 \forall u, v \in V. \label{eq:otd}
 \end{equation}
\end{Proposition}
\begin{proof}
	By Lemma~\ref{lem:o2r}, the linear map $\hat T$ defined by equation~(\ref{eq:hatT}) is a Rota--Baxter operator on the algebra $(A \ltimes_{l, r} V, \cdot)$.
	Then there is a dendriform algebra structure on the direct sum $A\oplus V$ of vector spaces defined by
	\begin{gather*}
		(a+u) \succ (b+v) :=\hat T(a+u) \cdot (b+v)=T(u)\cdot (b+v)=T(u)\cdot b+l(T(u)) v, \\
		(a+u) \prec (b+v) := (a+u) \cdot \hat T(b+v)=(a+u)\cdot T(v)=a\cdot T(v)+ r(T(v)) u,
	\end{gather*}
	for all $a, b \in A, u, v \in V$.
	Set $\Phi = \{\partial_k\}_{k=1}^m$ and $\Omega = \{\alpha_k\}_{k=1}^m$.
	By equation~\eqref{eq:oop2}, $\hat T$ commutes with $\partial_k + \alpha_k$ for all $k=1, \dots, m$.
	Hence by Lemma~\ref{lem:ff}, $(A\oplus V, \succ, \prec, \Phi + \Omega)$ is a differential dendriform algebra.
	In particular, on the vector space $V$, there is a differential dendriform subalgebra $(V,\succ,\prec, \Omega)$,
	in which $\succ$ and $\prec$ are exactly defined by equation~(\ref{eq:otd}).
\end{proof}

At the end of this section, we illustrate a construction of
antisymmetric solutions of admissible AYBE in differential
algebras and thus differential ASI bialgebras from differential
dendriform algebras.

\begin{Proposition}\label{pro:denddasi}
 Let $(A, \succ, \prec, \Phi = \{\partial_k\}_{k=1}^m)$ be a differential dendriform algebra and $(A, \cdot, \Phi)$ be the associated differential algebra.
 Let $\{e_1, e_2, \dots, e_n\}$ be a basis of $A$ and $\{e_1^*, e_2^*, \dots, e_n^*\}$ be the dual basis.
 Let $(\theta_1, \dots, \theta_m) \in \mathbb{F}^m$ be given.
 Then $r = \sum_{i=1}^n (e_i \otimes e_i^* - e_i^* \otimes e_i)$ is an antisymmetric solution of
 $(\{-\partial_k + \theta_k \id_A\}_{k=1}^m + \Phi^*)$-admissible AYBE in the
 $(\{-\partial_k + \theta_k \id_A\}_{k=1}^m + \Phi^*)$-admissible differential algebra
 $(A \ltimes_{R^*_{\prec}, L^*_{\succ}}A^*, \cdot, \Phi + \{-\partial_k^* + \theta_k \id_{A^*}\})$.
 Therefore there is a differential ASI bialgebra
 $(A \ltimes_{R^*_{\prec}, L^*_{\succ}}A^*, \cdot, \Delta, \Phi + \{-\partial_k^* + \theta_k \id_{A^*}\}_{k=1}^m, \{-\partial_k + \theta_k \id_A\}_{k=1}^m + \Phi^*)$,
 where the linear map $\Delta$ is defined by equation~\eqref{eq:cbd} with the above $r$.
\end{Proposition}
\begin{proof}
 By Proposition~\ref{pro:assda}, $(A, \cdot, \Phi)$ is a differential algebra and the identity map $\id$ is an $\mathcal{O}$-operator of $(A, \cdot, \Phi)$ associated to $(A,L_\succ, R_\prec, \Phi)$.
 Note that $\id = \sum^n_{i=1} e_i\otimes e_i^*$.
 Then the conclusion follows from Corollary~\ref{cor:qminp}.
\end{proof}

\begin{Remark}Corollary~\ref{cor:idoop} is a special case of Proposition~\ref{pro:denddasi},
 that is, the former corresponds to the trivial differential dendriform algebra structure
 $(A, \succ, \prec, \Phi)$ on a differential algebra $(A, \cdot, \Phi)$ given by $\succ = \cdot$, $\prec = 0$ or $\succ = 0$, $\prec = \cdot$.
\end{Remark}

\section[Poisson bialgebras via commutative and cocommutative differential ASI bialgebras]{Poisson bialgebras via commutative\\ and cocommutative differential ASI bialgebras}\label{sec:poisson}

We generalize the construction of Poisson algebras from commutative algebras with a pair of commuting derivations to the context of bialgebras,
that is, we construct Poisson bialgebras introduced in~\cite{ni2013poisson} from commutative and cocommutative differential ASI bialgebras.
We establish the explicit relationships between them, as well as the equivalent interpretation in terms of the corresponding double constructions (Manin triples) and matched pairs.
For the coboundary cases, the relationships between the involved structures on commutative differential algebras and Poisson algebras,
such as admissible AYBE and Poisson Yang--Baxter equation (PYBE), $\mathcal{O}$-operators of the two structures, and differential Zinbiel (commutative dendriform) algebras and pre-Poisson algebras, are also given.
In particular, we give a construction of Poisson bialgebras from differential Zinbiel algebras.

In this section, we always assume that $(A, \cdot)$ is a commutative algebra.
In this case, we use $(V, \mu)$ to denote an $A$-bimodule $(V,l,r)$,
where $\mu = l = r$, and call $(V, \mu)$ an \textit{$A$-module}.
In particular, $(A, L_A)$ is an $A$-module, and $(V^*, \mu^*)$ is again an $A$-module if $(V, \mu)$ is an $A$-module.
Further, we use $(A \ltimes_\mu V, \cdot)$ to denote the semi-direct product algebra by $(A, \cdot)$ and $(V, \mu)$.
Bimodules of commutative differential algebras are also called \textit{modules of commutative differential algebras}.

\begin{Definition}
 A \textit{Poisson algebra} is a triple $(A, [\ ,\ ], \cdot)$, where $(A, [\ ,\ ])$ is a Lie algebra and $(A, \cdot)$ is a commutative algebra satisfying the following equation:
 \begin{equation*}
 [a, b \cdot c] = [a, b] \cdot c + b \cdot [a, c], \qquad
 \forall a, b, c \in A.
 \end{equation*}
\end{Definition}

The following result is known (cf.~\cite{BV1}).

\begin{Proposition}
 Let $(A, \cdot, \Phi = \{\partial_1, \partial_2\})$ be a commutative differential algebra.
 Then $(A, [\ ,\ ], \cdot)$ is a Poisson algebra,
 called the \textit{induced Poisson algebra of $(A, \cdot, \Phi)$},
 where $[\ ,\ ]$ is defined by
 \begin{equation}
 [a, b] := \partial_1(a) \cdot \partial_2(b) - \partial_2(a) \cdot \partial_1(b), \qquad
 \forall a, b \in A. \label{eq:lie}
 \end{equation}
\end{Proposition}

Moreover, if $(A, \cdot, \Phi = \{\partial_1^k, \partial_2^k\}_{k=1}^m)$ is a commutative differential algebra,
that is, the number of the commuting derivations is even, then
\begin{equation*}
	[a, b] := \sum_{k=1}^m \partial^k_1(a) \partial^k_2(b) - \partial^k_2(a) \partial^k_1(b),\qquad \forall a, b \in A,
\end{equation*}
still defines a Lie algebra $(A,[\ ,\ ])$ such that $(A, [\ ,\ ], \cdot)$ is a Poisson algebra.
Therefore all the results of this section that are proved for one pair $\{\partial_1, \partial_2\}$ of derivations remain valid for several pairs,
whereas we illustrate these results in terms of one pair of derivations for simplicity.

\subsection[Poisson algebras via commutative differential algebras: modules and matched pairs]{Poisson algebras via commutative differential algebras:\\ modules and matched pairs}

\begin{Definition}[\cite{ni2013poisson}]
 Let $(A, [\ ,\ ], \cdot)$ be a Poisson algebra, $V$ be a vector space and $\rho, \mu\colon A \to \End (V)$ be two linear maps.
 Then $(V, \rho, \mu)$ is called a \textit{module of the Poisson algebra} $(A, [\ ,\ ], \cdot)$
 if $(V, \rho)$ is a module of the Lie algebra $(A, [\ ,\ ])$
 and $(V, \mu)$ is an $A$-module such that the following equations hold:
 \begin{equation*}
 \rho(a \cdot b) = \mu(b) \rho(a) + \mu(a) \rho(b), \qquad \mu([a,b])= \rho(a) \mu(b) - \mu(b) \rho(a), \qquad \forall a,b\in A.
 \end{equation*}
\end{Definition}

\begin{Lemma}\label{lem:rep_poisson}
 Let $(A, \cdot, \Phi = \{\partial_1, \partial_2\})$ be a commutative differential algebra and $(A, [\ ,\ ], \cdot)$ be the induced Poisson algebra of $(A, \cdot, \Phi)$.
 Suppose that $(V, \mu, \Omega = \{\alpha_1, \alpha_2\})$ is a module of $(A, \cdot, \Phi)$.
 Define $\rho_\mu\colon A \to \End (V)$ by
 \begin{equation}
 \rho_\mu(a) := \mu(\partial_1(a)) \alpha_2 - \mu(\partial_2(a))\alpha_1, \qquad
 \forall a \in A. \label{eq:def_rep_lie}
 \end{equation}
 Then $(V, \rho_\mu, \mu)$ is a module of the Poisson algebra $(A, [\ ,\ ], \cdot)$,
 called the \textit{induced module of $(A, [\ ,\ ], \cdot)$ with respect to $(V, \mu, \Omega)$}.
\end{Lemma}
\begin{proof}
 Let $a, b \in A$.
 Note that
 \begin{align*}
 	\rho_\mu(a) &\overset{\hphantom{(2.3)}}{=} \mu(\partial_1(a))\alpha_2 - \mu(\partial_2(a))\alpha_1 \\
 	&\overset{\eqref{eq:repd1}}{=} (\alpha_1 \mu(a) - \mu(a) \alpha_1)\alpha_2 - (\alpha_2 \mu(a) - \mu(a) \alpha_2)\alpha_1 = \alpha_1 \mu(a) \alpha_2 - \alpha_2 \mu(a) \alpha_1.
 \end{align*}
 Hence we have
 \begin{equation*}
 \begin{aligned}
 &\rho_\mu(a)\rho_\mu(b) \\
 &\overset{\hphantom{(2.3)}}{=} \alpha_1 \mu(a)\alpha_2 \alpha_1 \mu(b) \alpha_2 - \alpha_1 \mu(a)\alpha_2 \alpha_2 \mu(b) \alpha_1 - \alpha_2 \mu(a)\alpha_1 \alpha_1 \mu(b) \alpha_2 + \alpha_2 \mu(a)\alpha_1 \alpha_2 \mu(b) \alpha_1 \\
 &\overset{\eqref{eq:repd1}}{=} \alpha_1(\alpha_2 \mu(a) - \mu(\partial_2(a)))(\mu(\partial_1(b)) + \mu(b)\alpha_1) \alpha_2 \\
 &\qquad {}- \alpha_1(\alpha_2\mu(a)- \mu(\partial_2(a)))(\mu(\partial_2(b)) + \mu(b)\alpha_2)\alpha_1 \\
 & \qquad {}-\alpha_2(\alpha_1 \mu(a) - \mu(\partial_1(a)))(\mu(\partial_1(b)) + \mu(b)\alpha_1) \alpha_2 \\
 &\qquad {}+ \alpha_2(\alpha_1 \mu(a) - \mu(\partial_1(a)))(\mu(\partial_2(b)) + \mu(b)\alpha_2)\alpha_1 \\
 &\overset{\hphantom{(2.3)}}{=} \alpha_1 \mu(\partial_2(a) \cdot \partial_2(b)) \alpha_1 - \alpha_1 \mu(\partial_2(a) \cdot \partial_1(b)) \alpha_2 + \alpha_2 \mu(\partial_1(a) \cdot \partial_1(b)) \alpha_2 - \alpha_2 \mu(\partial_1(a) \cdot \partial_2(b)) \alpha_1.
 \end{aligned}
 \end{equation*}
 Therefore we have
 \begin{equation*}
 [\rho_\mu(a), \rho_\mu(b)] = \rho_\mu(a)\rho_\mu(b) - \rho_\mu(b)\rho_\mu(a) = \alpha_1 \mu([a,b]) \alpha_2 - \alpha_2 \mu([a,b]) \alpha_1 = \rho_\mu([a,b]).
 \end{equation*}
 That is, $(V, \rho_\mu)$ is a module of the Lie algebra $(A, [\ ,\ ])$.
 Similarly, we show that
 \begin{equation*}
 \rho_\mu(a)\mu(b) - \mu(b)\rho_\mu(a) = \mu([a,b]), \qquad \mu(b)\rho_\mu(a) + \mu(a) \rho_\mu(b) = \rho_\mu(a \cdot b).
 \end{equation*}
 Therefore $(V, \rho_\mu, \mu)$ is a module of $(A, [\ ,\ ], \cdot)$.
\end{proof}

\begin{Example} Let $(A, \cdot, \Phi = \{\partial_1, \partial_2\})$ be a commutative differential algebra and $(A, [\ ,\ ], \cdot)$ be the induced Poisson algebra of $(A, \cdot, \Phi)$.
 The induced module $(A, \rho_{L_A}, L_A)$ of $(A, [\ ,\ ], \cdot)$
 with respect to the module $(A, L_A, \Phi)$ of $(A, \cdot, \Phi)$
 is exactly the module $(A, \ad_A, L_A)$ of $(A, [\ ,\ ], \cdot)$,
 where $\ad_A\colon A \to \End(A)$ is defined by $\ad_A(a)(b) = [a, b]$ for all $a, b \in A$.
\end{Example}

\begin{Lemma}[{\cite[Proposition~2.5]{liu2020noncommutative}}]
 Let $(A, [\ ,\ ], \cdot)$ be a Poisson algebra and $(V, \rho, \mu)$ be a module of $(A, [\ ,\ ], \cdot)$.
 Then $(V^*, -\rho^*, \mu^*)$ is a module of $(A, [\ ,\ ], \cdot)$.
\end{Lemma}

\begin{Proposition}\label{pro:assdualrep}
 Let $(A, \cdot, \Phi = \{\partial_1, \partial_2\})$ be a commutative differential algebra and $(A, [\ ,\ ], \cdot)$ be the induced Poisson algebra of $(A, \cdot, \Phi)$.
 Suppose that $(V, \mu, \Omega = \{\alpha_1, \alpha_2\} )$ is a module of $(A, \cdot, \Phi)$ and
 $\Pi = \{\beta_1, \beta_2\}$ is admissible to $(A, \cdot, \Phi)$ on $(V, \mu)$.
 Then the induced module $(V^*, \rho_{\mu^*}, \mu^*)$ of the Poisson algebra $(A, [\ ,\ ], \cdot)$ with respect to $(V^*, \mu^*, \Pi^*)$ is exactly the module $(V^*, -\rho^*_\mu, \mu^*)$ of $(A, [\ ,\ ], \cdot)$,
 where $(V, \rho_\mu, \mu)$ is the induced module of $(A, [\ ,\ ], \cdot)$ with respect to $(V, \mu, \Omega)$,
 if and only if the following equation holds:
 \begin{gather}
 \beta_2(\mu(\partial_1(a))v) - \beta_1(\mu(\partial_2(a))v) + \mu(\partial_1(a))\alpha_2(v) - \mu(\partial_2(a))\alpha_1(v) = 0, \label{eq:dra}
 \end{gather}
for all $a \in A$, $v \in V$. In particular, when taking $(V, \mu, \Omega) = (A, L_A, \Phi)$,
 the induced module $(A^*, \rho_{L_A^*},\allowbreak L_A^*)$ of $(A, [\ ,\ ], \cdot)$ with respect to $(A^*, L_A^*, \Pi^*)$
 is exactly the module $(A^*, -\ad_A^*, L_A^*)$ of $(A, [\ ,\ ], \cdot)$
 if and only if the following equation holds:
 \begin{equation}
 \beta_2(\partial_1(a)) \cdot b = \beta_1(\partial_2(a)) \cdot b, \qquad \forall a, b \in A. \label{eq:coadjoint}
 \end{equation}
\end{Proposition}
\begin{proof}
 For all $a \in A$, $v \in V$, $v^* \in V^*$, we have
 \begin{gather*}
\langle \rho_{\mu^*}(a) v^*, v \rangle = \langle \big(\mu^*(\partial_1(a)) \beta_2^* - \mu^*(\partial_2(a)) \beta_1^*\big)(v^*), v \rangle \\
\hphantom{\langle \rho_{\mu^*}(a) v^*, v \rangle}{}
= \langle v^*, \beta_2(\mu(\partial_1(a))v) - \beta_1(\mu(\partial_2(a))v)\rangle, \\
\langle -\rho^*_{\mu}(a) v^*, v \rangle = \langle v^*, -\rho_{\mu}(a) v \rangle = \langle v^*, -\mu(\partial_1(a))\alpha_2(v) + \mu(\partial_2(a))\alpha_1(v) \rangle.
 \end{gather*}
 Therefore $\rho_{\mu^*} = -\rho^*_{\mu}$ if and only if equation~\eqref{eq:dra} holds.
 For the particular case, we have
 \begin{gather*}
 \beta_2(\partial_1(a) \cdot b) - \beta_1(\partial_2(a) \cdot b) + \partial_1(a) \cdot \partial_2(b) - \partial_2(a) \cdot \partial_1(b)\\
 \qquad{} \overset{\eqref{eq:qadm1}}{=} \beta_2(\partial_1(a)) \cdot b - \beta_1(\partial_2(a)) \cdot b, \qquad \forall a,b \in A.
 \end{gather*}
 Note the induced module of $(A, [\ ,\ ], \cdot)$ with respect to $(A, L_A, \Phi)$ is $(A, \ad_A, L_A)$.
 We conclude $\rho_{L_A^*} = -\ad_A^*$ if and only if equation~\eqref{eq:coadjoint} holds.
\end{proof}

\begin{Definition}\label{def:matpoi}
 A \textit{matched pair of Poisson algebras} consists of Poisson algebras $(A, [\ ,\ ]_A$, $\cdot_A)$ and $(B, [\ ,\ ]_B,\cdot_B)$, together with linear maps $\rho_A, \mu_A\colon A \to \End(B)$ and $\rho_B, \mu_B\colon B \to \End(A)$ such that $(A \oplus B, [\ ,\ ], \cdot)$ is a Poisson algebra, where $[\ ,\ ]$ and $\cdot$ are respectively defined by
 \begin{gather*}
 	[a + b, a^\prime + b^\prime] := ([a, a^\prime]_A + \rho_B(b)a^\prime - \rho_B(b^\prime) a ) + ([b, b^\prime]_B + \rho_A(a)b^\prime - \rho_A(a^\prime)b ), \\
 	(a + b) \cdot (a^\prime + b^\prime) := (a \cdot_A a^\prime + \mu_B(b)a^\prime + \mu_B(b^\prime)a) + (b \cdot_B b^\prime + \mu_A(a) b^\prime + \mu_A(a^\prime)b ),
 \end{gather*}
 for all $a,a^\prime\in A$ and $b,b^\prime\in B$.
 The matched pair of Poisson algebras is denoted by $((A, [\ ,\ ]_A, \cdot_A)$, $(B, [\ ,\ ]_B$, $\cdot_B)$, $\rho_A, \mu_A, \rho_B, \mu_B)$ and the resulting Poisson algebra $(A \oplus B, [\ ,\ ], \cdot)$ is denoted by $\big(A \bowtie_{\rho_A, \mu_A}^{\rho_B, \mu_B} B, [\ ,\ ], \cdot\big)$ or simply $(A \bowtie B, [\ ,\ ], \cdot)$.
\end{Definition}

Note that such a notion of a matched pair of Poisson algebras is equivalent to the one given in~\cite[Theorem 1]{ni2013poisson}.
Moreover, for a matched pair of Poisson algebras $((A, [\ ,\ ]_A, \cdot_A), (B, [\ ,\ ]_B$, $\cdot_B)$, $\rho_A$, $\mu_A, \rho_B, \mu_B)$, $(A,\rho_B,\mu_B)$ is a module of
$(B, [\ ,\ ]_B, \cdot_B)$ and $(B,\rho_A,\mu_A)$ is a module of $(A, [\ ,\ ]_A, \cdot_A)$. In particular, for the case that $B=V$ which is a vector
space equipped with the zero multiplication, we have the following conclusion.

\begin{Lemma}[\cite{ni2013poisson}]
 Let $(A, [\ ,\ ], \cdot)$ be a Poisson algebra and $(V, \rho,\mu)$ be a module of $(A, [\ ,\ ],
 \cdot)$.
 Define two bilinear multiplications still denoted by $[\ ,\ ]$ and
 $\cdot$ on $A\oplus V$ respectively by
	\begin{gather*}
 [(a+u), (b+v)] := [a,b] + (\rho(a)v - \rho(b)u),\\
 (a+u) \cdot (b+v) := a \cdot b + (\mu(a)v + \mu(b)u), \qquad
 \forall a, b \in A, u,v \in V.
 \end{gather*}
 Then $(A \oplus V, [\ ,\ ], \cdot)$ is a Poisson algebra, which is denoted by $(A \ltimes_{\rho, \mu} V, [\ ,\ ], \cdot)$ and called the \textit{semi-direct product Poisson algebra} by $(A, [\ ,\ ], \cdot)$ and $(V, \rho, \mu)$.
\end{Lemma}

\begin{Proposition}\label{pro:assmat}
 Let $((A, \cdot_A, \Phi_A = \{\partial_{A, 1}, \partial_{A, 2}\}), (B, \cdot_B, \Phi_B= \{\partial_{B,1},\partial_{B,2}\}), \mu_A,\mu_B)$ be a~mat\-ched pair of commutative differential algebras.
 Let $(A, [\ ,\ ]_A, \cdot_A)$ and $(B, [\ ,\ ]_B, \cdot_B)$ be the induced Poisson algebras of $(A, \cdot_A, \Phi_A)$ and $(B, \cdot_B, \Phi_B)$ respectively.
 Then $((A, [\ ,\ ]_A, \cdot_A), (B, [\ ,\ ]_B, \cdot_B), \allowbreak \rho_{\mu_A}, \mu_A, \rho_{\mu_B}, \mu_B)$ is a matched pair of Poisson algebras,
 called the induced matched pair of Poisson algebras with respect to $((A, \cdot_A, \Phi_A), (B, \cdot_B, \Phi_B),\mu_A,\mu_B)$,
 where $(B, \rho_{\mu_A}, \mu_A)$ is the induced module of $(A, [\ ,\ ]_A, \cdot_A)$ with respect to $(B, \mu_A,\Phi_B)$ and
 $(A, \rho_{\mu_B}, \mu_B)$ is the induced module of $(B, [\ ,\ ]_B, \cdot_B)$ with respect to $(A, \mu_B, \Phi_A)$.
 Moreover, the Poisson algebra $(A \bowtie B, [\ ,\ ], \cdot)$ obtained in Definition~{\rm \ref{def:matpoi}} by the matched pair of Poisson algebras
 $((A, [\ ,\ ]_A, \cdot_A), (B, [\ ,\ ]_B, \cdot_B), \rho_{\mu_A}, \mu_A,\allowbreak \rho_{\mu_B}, \mu_B)$ is exactly
 the induced Poisson algebra of the commutative differential algebra $(A \bowtie B,\allowbreak \star, \Phi_A + \Phi_B)$ obtained in Theorem~{\rm \ref{thm:matda}} by the matched pair of commutative differential algebras $((A, \cdot_A, \Phi_A), (B, \cdot_B, \Phi_B), \mu_A, \mu_B)$.
\end{Proposition}
\begin{proof}
 Let $(A \bowtie B, \star, \Phi_A + \Phi_B)$ be the commutative differential algebra obtained in Theorem~\ref{thm:matda} by the matched pair of commutative differential algebras $((A, \cdot_A, \Phi_A), (B, \cdot_B, \Phi_B), \mu_A, \allowbreak \mu_B)$.
 Suppose that $(A \oplus B, \{\ ,\ \}, \star)$ is the induced Poisson algebra of $(A \bowtie B, \star, \Phi_A + \Phi_B)$.
 Therefore for all $a, a^\prime \in A$ and $b, b^\prime \in B$, we have
 \begin{gather*}
 \{a + b, a^\prime + b^\prime\} \\
 = (\partial_{A,1}(a)+\partial_{B,1}(b)) \star (\partial_{A,2}(a^\prime)+\partial_{B,2}(b^\prime)) - (\partial_{A,2}(a)+\partial_{B,2}(b)) \star (\partial_{A,1}(a^\prime)+\partial_{B,1}(b^\prime)) \\
 = (\partial_{A,1}(a) \cdot_A \partial_{A,2}(a^\prime) + \mu_B(\partial_{B,2}(b^\prime))\partial_{A,1}(a) + \mu_B(\partial_{B,1}(b))\partial_{A,2}(a^\prime)) \\
 \quad + (\partial_{B,1}(b) \cdot_B \partial_{B,2}(b^\prime) + \mu_A(\partial_{A,1}(a))\partial_{B,2}(b^\prime) + \mu_A(\partial_{A,2}(a^\prime))\partial_{B,1}(b)) \\
 \quad - (\partial_{A,2}(a) \cdot_A \partial_{A,1}(a^\prime) + \mu_B(\partial_{B,1}(b^\prime))\partial_{A,2}(a) + \mu_B(\partial_{B,2}(b))\partial_{A,1}(a^\prime)) \\
 \quad - (\partial_{B,2}(b) \cdot_B \partial_{B,1}(b^\prime) + \mu_A(\partial_{A,2}(a))\partial_{B,1}(b^\prime) + \mu_A(\partial_{A,1}(a^\prime))\partial_{B,2}(b)) \\
 = ([a, a^\prime]_A - \rho_{\mu_B}(b^\prime)a + \rho_{\mu_B}(b)a^\prime) + ([b, b^\prime]_B + \rho_{\mu_A}(a)b^\prime - \rho_{\mu_A}(a^\prime)b).
 \end{gather*}
 Hence $((A, [\ ,\ ]_A, \cdot_A), (B, [\ ,\ ]_B, \cdot_B), \rho_{\mu_A}, \mu_A, \rho_{\mu_B}, \mu_B)$ is a matched pair of Poisson algebras.
 Note that by the above proof, we have already shown that the Poisson algebra structure on $A\oplus B$ obtained from this induced matched pair of
 Poisson algebras is exactly the induced Poisson algebra of the
 commutative differential algebra $(A \bowtie B, \star, \Phi_A + \Phi_B)$.
\end{proof}

When taking $B = V$ which is a vector space equipped with the zero multiplication in Proposition~\ref{pro:assmat}, we have the following conclusion.
\begin{Corollary}\label{cor:asssemi}
 Let $(A, \cdot, \Phi = \{\partial_1, \partial_2\})$ be a commutative differential algebra and $(A, [\ ,\ ], \cdot)$ be the induced Poisson algebra of $(A, \cdot, \Phi)$.
 	Let $(V, \mu, \Omega)$ be a module of $(A, \cdot, \Phi)$.
 	Suppose that $(A \ltimes_\mu V, \cdot, \Phi + \Omega)$ is the semi-direct product (commutative) differential algebra by $(A, \cdot, \Phi)$ and $(V, \mu, \Omega)$.
 Then the semi-direct product Poisson algebra by $(A, [\ ,\ ], \cdot)$ and $(V, \rho_\mu, \mu)$,
 where $(V, \rho_\mu, \mu)$ is the induced module of $(A, [\ ,\ ], \cdot)$ with respect to $(V, \mu, \Omega)$,
 is exactly the induced Poisson algebra of $(A \ltimes_\mu V, \cdot, \Phi + \Omega)$.
\end{Corollary}

\begin{Corollary}\label{cor:adjassmat}
 Let $(A, \cdot, \Phi = \{\partial_1, \partial_2\})$ be a commutative differential algebra.
 Suppose that there is a commutative differential algebra structure $(A^*, \circ, \Psi^*= \{\eth_1^*,\eth_2^*\})$ on $A^*$.
 Let $(A, [\ ,\ ]_A, \cdot)$ and $(A^*, [\ ,\ ]_{A^*}, \circ)$ be the induced Poisson algebras of $(A, \cdot, \Phi)$ and $(A^*, \circ, \Psi^*)$ respectively.
 Suppose that $((A, \cdot, \Phi), (A^*, \circ, \Psi^*)$, $L_{A}^*$, $L_{A^*}^*)$ is a matched pair of commutative differential algebras.
 Then $((A, [\ ,\ ]_A, \cdot), (A^*, [\ ,\ ]_{A^*}, \circ), -\ad_A^*, L_A^*$, $-\ad_{A^*}^*$, $L_{A^*}^*)$ is a matched pair of Poisson algebras
 such that it is the induced matched pair of Poisson algebras with respect to
 $((A, \cdot, \Phi)$, $(A^*, \circ, \Psi^*)$,$L_{A}^*,L_{A^*}^*)$
 if and only if the following equations hold:
 \begin{gather}
 \partial_2^*(\eth_1^*(a^*)) \circ b^* = \partial_1^*(\eth_2^*(a^*)) \circ b^* , \qquad
 \forall a^*, b^* \in A^*, \label{eq:bvip2}\\
\eth_2(\partial_1(a)) \cdot b = \eth_1(\partial_2(a)) \cdot b , \qquad
 \forall a, b \in A. \label{eq:vip1}
 \end{gather}
\end{Corollary}
\begin{proof}
 It follows from Propositions~\ref{pro:assmat} and \ref{pro:assdualrep}.
\end{proof}

\begin{Remark}\label{rmk:cc}
 Let $\Delta\colon A \to A \otimes A$ denote the linear dual of the multiplication $\circ\colon A^* \otimes A^* \to A^*$, we rewrite equation~\eqref{eq:bvip2} in terms of the comultiplication as
 \begin{equation}
 	(\eth_2\partial_1 \otimes \id)\Delta = (\eth_1\partial_2 \otimes \id)\Delta. \label{eq:vip2}
 \end{equation}
\end{Remark}

Given a matched pair of commutative differential algebras $((A, \cdot, \Phi)$, $(A^*, \circ, \Psi^*)$, $L_{A}^*,L_{A^*}^*)$,
there is a more general condition such that $((A, [\ ,\ ]_A, \cdot)$, $(A^*, [\ ,\ ]_{A^*}, \circ)$, $-\ad_A^*, L_A^*, -\ad_{A^*}^*, L_{A^*}^*)$ is a matched pair of Poisson algebras,
\textit{without} the requirement that $((A, [\ ,\ ]_A, \cdot)$, $(A^*, [\ ,\ ]_{A^*}, \circ)$, $-\ad_A^*$, $L_A^*$, $-\ad_{A^*}^*, L_{A^*}^*)$ is the induced matched pair of Poisson algebras with respect to $((A, \cdot, \Phi)$, $(A^*, \circ, \Psi^*)$,$L_{A}^*,L_{A^*}^*)$.
Explicitly,
\begin{Proposition}\label{pro:gmp}
 With the same assumptions in Corollary~{\rm \ref{cor:adjassmat}}.
 Let $\Delta\colon A \to A \otimes A$ denote the linear dual of the multiplication $\circ\colon A^* \otimes A^* \to A^*$.
 Then $((A, [\ ,\ ]_A, \cdot)$, $(A^*, [\ ,\ ]_{A^*}, \circ)$, $-\ad_A^*, L_A^*, -\ad_{A^*}^*, L_{A^*}^*)$ is a matched pair of Poisson algebras
 if and only if the following equations hold:
 \begin{align}
 &(\id \otimes L_A(b) - L_A(b) \otimes \id )(S \otimes \id) \Delta(a) + ( L_A(S(a)) \otimes \id) \Delta(b)=0, \label{eq:dpb1} \\
 &(\ad_A (a) \otimes S - S \otimes \ad_A (a) )\Delta(b) - (\ad_A (b) \otimes S - S \otimes \ad_A (b) )\Delta(a) \nonumber \\
 &\quad - (L_A(S(a)) \otimes S - S \otimes L_A(S(a)))\Delta(b) + (L_A(S(b)) \otimes S - S \otimes L_A(S(b)))\Delta(a)\nonumber\\
 &\quad + (L_A(S(a) \otimes \id + \id \otimes L_A(S(a)))(\eth_1 \otimes \eth_2 - \eth_2 \otimes \eth_1) )\Delta(b) \nonumber \\
 &\quad - (L_A(S(b) \otimes \id + \id \otimes L_A(S(b)))(\eth_1 \otimes \eth_2 - \eth_2 \otimes \eth_1) )\Delta(a)=0, \label{eq:dpb2}
 \end{align}
 for all $a,b \in A$, where $S:= \eth_2 \partial_1 - \eth_1 \partial_2$.
\end{Proposition}
\begin{proof}
	Note that $(A^*, -\ad_A^*, L_A^*)$ is a module of $(A, [\ ,\ ]_A, \cdot)$ and $(A, -\ad_{A^*}^*, L_{A^*}^*)$ is a module of $(A^*, [\ ,\ ]_{A^*}, \circ)$. By \cite[Theorem 1]{ni2013poisson}, it is sufficient to show equations~\eqref{eq:dpb1}--\eqref{eq:dpb2} hold if and only if the following equations hold:
 \begin{gather}
 -\ad_A^*(a)[a^*, b^*]_{A^*} + [\ad_A^*(a)a^*, b^*]_{A^*} + [a^*, \ad_A^*(a)b^*]_{A^*} \nonumber\\
 \qquad{}= - \ad_A^*(\ad_{A^*}^*(a^*)a) b^* + \ad_A^*(\ad_{A^*}^*(b^*)a) a^*, \label{eq:matlie1}\\
 -\ad_{A^*}^*(a^*)[a, b]_A + [\ad_{A^*}^*(a^*)a, b]_A + [a, \ad_{A^*}^*(a^*)b]_A \nonumber\\
\qquad{} = - \ad_{A^*}^*(\ad_A^*(a)a^*) b + \ad_{A^*}^*(\ad_A^*(b)a^*) a, \label{eq:matlie2} \\
 -\ad_{A^*}^*(a^*)(a \cdot b) - (-\ad_{A^*}^*(a^*) a) \cdot b - a \cdot (-\ad_{A^*}^*(a^*) b)\nonumber \\
 \qquad{} = - L_{A^*}^*(-\ad_A^*(a)a^*)b + L_{A^*}^*(\ad_A^*(b)a^*) a, \label{eq:matpoi1} \\
 L_{A^*}^*(-\ad_A^*(a)a^*)b + (\ad_{A^*}^*(a^*)a)\cdot b + L_{A^*}^*(a^*)([a, b]_A) \nonumber\\
\qquad{} = [a, L_{A^*}^*(a^*)b]_A + \ad_{A^*}^*(L_A^*(b)a^*)a, \label{eq:matpoi2} \\
 -\ad_A^*(a)(a^* \circ b^*) - (-\ad_A^*(a)a^*) \circ b^* - a^* \circ (-\ad_A^*(a)b^*) \nonumber\\
\qquad{} = - L_A^*(-\ad_{A^*}^*(a^*)a)b^* + L_A^*(\ad_{A^*}^*(b^*)a)a^*, \label{eq:matpoi3}\\
 L_A^*(-\ad_{A^*}^*(a^*)a)b^* + (\ad_A^*(a)a^*) \circ b^* + L_A^*(a)([a^*,b^*]_{A^*}) \nonumber \\
\qquad{} = [a^*, L_A^*(a)b^*]_{A^*} + \ad_A^*(L_{A^*}^*(b^*)a)a^*, \label{eq:matpoi4}
 \end{gather}
 for all $a, b \in A$ and $a^*, b^* \in A^*$. Let $(A^*,\rho_{L_A^*}, L_A^*)$ be the induced module of $(A, [\ ,\ ]_A, \cdot)$ with respect to $(A^*,L_A^*, \Psi^*)$ and $(A,\rho_{L_{A^*}^*}, L_{A^*}^*)$ be the induced module of $(A^*, [\ ,\ ]_{A^*}, \circ)$ with respect to $(A, L_{A^*}^*, \Phi)$.
 By the proof of Proposition~\ref{pro:assdualrep}, we have
 \begin{gather*}
 -\ad^*_A(a) = \rho_{L^*_A}(a) - L^*_A(S(a)),\qquad
 \forall a \in A; \\
 -\ad^*_{A^*}(a^*) = \rho_{L^*_{A^*}}(a^*) + L^*_{A^*}(S^*(a^*)), \qquad
 \forall a^* \in A^*.
 \end{gather*}
	Note that here we have $\partial_2^* \eth_1^* - \partial_1^* \eth_2^* = - (\eth_2 \partial_1 - \eth_1 \partial_2)^* = -S^*$.
 Let $a, b \in A$ and $a^*, b^* \in A^*$.
 Then we have
 \begin{gather*}
 -(\ad^*_{A^*}(a^*) a) \cdot b - a \cdot (\ad^*_{A^*}(a^*) b) + L^*_{A^*}(\ad^*_{A}(a)a^*)b + L^*_{A^*}(\ad^*_{A}(b)a^*) a + \ad^*_{A^*}(a^*)(a \cdot b) \\
 = (\rho_{L^*_{A^*}}(a^*) a) \cdot b + a \cdot (\rho_{L^*_{A^*}}(a^*) b) - L^*_{A^*}(\rho_{L^*_A}(a)a^*)b - L^*_{A^*}(\rho_{L^*_A}(b)a^*) a - \rho_{L^*_{A^*}}(a^*)(a \cdot b) \\
 \quad {}+(L^*_{A^*}(S^*(a^*)) a) \cdot b + a \cdot (L^*_{A^*}(S^*(a^*)) b) + L^*_{A^*}(L^*_A(S(a))a^*)b \\
 \quad {}+ L^*_{A^*}(L^*_A(S(b))a^*) a - L^*_{A^*}(S^*(a^*))(a \cdot b) \\
 = (L^*_{A^*}(S^*(a^*)) a) \cdot b + a \cdot (L^*_{A^*}(S^*(a^*)) b) + L^*_{A^*}(L^*_A(S(a))a^*)b \\
 \quad {}+ L^*_{A^*}(L^*_A(S(b))a^*) a - L^*_{A^*}(S^*(a^*))(a \cdot b) \\
 = (L^*_{A^*}(S^*(a^*)) a) \cdot b + L^*_{A^*}(L^*_A(S(a))a^*)b + L^*_{A^*}(L^*_A(S(b))a^*) a - L^*_{A^*}(L^*_A(b)S^*(a^*))a.
 \end{gather*}
 Hence, equation~\eqref{eq:matpoi1} holds if and only if the following equation holds:
 \begin{equation*}
 (L^*_{A^*}(S^*(a^*)) a) \cdot b + L^*_{A^*}(L^*_A(S(a))a^*)b + L^*_{A^*}(L^*_A(S(b))a^*) a - L^*_{A^*}(L^*_A(b)S^*(a^*))a = 0.
 \end{equation*}
 Rewrite the above equation in terms of the comultiplication as
 \begin{equation*}
 (S \otimes L_A(b) + L_A(S(b)) \otimes \id - S(L_A(b)) \otimes \id ) \Delta(a) + ( L_A(S(a)) \otimes \id) \Delta(b) = 0.
 \end{equation*}
 Note that $S(a \cdot b) = a \cdot S(b) + S(a) \cdot b$.
 Therefore equation~\eqref{eq:matpoi1} holds if and only if equation~\eqref{eq:dpb1} holds.
 Similarly we show that any equation of equations~\eqref{eq:matpoi2}--\eqref{eq:matpoi4} holds if and only if equation~\eqref{eq:dpb1} holds,
 and any equation of equations~\eqref{eq:matlie1}--\eqref{eq:matlie2} holds if and only if equation~\eqref{eq:dpb2} holds.
\end{proof}

\begin{Remark}
 If equations~\eqref{eq:vip1}--\eqref{eq:vip2} hold,
 then equations~\eqref{eq:dpb1}--\eqref{eq:dpb2} hold naturally, but the converse is not true.
 A counterexample could be given when $\Delta = 0$ with $S(a) \cdot b \neq 0$.
 That is, Corollary~\ref{cor:adjassmat} is a particular case of Proposition~\ref{pro:gmp}.
 Furthermore, note that in the case of Proposition~\ref{pro:gmp},
 the Poisson algebra structure on $A \oplus A^*$ obtained from the matched pair $((A, [\ ,\ ]_A, \cdot)$, $(A^*, [\ ,\ ]_{A^*}, \circ)$, $-\ad_A^*, L_A^*, -\ad_{A^*}^*, L_{A^*}^*)$ might not be
 the induced Poisson algebra of the commutative differential algebra on $A \oplus A^*$ obtained from the matched pair $((A, \cdot, \Phi)$, $(A^*, \circ, \Psi^*)$, $L_{A}^*$, $L_{A^*}^*)$.
\end{Remark}

\subsection[Manin triples of Poisson algebras via double constructions of commutative differential Frobenius algebras]{Manin triples of Poisson algebras via double constructions\\ of commutative differential Frobenius algebras}

Recall that a bilinear form $\mathfrak B$ on a Poisson algebra $(A, [\ ,\ ], \cdot)$ is called \textit{invariant} if
\begin{equation*}
	\mathfrak{B}([a, b], c) = \mathfrak{B}(a, [b, c]), \qquad \mathfrak{B}(a \cdot b, c) = \mathfrak{B}(a, b \cdot c), \qquad \forall a, b, c \in A.
\end{equation*}

\begin{Definition}[\cite{ni2013poisson}]
 A \textit{Manin triple of Poisson algebras} is a triple $((A, [\ ,\ ], \cdot, \mathfrak{B})$, $A^{+},
 A^{-})$,
 where $(A, [\ ,\ ], \cdot)$ is a Poisson algebra and $\mathfrak{B}$ is a nondegenerate symmetric invariant bilinear form on $(A, [\ ,\ ], \cdot)$ such that the following conditions are satisfied:
 \begin{enumerate}\itemsep=0pt
 \item[(1)] $A^{+}$ and $A^{-}$ are Poisson subalgebras of $A$;
 \item[(2)] $A = A^{+} \oplus A^{-}$ as vector spaces;
 \item[(3)]$A^{+}$ and $A^{-}$ are isotropic with respect to $\mathfrak{B}$.
 \end{enumerate}
\end{Definition}

\begin{Remark}It is obvious that a Manin triple of Poisson algebras is simultaneously a Manin triple of Lie algebras \cite{chari1995guide} and a double construction of Frobenius algebra.
\end{Remark}

\begin{Proposition}\label{pro:assFro}
 Let $(A, \cdot, \Phi = \{\partial_1, \partial_2\})$ be a commutative differential algebra.
 Suppose that there is a commutative differential algebra structure $(A^*, \circ, \Psi^*= \{\eth_1^*,\eth_2^*\})$ on $A^*$.
 Let $(A, [\ ,\ ]_A, \cdot)$ and $(A^*, [\ ,\ ]_{A^*}, \circ)$ be the induced Poisson algebras of $(A, \cdot, \Phi)$ and $(A^*, \circ, \Psi^*)$ respectively.
 Let $\Delta\colon A \to A \otimes A$ denote the linear dual of the multiplication $\circ\colon A^* \otimes A^* \to A^*$.
 Suppose that $(A \bowtie A^*, \star, \Phi + \Psi^*, \mathfrak{B}_d)$ is a double construction of commutative differential Frobenius algebra associated to $(A, \cdot, \Phi)$ and $(A^*, \circ, \Psi^*)$.
 Then there is a Manin triple $((A \oplus A^*,[\ ,\ ], \star, \mathfrak{B}_d), A, A^*)$ of Poisson algebras
 such that the Poisson algebra structure on $A \oplus A^*$ is the induced Poisson algebra of $(A \bowtie A^*, \star, \Phi + \Psi^*)$
 if and only equations~\eqref{eq:vip1}--\eqref{eq:vip2} hold.
 \end{Proposition}
\begin{proof} Let $(A \oplus A^*, [\ ,\ ], \star)$ be the induced Poisson algebra of $(A \bowtie A^*, \star, \Phi + \Psi^*)$.
 It is sufficient to prove that
 $\mathfrak{B}_d$ is invariant on $(A \oplus A^*, [\ ,\ ], \star)$
 if and only if equations~\eqref{eq:vip1}--\eqref{eq:vip2} hold.
 Let $a, b, c \in A$, $a^*, b^*, c^* \in A^*$.
 Then we have
 \begin{gather*}
 \mathfrak{B}_d(a + a^*, [b + b^*, c + c^*]) \\
 = \mathfrak{B}_d(a + a^*, (\partial_1(b)+\eth_1^{*}(b^*))\star(\partial_2(c)+\eth_2^{*}(c^*)) - (\partial_2(b)+\eth_2^{*}(b^*))\star(\partial_1(c)+\eth_1^{*}(c^*))) \\
 =\mathfrak{B}_d(a + a^*, \partial_1(b) \star \partial_2(c) + \partial_1(b) \star \eth_2^{*}(c^*) + \eth_1^{*}(b^*) \star \partial_2(c) + \eth_1^{*}(b^*) \star \eth_2^{*}(c^*))\\
 \quad {}- \mathfrak{B}_d(a + a^*, \partial_2(b) \star \partial_1(c) + \partial_2(b) \star \eth_1^{*}(c^*) + \eth_2^{*}(b^*) \star \partial_1(c) + \eth_2^{*}(b^*) \star \eth_1^{*}(c^*)) \\
 = \mathfrak{B}_d(a, \partial_1(b) \cdot \partial_2(c)) + \mathfrak{B}_d( \eth_2^{*}(c^*), a \cdot \partial_1(b)) + \mathfrak{B}_d( \eth_1^{*}(b^*), \partial_2(c) \cdot a) \\
 \quad {}+ \mathfrak{B}_d( a, \eth_1^{*}(b^*) \circ \eth_2^{*}(c^*)) + \mathfrak{B}_d( a^*, \partial_1(b) \cdot \partial_2(c) ) + \mathfrak{B}_d(\eth_2^{*}(c^*) \circ a^*, \partial_1(b)) \\
 \quad {}+ \mathfrak{B}_d(a^* \circ \eth_1^{*}(b^*),\partial_2(c)) + \mathfrak{B}_d(a^*,\eth_1^{*}(b^*) \circ \eth_2^{*}(c^*)) \\
 \quad {}- \mathfrak{B}_d(a, \partial_2(b) \cdot \partial_1(c)) - \mathfrak{B}_d( \eth_1^{*}(c^*), a \cdot \partial_2(b)) - \mathfrak{B}_d( \eth_2^{*}(b^*), \partial_1(c) \cdot a) \\
 \quad {}- \mathfrak{B}_d( a, \eth_2^{*}(b^*) \circ \eth_1^{*}(c^*)) - \mathfrak{B}_d( a^*, \partial_2(b) \cdot \partial_1(c) ) - \mathfrak{B}_d(\eth_1^{*}(c^*) \circ a^*, \partial_2(b)) \\
 \quad {}- \mathfrak{B}_d(a^* \circ \eth_2^{*}(b^*),\partial_1(c)) - \mathfrak{B}_d(a^*,\eth_2^{*}(b^*) \circ \eth_1^{*}(c^*)) \\
 = \langle a, \eth_1^{*}(b^*) \circ \eth_2^{*}(c^*) - \eth_2^{*}(b^*) \circ \eth_1^{*}(c^*) \rangle + \langle a^*, \partial_1(b) \cdot \partial_2(c) - \partial_2(b) \cdot \partial_1(c) \rangle \\
 \quad {}+ \langle b, \partial_1^{*}(\eth_2^{*}(c^*) \circ a^*) - \partial_2^{*}(\eth_1^{*}(c^*) \circ a^*) \rangle + \langle b^*, \eth_1(\partial_2(c) \cdot a)- \eth_2(\partial_1(c) \cdot a) \rangle \\
 \quad {}+ \langle c, \partial_2^{*}(\eth_1^{*}(b^*) \circ a^*) - \partial_1^{*}(\eth_2^{*}(b^*) \circ a^*) \rangle + \langle c^*, \eth_2(\partial_1(b) \cdot a) - \eth_1(\partial_2(b) \cdot a) \rangle.
 \end{gather*}
 Similarly, we show that
 \begin{gather*}
 \mathfrak{B}_d([a + a^*, b + b^*], c + c^*) \\
 = \langle a, \partial_1^{*}(\eth_2^{*}(b^*) \circ c^*) - \partial_2^{*}(\eth_1^{*}(b^*) \circ c^*) \rangle + \langle a^*, \eth_1(\partial_2(b) \cdot c)- \eth_2(\partial_1(b) \cdot c) \rangle \\
 \quad {}+ \langle b, \partial_2^{*}(\eth_1^{*}(a^*) \circ c^*) - \partial_1^{*}(\eth_2^{*}(a^*) \circ c^*) \rangle + \langle b^*, \eth_2(\partial_1(a) \cdot c) - \eth_1(\partial_2(a) \cdot c) \rangle \\
 \quad {}+ \langle c, \eth_1^{*}(a^*) \circ \eth_2^{*}(b^*) - \eth_2^{*}(a^*) \circ \eth_1^{*}(b^*) \rangle + \langle c^*, \partial_1(a) \cdot \partial_2(b) - \partial_2(a) \cdot \partial_1(b) \rangle.
 \end{gather*}
 Therefore $\mathfrak{B}_d(a + a^*, [b + b^*, c + c^*]) = \mathfrak{B}_d([a + a^*, b + b^*], c + c^*)$ if and only if the following six equations hold:
 \begin{gather}
 \eth_1^{*}(b^*) \circ \eth_2^{*}(c^*) - \eth_2^{*}(b^*) \circ \eth_1^{*}(c^*) = \partial_1^{*}(\eth_2^{*}(b^*) \circ c^*) - \partial_2^{*}(\eth_1^{*}(b^*) \circ c^*), \label{eq:pf151}\\
 \partial_1^{*}(\eth_2^{*}(c^*) \circ a^*) - \partial_2^{*}(\eth_1^{*}(c^*) \circ a^*) = \partial_2^{*}(\eth_1^{*}(a^*) \circ c^*) - \partial_1^{*}(\eth_2^{*}(a^*) \circ c^*), \label{eq:pf152}\\
 \partial_2^{*}(\eth_1^{*}(b^*) \circ a^*) - \partial_1^{*}(\eth_2^{*}(b^*) \circ a^*) = \eth_1^{*}(a^*) \circ \eth_2^{*}(b^*) - \eth_2^{*}(a^*) \circ \eth_1^{*}(b^*), \label{eq:pf153}\\
 \partial_1(b) \cdot \partial_2(c) - \partial_2(b) \cdot \partial_1(c) = \eth_1(\partial_2(b) \cdot c)- \eth_2(\partial_1(b) \cdot c), \label{eq:pf154}\\
 \eth_1(\partial_2(c) \cdot a)- \eth_2(\partial_1(c) \cdot a) = \eth_2(\partial_1(a) \cdot c) - \eth_1(\partial_2(a) \cdot c), \label{eq:pf155}\\
 \eth_2(\partial_1(b) \cdot a) - \eth_1(\partial_2(b) \cdot a) = \partial_1(a) \cdot \partial_2(b) - \partial_2(a) \cdot \partial_1(b). \label{eq:pf156}
 \end{gather}
 Note that we have the following relationships:
 \begin{alignat*}{3}
 &{\rm equation~\eqref{eq:pf151}} \Longleftrightarrow {\rm equation~\eqref{eq:pf153}}, \qquad&&
 {\rm equation~\eqref{eq:pf154}} \Longleftrightarrow {\rm equation~\eqref{eq:pf156}}, &\\
 &{\rm equation~\eqref{eq:pf151}} \Longrightarrow {\rm equation~\eqref{eq:pf152}}, \qquad&&
 {\rm equation~\eqref{eq:pf154}} \Longrightarrow {\rm equation~\eqref{eq:pf155}}.&
 \end{alignat*}
 Hence $\mathfrak{B}_d(a + a^*, [b + b^*, c + c^*]) = \mathfrak{B}_d([a + a^*, b + b^*], c + c^*)$ if and only if equations~\eqref{eq:pf151} and~\eqref{eq:pf154} hold.

 On the other hand, by Lemma~\ref{lem:douadm},
 $\Psi$ is admissible to $(A, \cdot, \Phi)$ and $\Phi^*$ is admissible to $(A^*, \circ, \Psi^*)$.
 By Corollary~\ref{cor:qadm}, we have
 \begin{gather*}
 \partial_1(b) \cdot \partial_2(c) - \partial_2(b) \cdot \partial_1(c) - \eth_1(\partial_2(b) \cdot c) + \eth_2(\partial_1(b) \cdot c) = \eth_2(\partial_1(b)) \cdot c - \eth_1(\partial_2(b)) \cdot c,\\
 \eth_1^{*}(b^*) \circ \eth_2^{*}(c^*) - \eth_2^{*}(b^*) \circ \eth_1^{*}(c^*) - \partial_1^{*}(\eth_2^{*}(b^*) \circ c^*) + \partial_2^{*}(\eth_1^{*}(b^*) \circ c^*) \\
 \qquad{} = \partial_2^*(\eth_1^*(b^*)) \circ c^* - \partial_1^*(\eth_2^*(b^*)) \circ c^*.
 \end{gather*}
 Hence equations~\eqref{eq:pf151} and \eqref{eq:pf154} hold if and only if equations~(\ref{eq:bvip2})--(\ref{eq:vip1}) hold.
 By Remark~\ref{rmk:cc}, the conclusion follows.
\end{proof}

With the conditions above, if equations~\eqref{eq:vip1}--\eqref{eq:vip2} hold, then the resulting Manin triple of Poisson algebras is called the \textit{induced Manin triple of Poisson algebras with respect to $(A \bowtie A^*, \star, \Phi + \Psi^*, \mathfrak{B}_d)$}.

As in the study of matched pairs, given a double construction of commutative differential Frobenius algebra $(A \bowtie A^*, \star, \Phi + \Psi^*, \mathfrak{B}_d)$,
there is a more general condition, such that $((A \oplus A^*,\allowbreak [\ ,\ ], \star, \mathfrak{B}_d), A, A^*)$ is a Manin triple of Poisson algebras,
\textit{without} the requirement that $(A \oplus A^*,\allowbreak [\ ,\ ], \star)$ is the induced Poisson algebra of $(A \bowtie A^*, \star, \Phi + \Psi^*)$.
With a similar study as the one of Proposition~\ref{pro:gmp}, we give the following conclusion omitting the proof.

\begin{Proposition}\label{pro:gmtf}
 With the same assumptions in Proposition~{\rm \ref{pro:assFro}}.
 Then there is a Manin triple $((A \oplus A^*,[\ ,\ ], \star, \mathfrak{B}_d), A, A^*)$ of Poisson algebras if and only equations~\eqref{eq:dpb1}--\eqref{eq:dpb2} hold.
\end{Proposition}

\subsection[Poisson bialgebras via commutative and cocommutative differential ASI bialgebras: the general case]{Poisson bialgebras via commutative and cocommutative differential\\ ASI bialgebras: the general case}

Recall that a pair $(A, \delta)$ is called a \textit{Lie
coalgebra}, where $A$ is a vector space and $\delta\colon A \to A
\otimes A$ is a linear map, if $\delta$ is
co-antisymmetric, in the sense that $\delta = -\sigma \delta$, and
satisfies the co-Jacobian identity:
\begin{equation*}
 \big(\id + \tau + \tau^2\big)(\id \otimes \delta)\delta = 0,
\end{equation*}
where $\tau(a \otimes b \otimes c):= c \otimes a \otimes b$ for all $a, b, c \in A$.
A \textit{Lie bialgebra} is a triple $(A, [\ ,\ ], \delta)$,
where $(A, [\ ,\ ])$ is a Lie algebra and $(A, \delta)$ is a Lie coalgebra,
satisfying the following equation:
\begin{gather}
 \delta([a, b]) = (\ad_A(a) \otimes \id + \id \otimes \ad_A (a) )\delta(b) - (\ad_A (b) \otimes \id + \id \otimes \ad_A (b) )\delta(a),
\label{eq:lie_bi}
\end{gather}
for all $a, b \in A$.

\begin{Definition}[\cite{ni2013poisson}]
 A \textit{Poisson coalgebra} is a triple $(A, \delta, \Delta)$,
 where $(A, \delta)$ is a Lie coalgebra and $(A, \Delta)$ is a commutative coalgebra
 such that the following equation holds:
 \begin{equation*}
 (\id \otimes \Delta)\delta(a) = (\delta \otimes \id)\Delta(a) + (\sigma \otimes \id)(\id \otimes \delta)\Delta(a), \qquad \forall a \in A. %\label{eq:copoi}
 \end{equation*}
\end{Definition}
\begin{Remark}
 The notion of a Poisson coalgebra is the dualization of the notion of a Poisson algebra, that is, $(A, \delta, \Delta)$ is a Poisson coalgebra if and only if $(A^*, \delta^*, \Delta^*)$ is a Poisson algebra.
\end{Remark}

\begin{Lemma} Let $(A, \Delta, \Psi \!=\! \{\eth_1, \eth_2\})$ be a cocommutative differential coalgebra.
 Then $(A, \delta, \Delta)$ is Poisson coalgebra, called the \textit{induced Poisson coalgebra} of $(A, \Delta, \Psi)$, where $\delta$ is defined by
 \begin{equation}
 \delta = (\eth_1 \otimes \eth_2 - \eth_2 \otimes \eth_1)
 \Delta. \label{eq:coP}
 \end{equation}
 Moreover, $(A^*, \delta^*, \Delta^*)$ is exactly the induced Poisson algebra of the commutative differential algebra $(A^*, \Delta^*, \Psi^*)$.
\end{Lemma}
\begin{proof}
 By assumption, $(A^*, \Delta^*, \Psi^*)$ is a commutative differential algebra.
 Let $(A^*, [\ ,\ ]_{A^*}, \Delta^*)$ be the induced Poisson algebra of $(A^*, \Delta^*, \Psi^*)$, that is, $[\ ,\ ]_{A^*}$ is defined by
 \begin{equation*}
 [a^*, b^*]_{A^*} := \eth_1^*(a^*) \circ \eth_2^*(b^*) - \eth_2^*(a^*) \circ \eth_1^*(b^*), \qquad \forall a^*, b^*\in A^*,
 \end{equation*}
 where $\circ = \Delta^*$.
 It is straightforward that $[\ ,\ ]_{A^*}$ is the linear dual of $\delta$.
 Hence $(A^*, \delta^*, \Delta^*)$ is the induced Poisson algebra of $(A^*, \Delta^*, \Psi^*)$ and thus $(A, \delta, \Delta)$ is a Poisson coalgebra.
\end{proof}

\begin{Definition}[\cite{ni2013poisson}]
 Let $(A, [\ ,\ ], \cdot)$ be a Poisson algebra.
 Suppose that it is equipped with two comultiplications $\delta, \Delta\colon A \to A \otimes A$ such that $(A, \delta, \Delta)$ is a Poisson coalgebra.
 If in addition, $(A, [\ ,\ ], \delta)$ is a Lie bialgebra,
 $(A, \cdot, \Delta)$ is a commutative and cocommutative ASI bialgebra,
 $\delta$ and $\Delta$ are compatible in the following sense:
 \begin{gather}
 	\delta(a \cdot b) = (L_A(a) \otimes \id)\delta(b) + (L_A(b) \otimes \id)\delta(a) \nonumber\\
 \hphantom{\delta(a \cdot b) =}{} + (\id \otimes \ad_A(a))\Delta(b) + (\id \otimes \ad_A(b))\Delta(a), \label{eq:poibi1} \\
 \Delta([a, b]) = (\ad_A(a) \otimes \id + \id \otimes \ad_A(a))\Delta(b) + (L_A(b) \otimes \id - \id \otimes L_A(b))\delta(a), \label{eq:poibi2}
 \end{gather}
 for all $a, b \in A$, then $(A, [\ ,\ ], \cdot, \delta, \Delta)$ is called a \textit{Poisson bialgebra}.
\end{Definition}

\begin{Theorem}\label{thm:asspoibi}
 Let $(A,\cdot, \Phi = \{\partial_1, \partial_2\})$ be a commutative differential algebra and $(A, \Delta, \Psi = \{\eth_1, \eth_2\})$ be a cocommutative differential coalgebra.
 Let $(A, [\ ,\ ], \cdot)$ be the induced Poisson algebra of $(A,\cdot, \Phi)$ and $(A, \delta, \Delta)$ be the induced Poisson coalgebra of $(A, \Delta, \Psi)$,
 that is, $[\ ,\ ]$ is defined by equation~\eqref{eq:lie} and $\delta$ is defined by equation~\eqref{eq:coP}.
 Suppose that $(A, \cdot, \Delta, \Phi, \Psi)$ is a~commutative and cocommutative differential ASI bialgebra.
 Then $(A, [\ ,\ ], \cdot, \delta, \Delta)$ is a Poisson bialgebra if and only if equations~\eqref{eq:dpb1}--\eqref{eq:dpb2} hold.
 In particular, if equations~\eqref{eq:vip1}--\eqref{eq:vip2} hold, then $(A, [\ ,\ ], \cdot, \delta, \Delta)$ is a Poisson bialgebra.
\end{Theorem}
\begin{proof} Let $a, b \in A$. Then we have
 \begin{align*}
 [a,b] &\overset{\hphantom{(2.7)}}{=} \partial_1(a) \cdot \partial_2(b) - \partial_2(a) \cdot \partial_1(b) \\
 & \overset{\eqref{eq:qadm1}}{=} -\eth_2(\partial_1(a) \cdot b) + \eth_1(\partial_2(a) \cdot b) + \eth_2(\partial_1(a)) \cdot b - \eth_1(\partial_2(a)) \cdot b.
 \end{align*}
 Hence $\ad_A(a) = L_A(\partial_1(a))\partial_2 - L_A(\partial_2(a))\partial_1 = \eth_1 L_A(\partial_2(a)) - \eth_2 L_A(\partial_1(a)) + L_A(S(a))$. Meanwhile,
 \begin{align*}
 \delta &\overset{\hphantom{(3.6)}}{=} (\eth_1 \otimes \eth_2 - \eth_2 \otimes \eth_1)\Delta \\
 &\overset{\eqref{eq:psadm2}}{=}(\id \otimes \eth_2)( (\id \otimes \partial_1 )\Delta - \Delta \partial_1) - (\id \otimes \eth_1)( (\id \otimes \partial_2)\Delta - \Delta \partial_2) \\
 &\overset{\hphantom{(3.6)}}{=} (\id \otimes \eth_1)\Delta \partial_2 - (\id \otimes \eth_2)\Delta \partial_1 + (\id \otimes S)\Delta.
 \end{align*}
 Therefore we have
 \begin{gather*}
 (L_A(a) \otimes \id)\delta(b) + (L_A(b) \otimes \id)\delta(a) + (\id \otimes \ad_A(a))\Delta(b) +(\id \otimes \ad_A(b))\Delta(a) -\delta(a \cdot b) \\
 \overset{\hphantom{(3.2)}}{=} (L_A(a) \otimes \id)( (\id \otimes \eth_1)\Delta(\partial_2(b)) -(\id \otimes \eth_2)\Delta(\partial_1(b)) + (\id \otimes S)\Delta (b)) \\
 \quad {}+ (L_A(b) \otimes \id)( (\id \otimes \eth_1)\Delta(\partial_2(a))-(\id \otimes \eth_2)\Delta(\partial_1(a)) + (\id \otimes S)\Delta (a) ) \\
 \quad {}+ (\id \otimes (\eth_1 L_A(\partial_2(a)) -\eth_2 L_A(\partial_1(a)) + L_A(S(a)) ) )\Delta(b) \\
 \quad {}+ (\id \otimes (\eth_1 L_A(\partial_2(b)) -\eth_2 L_A(\partial_1(b)) + L_A(S(b)) ) )\Delta(a) \\
 \quad {}- (\id \otimes \eth_1)\Delta(\partial_2(a \cdot b)) + (\id \otimes \eth_2)\Delta(\partial_1(a \cdot b)) - (\id \otimes S)\Delta(a \cdot b) \\
 \overset{\eqref{eq:asiifs}}{=} (\id \otimes \eth_1) \Delta(a \cdot \partial_2(b) + \partial_2(a) \cdot b ) - (\id \otimes \eth_2) \Delta(a \cdot \partial_1(b)+ \partial_1(a) \cdot b) \\
 \quad {}- (\id \otimes \eth_1)\Delta(\partial_2(a \cdot b)) + (\id \otimes \eth_2)\Delta(\partial_1(a \cdot b)) \\
 \quad {}+ ( - (\id \otimes S)(\id \otimes L_A(b)) + L_A(b) \otimes S + \id \otimes L_A(S(b)) ) \Delta(a) + ( \id \otimes L_A(S(a)) ) \Delta(b) \\
 \overset{\hphantom{(3.2)}}{=} (L_A(b) \otimes \id - \id \otimes L_A(b))(\id \otimes S) \Delta(a) + ( \id \otimes L_A(S(a)) ) \Delta(b).
	\end{gather*}
 Hence equation~\eqref{eq:poibi1} holds if and only if equation~\eqref{eq:dpb1} holds.
 Similarly we show that equations~\eqref{eq:poibi2} and \eqref{eq:lie_bi} hold if and only if equations~\eqref{eq:dpb1} and \eqref{eq:dpb2} hold respectively.
 So the first conclusion follows. The particular case is obvious.
\end{proof}

With the conditions above, if equations~\eqref{eq:vip1}--\eqref{eq:vip2} hold, then $(A, [\ ,\ ], \cdot, \delta, \Delta)$ is called the \textit{induced Poisson bialgebra of $(A, \cdot, \Delta, \Phi, \Psi)$}.

\begin{Remark}\label{rmk:iff}
 Note that for a commutative and cocommutative differential ASI bialgebra $(A, \cdot, \Delta, \Phi, \Psi)$,
 there is a commutative differential algebra structure $(A \bowtie A^*, \star, \Phi + \Psi^*)$ on $A \oplus A^*$,
 and for a Poisson bialgebra $(A, [\ ,\ ], \cdot, \delta, \Delta)$,
 there is a Poisson algebra structure $(A \bowtie A^*, [\ ,\ ], \cdot)$ on $A \oplus A^*$.
 Hence with the conditions above, if equations~\eqref{eq:vip1}--\eqref{eq:vip2} hold,
 then the resulting induced Poisson bialgebra $(A, [\ ,\ ], \cdot, \delta, \Delta)$ of $(A, \cdot, \Delta, \Phi, \Psi)$ satisfies an additional condition that the Poisson structure on $A \oplus A^*$ is the induced Poisson algebra of $(A \bowtie A^*, \star,\allowbreak \Phi + \Psi^*)$.
 So if one characterizes the resulting Poisson bialgebra $(A, [\ ,\ ], \cdot, \delta, \Delta)$ of a commutative and cocommutative differential ASI bialgebra $(A, \cdot, \Delta, \Phi, \Psi)$ to be ``induced" by this additional condition,
 then the converse still holds, that is,
 $(A, [\ ,\ ], \cdot, \delta, \Delta)$ is the ``induced" Poisson bialgebra in this sense if and only if equations~\eqref{eq:vip1}--\eqref{eq:vip2} hold.
\end{Remark}

\begin{Theorem}[{\cite[structure Theorem~(II)]{ni2013poisson}}]\label{thm:poibieq}
 Let $(A, [\ ,\ ]_A, \cdot)$ be a Poisson algebra.
 Suppose that there is a Poisson algebra structure $(A^*, [\ ,\ ]_{A^*}, \circ)$ on $A^*$.
 Let $\delta, \Delta\colon A \to A \otimes A$ denote the linear duals of the Lie bracket $[\ ,\ ]_{A^*}\colon A^* \otimes A^* \to A^*$ and the multiplication $\circ\colon A^* \otimes A^* \to A^*$ respectively.
 Then the following conditions are equivalent:
 \begin{enumerate}\itemsep=0pt
 \item[$(1)$] $(A, [\ ,\ ]_A, \cdot, \delta, \Delta)$ is a Poisson bialgebra.
 \item[$(2)$] $((A, [\ ,\ ]_A, \cdot), (A^*, [\ ,\ ]_{A^*}, \circ), -\ad_{A}^*, L_{A}^*, -\ad_{A^*}^*, L_{A^*}^*)$ is a matched pair of Poisson algebras.
 \item[$(3)$] There is a Manin triple of Poisson algebras $((A \oplus A^*, [\ ,\ ], \star, \mathfrak{B}_d), A, A^*)$ associated to $(A, [\ ,\ ]_A, \cdot)$ and $(A^*, [\ ,\ ]_{A^*}, \circ)$, where $\mathfrak{B}_d$ is defined by equation~\eqref{eq:bilform}.
 \end{enumerate}
\end{Theorem}

By the equivalences given in Theorems~\ref{thm:asieqv} and \ref{thm:poibieq}, Proposition~\ref{pro:gmp} and Corollary~\ref{cor:adjassmat} (in the induced case) are the ``matched pair version" of Theorem~\ref{thm:asspoibi} and Remark~\ref{rmk:iff},
whereas Propositions~\ref{pro:gmtf} and~\ref{pro:assFro} (in the induced case) are the ``Manin triple (double construction) version" of the latter.
We illustrate these relationships in the ``induced case'' through a commutative diagram.
Explicitly, let $(A, \cdot, \Phi = \{\partial_1, \partial_2\})$ be a commutative differential algebra.
Suppose that there is a commutative differential algebra structure $(A^*, \circ, \Psi^*= \{\eth_1^*,\eth_2^*\})$ on $A^*$.
Let $(A, [\ ,\ ]_A, \cdot)$ and $(A^*, [\ ,\ ]_{A^*}, \circ)$ be the induced Poisson algebras of $(A, \cdot, \Phi)$ and $(A^*, \circ, \Psi^*)$ respectively,
that is, both $[\ ,\ ]_A$ and $[\ ,\ ]_{A^*}$ are defined by equation~\eqref{eq:lie}.
Let $\delta, \Delta\colon A \to A \otimes A$ be the linear duals of $[\ ,\ ]_{A^*}$ and $\circ$ respectively.
So $\delta$ satisfies equation~\eqref{eq:coP}.
Suppose that equations~\eqref{eq:vip1}--\eqref{eq:vip2} hold.
Then we have the following commutative diagram:
\begin{equation*}
 \xymatrix@C=1.12cm{
 \txt{$((A, \cdot, \Phi)$, $(A^*, \circ, \Psi^*)$,$L_{A}^*,L_{A^*}^*)$ \\ a matched pair of \\ commutative differential \\ algebras} \ar[d]_-{{\rm Cor.}~\ref{cor:adjassmat}} \ar@{<->}[r]^-{{\rm Thm.}~\ref{thm:asieqv}} &
 \txt{$(A, \cdot, \Delta, \Phi, \Psi)$ \\ a commutative and \\ cocommutative \\ differential ASI \\ bialgebra} \ar@{<->}[r]^-{{\rm Thm.}~\ref{thm:asieqv}} \ar[d]_-{{\rm Thm.}~\ref{thm:asspoibi}} &
 \txt{$(A \bowtie A^*, \star , \mathfrak{B}_d)$ \\ a double construction of \\ commutative differential \\ Frobenius algebra} \ar[d]_-{{\rm Prop.}~\ref{pro:assFro}} \\
 \txt{$(A, A^*, -\ad_{A}^*, L_{A}^*, -\ad_{A^*}^*, L_{A^*}^*)$ \\ the induced matched \\ pair of Poisson algebras} \ar@{<->}[r]^-{{\rm Thm.}~\ref{thm:poibieq}} &
 \txt{$(A, [\ ,\ ]_A, \cdot, \delta, \Delta)$ \\ the induced \\ Poisson bialgebra} \ar@{<->}[r]^-{{\rm Thm.}~\ref{thm:poibieq}} &
 \txt{$((A \oplus A^*,[\ ,\ ], \star, \mathfrak{B}_d), A, A^*)$ \\ the induced Manin triple \\ of Poisson algebras}
 }
\end{equation*}

\subsection[Coboundary Poisson bialgebras via coboundary commutative and cocommutative differential ASI bialgebras]{Coboundary Poisson bialgebras via coboundary commutative\\ and cocommutative differential ASI bialgebras}

\begin{Definition}[\cite{ni2013poisson}]
 A Poisson bialgebra $(A, [\ ,\ ], \cdot, \delta, \Delta)$ is called \textit{coboundary}
 if $\delta$ and $\Delta$ are respectively defined by
	\begin{gather}
 \delta(a) := (\id \otimes \ad_A(a) + \ad_A(a) \otimes \id )(r), \label{eq:pcbd1}\\
 \Delta(a) := (\id \otimes L_A(a) - L_A(a) \otimes \id)(r), \qquad \forall a \in A,\label{eq:pcbd2}
 \end{gather}
for some $r \in A \otimes A$.
\end{Definition}

\begin{Definition}[\cite{ni2013poisson}]
 Let $(A, [\ ,\ ], \cdot)$ be a Poisson algebra.
 $r \in A \otimes A$ is called a solution of \textit{Poisson Yang--Baxter equation (PYBE) in $(A, [\ ,\ ], \cdot)$} if $r$ is a solution of both equation~\eqref{eq:aybe} and the following equation:
 \begin{equation}
 [r_{12}, r_{13}] + [r_{13}, r_{23}] + [r_{12}, r_{23}] = 0, \label{eq:CYBE}
 \end{equation}
 where for $r = \sum_i a_i \otimes b_i$, we denote
 \begin{gather*}
 [r_{12}, r_{13}] = \sum_{i,j} [a_i, a_j] \otimes b_i \otimes b_j, \qquad
 [r_{13}, r_{23}] = \sum_{i,j} a_i \otimes a_j \otimes [b_i, b_j], \\
 [r_{12},r_{23}] = \sum_{i,j} a_i \otimes [b_i, a_j] \otimes b_j.
 \end{gather*}
\end{Definition}

\begin{Proposition}[{\cite[Theorem~2]{ni2013poisson}}]\label{pro:pybebi}
Let $(A, [\ ,\ ], \cdot)$ be a Poisson algebra and $r \in A \otimes A$.
 Let $\delta\colon A \to A \otimes A$ and $\Delta\colon A \to A \otimes A$ be two linear maps defined by equations~\eqref{eq:pcbd1} and \eqref{eq:pcbd2} respectively.
 If $r$ is an antisymmetric solution of PYBE in $(A, [\ ,\ ], \cdot)$,
 then $(A, [\ ,\ ], \cdot, \delta, \Delta)$ is a~Poisson bialgebra.
\end{Proposition}

\begin{Lemma}\label{lem:tch}
 Let $(A, \cdot, \Phi = \{\partial_1, \partial_2\})$ be a $\Psi = \{\eth_1, \eth_2\}$-admissible commutative differential algebra and $(A, [\ ,\ ], \cdot)$ be the induced Poisson algebra of $(A, \cdot, \Phi)$.
 Let $r \in A \otimes A$ be an antisymmetric solution of $\Psi$-admissible AYBE in $(A, \cdot, \Phi)$.
 Suppose that $\Delta$ is defined by equation~\eqref{eq:pcbd2}.
 Then we have the following conclusions.
 \begin{enumerate}\itemsep=0pt
 \item[$(1)$] $(A, \Delta, \Psi)$ is a cocommutative differential coalgebra.
 \item[$(2)$] If equation~\eqref{eq:vip1} holds, then both equation~\eqref{eq:vip2} and the following equation hold:
 \begin{equation}
 [a, b] = \eth_2(a \cdot \partial_1(b)) - \eth_1(a \cdot \partial_2(b)), \qquad \forall a, b \in A. \label{eq:liepq}
 \end{equation}
 \end{enumerate}
\end{Lemma}
\begin{proof} (1)~It is straightforward to show that $(A, \Delta)$ is a cocommutative coalgebra.
 For all $a \in A$ and $k=1, 2$, we have
 \begin{gather*}
 (\eth_k \otimes \id + \id \otimes \eth_k)\Delta(a) = (\eth_k \otimes \id + \id \otimes \eth_k)(\id \otimes L_A(a) - L_A(a) \otimes \id )(r) \\
 \qquad{} \overset{\eqref{eq:pqadm1}}{=} (\id \otimes L_A(a) \partial_k + \id \otimes \eth_k L_A(a) - \eth_k L_A(a) \otimes \id - L_A(a) \partial_k \otimes \id ) (r) \\
\qquad{} \;\overset{\eqref{eq:qadm1}}{=} (\id \otimes L_A(\eth_k(a)) - L_A(\eth_k(a)) \otimes \id ) (r) = \Delta(\eth_k(a)).
	\end{gather*}
 Hence $\eth_k$ is a coderivation on $(A, \Delta)$.
 Therefore, $(A, \Delta, \Psi)$ is a cocommutative differential coalgebra.

 (2)
 Let $a, b \in A$. Then we have
 \begin{align*}
 \eth_1(\partial_2(a \cdot b)) &\overset{\hphantom{(0.0)}}{=} \eth_1(\partial_2(a) \cdot b + a \cdot \partial_2(b)) \\
 &\overset{\eqref{eq:qadm1}}{=} \eth_1(\partial_2(a)) \cdot b - \partial_2(a) \cdot \partial_1(b) + a \cdot \eth_1(\partial_2(b)) - \partial_1(a) \cdot \partial_2(b), \\
 \eth_2(\partial_1(a \cdot b)) &\overset{\hphantom{(0.0)}}{=} \eth_2(\partial_1(a) \cdot b + a \cdot \partial_1(b)) \\
 &\overset{\eqref{eq:qadm1}}{=} \eth_2(\partial_1(a)) \cdot b - \partial_1(a) \cdot \partial_2(b) + a \cdot \eth_2(\partial_1(b)) - \partial_2(a) \cdot \partial_1(b).
 \end{align*}
 Hence by assumption, we have
 $\eth_1(\partial_2(a \cdot b)) = \eth_2(\partial_1(a \cdot b))$, i.e.,
 $\eth_1 \partial_2 L_A(a) = \eth_2 \partial_1 L_A(a)$.
 Therefore
 \begin{gather*}
 (\id \otimes \eth_1 \partial_2)\Delta(a) - (\id \otimes \eth_2 \partial_1)\Delta(a) = (\id \otimes \eth_1 \partial_2 - \id \otimes \eth_2 \partial_1)(\id \otimes L_A(a) - L_A(a) \otimes \id )(r) \\
 = - (L_A(a) \otimes (\eth_1 \partial_2 - \eth_2 \partial_1) )(r) \overset{\eqref{eq:pqadm1}}{=} - (L_A(a) \eth_2 \partial_1 \otimes \id )(r) + (L_A(a) \eth_1 \partial_2 \otimes \id)(r) = 0.
	\end{gather*}
 That is, equation~\eqref{eq:vip2} holds. Moreover, we have
 \begin{gather*}
 \partial_1(a) \cdot \partial_2(b) + \eth_1(a \cdot \partial_2(b)) - \partial_2(a) \cdot \partial_1(b) - \eth_2(a \cdot \partial_1(b))\\
 \qquad{} \overset{\eqref{eq:qadm2}}{=} a \cdot \eth_1(\partial_2(b)) - a \cdot \eth_2(\partial_1(b)) \overset{\eqref{eq:vip1}}{=} 0.
 \end{gather*}
 Therefore equation~\eqref{eq:liepq} holds.
\end{proof}

\begin{Proposition}\label{pro:aybe_to_qybe}
 Let $(A, \cdot, \Phi = \{\partial_1, \partial_2\})$ be a $\Psi = \{\eth_1, \eth_2\}$-admissible commutative differential algebra and $(A, [\ ,\ ], \cdot)$ be the induced Poisson algebra of $(A, \cdot, \Phi)$.
 Suppose that equation~\eqref{eq:vip1} holds.
 Then every solution of $\Psi$-admissible AYBE in $(A, \cdot, \Phi)$ is a solution of PYBE in the Poisson algebra $(A, [\ ,\ ], \cdot)$.
\end{Proposition}
\begin{proof}
 Suppose that $r = \sum_i a_i \otimes b_i$ is a solution of $\Psi$-admissible AYBE in $(A, \cdot, \Phi)$.
 Then we have
 \begin{gather*}
 \sum_{i,j}[a_i,a_j] \otimes b_i \otimes b_j = \sum_{i,j}(\partial_1(a_i) \cdot \partial_2(a_j) - \partial_2(a_i) \cdot \partial_1(a_j)) \otimes b_i \otimes b_j\\
 \qquad{} \overset{\eqref{eq:pqadm1}}{=} \sum_{i,j}( a_i \cdot a_j \otimes \eth_1(b_i) \otimes \eth_2(b_j) - a_i \cdot a_j \otimes \eth_2(b_i) \otimes \eth_1(b_j) )\\
 \qquad{} \overset{\eqref{eq:aybe}}{=} \sum_{i,j}( a_i \otimes \eth_1(b_i \cdot a_j) \otimes \eth_2(b_j) - a_i \otimes \eth_1(a_j) \otimes \eth_2(b_i \cdot b_j)) \\
 \qquad\qquad {}- \sum_{i,j}( a_i \otimes \eth_2(b_i \cdot a_j) \otimes \eth_1(b_j) - a_i \otimes \eth_2(a_j) \otimes \eth_1(b_i \cdot b_j) )\\
 \qquad{} \overset{\eqref{eq:pqadm1}}{=} \sum_{i,j}( a_i \otimes \eth_1(b_i \cdot \partial_2(a_j)) \otimes b_j - a_i \otimes a_j \otimes \eth_2(b_i \cdot \partial_1(b_j))) \\
 \qquad\qquad {}- \sum_{i,j}(a_i \otimes \eth_2(b_i \cdot \partial_1(a_j)) \otimes b_j - a_i \otimes a_j \otimes \eth_1(b_i \cdot \partial_2(b_j)) )\\
 \qquad{} \overset{\eqref{eq:liepq}}{=} \sum_{i,j}( -a_i \otimes [b_i, a_j] \otimes b_j - a_i \otimes a_j \otimes [b_i,b_j]).
 \end{gather*}
 Hence $r$ satisfies equation~\eqref{eq:CYBE}.
 Therefore $r$ is a solution of PYBE in $(A, [\ ,\ ], \cdot)$.
\end{proof}

Let $(A, \cdot, \Phi = \{\partial_1, \partial_2\})$ be a $\Psi = \{\eth_1, \eth_2\}$-admissible commutative differential algebra and $(A, [\ ,\ ], \cdot)$ be the induced Poisson algebra of $(A, \cdot, \Phi)$.
Suppose that equation~\eqref{eq:vip1} holds and $r \in A \otimes A$ is an antisymmetric solution of $\Psi$-admissible AYBE in $(A, \cdot, \Phi)$.
On the one hand, by Corollary~\ref{cor:admAYBEdasi}, there is a commutative and cocommutative differential ASI bialgebra $(A, \cdot, \Delta, \Phi, \Psi)$, where $\Delta$ is defined by equation~\eqref{eq:pcbd2}.
Furthermore, by Lemma~\ref{lem:tch}, $(A, \Delta, \Psi)$ is a cocommutative differential coalgebra and equation~\eqref{eq:vip2} holds, and thus by Theorem~\ref{thm:asspoibi},
there is the induced Poisson bialgebra $(A,[\ ,\ ], \cdot, \delta, \Delta)$ of $(A, \cdot, \Delta, \Phi, \Psi)$,
where $\delta$ is defined by equation~\eqref{eq:coP}.
On the other hand, by Proposition~\ref{pro:aybe_to_qybe}, $r$ is an antisymmetric solution of PYBE in the Poisson algebra $(A, [\ ,\ ], \cdot)$
and hence there is a Poisson bialgebra $(A,[\ ,\ ], \cdot, \delta', \Delta)$ by Proposition~\ref{pro:pybebi},
where $\delta'$ is defined by equation~\eqref{eq:pcbd1}.

\begin{Corollary}\label{cor:asbi}
 With the conditions as above.
 Then the two Poisson bialgebras $(A,[\ ,\ ], \cdot, \delta, \Delta)$ and $(A,[\ ,\ ], \cdot, \delta', \Delta)$ coincide.
 Hence we have the following commutative diagram:
 \begin{equation*}
 \xymatrix@C=4cm{
 \txt{$r$ \\ an antisymmetric solution of \\ $\Psi$-admissible AYBE in $(A, \cdot, \Phi)$} \ar[d]_-{{\rm Prop.}~\ref{pro:aybe_to_qybe}} \ar[r]^{{\rm Cor.}~\ref{cor:admAYBEdasi}} &
 \txt{$(A, \cdot, \Delta, \Phi, \Psi)$ \\ a commutative and cocommutative \\ differential ASI bialgebra} \ar[d]^-{{\rm Thm.}~\ref{thm:asspoibi}} \\
 \txt{$r$ \\ an antisymmetric solution of \\ PYBE in $(A, [\ ,\ ], \cdot)$} \ar[r]^{{\rm Prop.}~\ref{pro:pybebi}} &
 \txt{$(A, [\ ,\ ], \cdot, \delta, \Delta)$ \\ the induced Poisson bialgebra}
 }
 \end{equation*}
\end{Corollary}
\begin{proof}
 Let $a \in A$.
 Then we have
 \begin{align*}
 \delta(a)&\overset{\hphantom{(0.00)}}{=} (\eth_1 \otimes \eth_2 - \eth_2 \otimes \eth_1) \Delta(a) =(\eth_1 \otimes \eth_2 - \eth_2 \otimes \eth_1) (\id \otimes L_A(a) - L_A(a) \otimes \id )(r) \\
 &\overset{\eqref{eq:pqadm1}}{=} (\id \otimes \eth_2 L_A(a) \partial_1 - \id \otimes \eth_1 L_A(a) \partial_2 )(r) - (\eth_1 L_A(a) \partial_2 \otimes \id - \eth_2 L_A(a) \partial_1 \otimes \id)(r) \\
 &\overset{\eqref{eq:liepq}}{=} (\id \otimes \ad_A (a))(r) + (\ad_A(a) \otimes \id)(r) = \delta'(a).
	\end{align*}
 Hence the two Poisson bialgebras $(A,[\ ,\ ], \cdot, \delta, \Delta)$ and $(A,[\ ,\ ], \cdot, \delta', \Delta)$ coincide.
\end{proof}

\begin{Definition}[\cite{liu2020noncommutative}]
 Let $(A, [\ ,\ ], \cdot)$ be a Poisson algebra and $(V, \rho, \mu)$ be a module of $(A, [\ ,\ ], \cdot)$.
 A linear map $T\colon V \to A$ is called an \textit{$\mathcal{O}$-operator of $(A, [\ ,\ ], \cdot)$ associated to $(V, \rho, \mu)$} if the following equations hold:
 \begin{gather*}
 [T(u), T(v)] = T(\rho(T(u))v - \rho(T(v))u),\\ % \label{eq:poiopp1} \\
 T(u) \cdot T(v) = T(\mu(T(u))v+ \mu(T(v))u), \qquad
 \forall u, v \in V. %\label{eq:poiopp2}
 \end{gather*}
	In particular, when $(V, \rho, \mu)$ is taken to be $(A, \ad_A, L_A)$, an $\mathcal O$-operator $R\colon A \to A$ is called a~\textit{Rota--Baxter operator (of weight zero)} on $(A, [\ ,\ ], \cdot)$, that is, $R$ satisfies
	\begin{gather*}
 	 R(a) \cdot R(b) = R(R(a) \cdot b + a \cdot R(b)),\\
 [R(a), R(b)] = R([R(a), b] + [a, R(b)]), \qquad \forall a, b \in A.
	\end{gather*}
\end{Definition}

It is straightforward to get the following conclusion which is similarly to Lemma~\ref{lem:o2r}.

\begin{Lemma}\label{lem:o2rp}
 Let $(A, [\ ,\ ], \cdot)$ be a Poisson algebra and $(V, \rho, \mu)$ be a module of $(A, [\ ,\ ], \cdot)$.
 A linear map $T\colon V \to A$ is an $\mathcal{O}$-operator of $(A, [\ ,\ ], \cdot)$ associated to $(V, \rho, \mu)$ if and only if $\hat T\colon A \oplus V \to A \oplus V$ is a Rota--Baxter operator on the semi-direct product Poisson algebra $(A \ltimes_{\rho, \mu} V, [\ ,\ ], \cdot)$, where $\hat T$ is defined by equation~{\rm (\ref{eq:hatT})}.
\end{Lemma}

\begin{Lemma}\label{lem:rod}
 Let $(A, \cdot, \Phi = \{\partial_1, \partial_2\})$ be a commutative differential algebra and
 $(A, [\ ,\ ], \cdot)$ be the induced Poisson algebra of $(A, \cdot, \Phi)$.
 If $R$ is a Rota--Baxter operator on $(A, \cdot)$ commuting with $\partial_k$ for all $k=1,2$, then $R$ is a Rota--Baxter operator on the Poisson algebra $(A, [\ ,\ ], \cdot)$.
\end{Lemma}
\begin{proof}
 For all $a, b \in A$, we have
 \begin{gather*}
 R([R(a), b] + [a, R(b)]) \\
\qquad{} \overset{\eqref{eq:lie}}{=} R\big(\partial_1(R(a)) \cdot \partial_2(b) - \partial_2(R(a)) \cdot \partial_1(b) + \partial_1(a) \cdot \partial_2(R(b)) - \partial_2(a) \cdot \partial_1(R(b))\big) \\
\qquad{} \overset{\hphantom{(5.1)}}{=} R\big(R(\partial_1(a)) \cdot \partial_2(b) - R(\partial_2(a)) \cdot \partial_1(b) + \partial_1(a) \cdot R(\partial_2(b)) - \partial_2(a) \cdot R(\partial_1(b))\big) \\
 \qquad{}\overset{\hphantom{(5.1)}}{=} R(\partial_1(a)) \cdot R(\partial_2(b)) - R(\partial_2(a)) \cdot R(\partial_1(b)) \\
 \qquad{}\overset{\hphantom{(5.1)}}{=} \partial_1(R(a)) \cdot \partial_2(R(b)) - \partial_2(R(a)) \cdot \partial_1(R(b)) = [R(a), R(b)].
	\end{gather*}
 We finish the proof.
\end{proof}

\begin{Proposition}[{\cite[Proposition~A.6, Theorem~2.4.7]{BaiDouble}}]\label{pro:pybe_o}
 Let $(A, [\ ,\ ], \cdot)$ be a Poisson algebra and $r \in A \otimes A$ be antisymmetric.
 Then $r$ is an antisymmetric solution of PYBE in $(A, [\ ,\ ], \cdot)$ if and only if $r^\sharp$ is an $\mathcal{O}$-operator of $(A, [\ ,\ ], \cdot)$ associated to the module $(A^*, -\ad_A^*, L_A^*)$.
\end{Proposition}

\begin{Proposition}\label{pro:assopp}
 Let $(A, \cdot, \Phi = \{\partial_1, \partial_2\})$ be a commutative differential algebra and
 $(A, [\ ,\ ], \cdot)$ be the induced Poisson algebra of $(A, \cdot, \Phi)$.
 Let $T\colon V \to A$ be an $\mathcal{O}$-operator of $(A, \cdot, \Phi)$
 associated to a module $(V, \mu, \Omega = \{\alpha_1, \alpha_2\})$.
 Then $T$ is an $\mathcal{O}$-operator of $(A, [\ ,\ ], \cdot)$ associated to $(V, \rho_\mu, \mu)$,
 where $(V, \rho_\mu,\mu)$ is the induced module of $(A, [\ ,\ ], \cdot)$ with respect to $(V, \mu, \Omega)$.
 In particular, when taking $(V, \mu,\Omega) = (A^*, L_A^*, \Psi^*=\{\eth_1^*, \eth_2^*\})$,
 if in addition equation~\eqref{eq:vip1} holds, then $T$ is an $\mathcal{O}$-operator of $(A, [\ ,\ ], \cdot)$ associated to $(A^*, -\ad_A^* , L_A^*)$.
\end{Proposition}
\begin{proof}
 Since $T\colon V \to A$ is an $\mathcal{O}$-operator of $(A, \cdot, \Phi)$
 associated to $(V, \mu, \Omega)$, by Lemma~\ref{lem:o2r}, the linear map $\hat T$ defined by equation~(\ref{eq:hatT}) is a Rota--Baxter operator on the semi-direct algebra $(A \ltimes_\mu V, \cdot)$. Moreover, since $\partial_k T = T\alpha_k$ for all $k=1,2$, we have $\hat T$ commutes with $\partial_k + \alpha_k$ for all $k=1,2$.
 By Lemma~\ref{lem:rod}, $\hat T$ is a Rota--Baxter operator on the induced Poisson algebra of $(A \ltimes_\mu V, \cdot, \Phi + \Omega)$.
 By Corollary~\ref{cor:asssemi}, the induced Poisson algebra is exactly the semi-direct product Poisson algebra by $(A, [\ ,\ ], \cdot)$ and $(V, \rho_\mu, \mu)$. Hence $\hat T$ is a Rota--Baxter operator on the semi-direct product Poisson algebra by $(A, [\ ,\ ], \cdot)$ and $(V, \rho_\mu, \mu)$.
 Therefore $T$ is an $\mathcal{O}$-operator of $(A, [\ ,\ ], \cdot)$ associated to $(V, \rho_\mu, \mu)$ by Lemma~\ref{lem:o2rp}.
 The particular case follows from Proposition~\ref{pro:assdualrep}.
\end{proof}

Combining Corollary~\ref{cor:aybe_o}, Propositions~\ref{pro:aybe_to_qybe}, \ref{pro:pybe_o} and \ref{pro:assopp} together, we have the following conclusion.
\begin{Corollary}
 Let $(A, \cdot, \Phi = \{\partial_1, \partial_2\})$ be a $\Psi = \{\eth_1, \eth_2\}$-admissible commutative differential algebra and
 $(A, [\ ,\ ], \cdot)$ be the induced Poisson algebra of $(A, \cdot, \Phi)$.
 Suppose that equation~\eqref{eq:vip1} holds and $r$ is an antisymmetric solution of $\Psi$-admissible AYBE in $(A,\cdot,\Phi)$.
 On the one hand, $r$ is a solution of PYBE in the Poisson algebra $(A, [\ ,\ ], \cdot)$ by Proposition~{\rm \ref{pro:aybe_to_qybe}} and
 hence $r^\sharp$ is an $\mathcal{O}$-operator of $(A, [\ ,\ ], \cdot)$ associated to $(A^*, -\ad_A^*, L_A^*)$ by Proposition~{\rm \ref{pro:pybe_o}}.
 On the other hand, $r^\sharp$ is an $\mathcal{O}$-operator of $(A, \cdot, \Phi)$ associated to $(A^*, L_A^*, \Psi^*)$ by Corollary~{\rm \ref{cor:aybe_o}} and
 hence $r^\sharp$ is an $\mathcal{O}$-operator of $(A, [\ ,\ ], \cdot)$ associated to $(A^*, -\ad_A^*, L_A^*)$ by Proposition~{\rm \ref{pro:assopp}}.
 That is, the two approaches to get $r^\sharp$ as an $\mathcal{O}$-operator of $(A, [\ ,\ ], \cdot)$ coincide, and thus we have the following commutative diagram:
 \begin{equation*}
 \xymatrix@C=3.5cm{
 \txt{$r$ \\ an antisymmetric solution of \\ $\Psi$-admissible AYBE in $(A, \cdot, \Phi)$} \ar[d]^-{{\rm Prop.}~\ref{pro:aybe_to_qybe}} \ar[r]^-{{\rm Cor.}~\ref{cor:aybe_o}} &
 \txt{$r^\sharp$\\ an $\mathcal{O}$-operator of $(A, \cdot, \Phi)$ \\ associated to $(A^*, L_A^*, \Psi^*)$} \ar[d]^-{{\rm Prop.}~\ref{pro:assopp} } \\
 \txt{$r$ \\ an antisymmetric solution of \\ PYBE in $(A, [\ ,\ ], \cdot)$ } \ar[r]^-{{\rm Prop.}~\ref{pro:pybe_o}}
 & \txt{$r^\sharp$ \\ an $\mathcal{O}$-operator of $(A, [\ ,\ ], \cdot)$ \\ associated to $(A^*, -\ad_A^*, L_A^*)$}
 }
 \end{equation*}
\end{Corollary}

\subsection[Poisson bialgebras via O-operators of commutative differential algebras and differential Zinbiel algebras]{Poisson bialgebras via $\boldsymbol{\mathcal{O}}$-operators of commutative differential algebras\\ and differential Zinbiel algebras}

\begin{Proposition}\label{pro:cond}
 Let $(A, \cdot, \Phi = \{\partial_1,\partial_2\})$ be a commutative differential algebra and
 $(A, [\ ,\ ], \cdot)$
 be the induced Poisson algebra of $(A, \cdot, \Phi)$.
 Suppose that $(V, \mu, \Omega = \{\alpha_1, \alpha_2\})$ is a module of $(A, \cdot, \Phi)$ and
 $(V, \rho_\mu, \mu)$ is the induced module of $(A, [\ ,\ ], \cdot)$ with respect to $(V, \mu, \Omega)$.
 Then $(A \ltimes_{\mu^*} V^*, \cdot, \Phi - \Omega^*)$ is a $(-\Phi + \Omega^*)$-admissible commutative differential algebra.
 Furthermore, $(A \ltimes_{-\rho_\mu^*, \mu^*} V^*, [\ ,\ ], \cdot)$ is the induced Poisson algebra of $(A \ltimes_{\mu^*} V^*, \cdot, \Phi - \Omega^*)$
 and equation~\eqref{eq:vip1} holds, where $\eth_k$ is replaced by $-\partial_k + \alpha_k^*$ and $\partial_k$ is replaced by $\partial_k - \alpha_k^*$ for all $k=1, 2$.
\end{Proposition}
\begin{proof}
By Example~\ref{ex:admissible} and the proof of Corollary~\ref{cor:qminp},
$(A \ltimes_{\mu^*} V^*, \cdot, \Phi - \Omega^*)$ is a $(-\Phi + \Omega^*)$-admissible commutative differential algebra.
By Corollary~\ref{cor:asssemi}, $(A \ltimes_{\rho_{\mu^*}, \mu^*} V^*, [\ ,\ ], \cdot)$ is the induced Poisson algebra of $(A \ltimes_{\mu^*} V^*, \cdot,
\Phi - \Omega^*)$,
where $(V^*, \rho_{\mu^*}, \mu^*)$ is the induced module of $(A, [\ ,\ ], \cdot)$ with respect to the module $(V^*, \mu^{*}, -\Omega^*)$.
For all $a \in A, v \in V$, we have
\begin{gather*}
 (-\alpha_2)(\mu(\partial_1(a))v) - (-\alpha_1)(\mu(\partial_2(a))v) + \mu(\partial_1(a))\alpha_2(v) - \mu(\partial_2(a))\alpha_1(v) \\
\qquad{} = (\mu(\partial_1(a))\alpha_2(v) - \alpha_2(\mu(\partial_1(a))v) ) - (\mu(\partial_2(a))\alpha_1(v) - \alpha_1(\mu(\partial_2(a))v) ) \\
 \qquad{} = - \mu(\partial_2(\partial_1(a)))v + \mu(\partial_1(\partial_2(a)))v = 0.
\end{gather*}
Therefore, by Proposition~\ref{pro:assdualrep}, we have
$\rho_{\mu^*} = - \rho_{\mu}^*$
and hence $(A \ltimes_{-\rho_\mu^*, \mu^*} V^*, [\ ,\ ], \cdot)$ is exactly the
induced Poisson algebra of $(A \ltimes_{\mu^*} V^*, \cdot, \Phi - \Omega)$.
Moreover, we have
\begin{equation*}
 (-\partial_2+\alpha_2^*)((\partial_1-\alpha_1^*)(a+v^*)) \cdot (b+w^*) =
 (-\partial_1+\alpha_1^*)((\partial_2-\alpha_2^*)(a+v^*)) \cdot (b+w^*),
\end{equation*}
for all $a, b \in A$, $v^*, w^* \in V^*$. Hence equation~\eqref{eq:vip1} holds.
\end{proof}

\begin{Remark}One may consider extending the above study to the case that
$\big(A \ltimes_{\mu^*} V^*, \cdot, \Phi + \{-\alpha_k^* + \theta_k\id_{V^*}\}_{k=1}^2\big)$ is a $\big(\{-\partial_k + \theta_k \id_A\}_{k=1}^2 + \Omega^*\big)$-admissible commutative differential algebra for any $\theta_1,\theta_2\in \mathbb F$ as in Corollary~\ref{cor:qminp}.
Then by a similar argument,
$(A \ltimes_{-\rho_\mu^*, \mu^*} V^*, [\ ,\ ], \cdot)$ is the induced Poisson algebra of $\big(A \ltimes_{\mu^*} V^*, \cdot, \Phi+\{-\alpha_k^* + \theta_k\id_{V^*}\}_{k=1}^2\big)$,
where $(V, \rho_\mu, \mu)$ is the induced module of $(A, [\ ,\ ], \cdot)$ with respect to $(V, \mu, \{\alpha_1 - \theta_1 \id_V, \alpha_2 - \theta_2 \id_V \})$.
Furthermore, it is straightforward to show that equation~\eqref{eq:vip1} holds
where $\eth_k$ is replaced by $-\partial_k + \theta_k \id_{A} + \alpha_k^*$ and $\partial_k$ is replaced by $\partial_k - \alpha_k^* + \theta_k \id_{V^*}$ for all $k = 1, 2$
if and only if
\begin{gather}
 \theta_2 \partial_1(a)b = \theta_1 \partial_2(a)b, \qquad\!
 \theta_2 \mu(\partial_1(a))v = \theta_1\mu(\partial_2(a))v, \qquad\!
 \theta_2 \alpha_1(\mu(a)v) = \theta_1 \alpha_2(\mu(a)v),\!\!\!\! \label{eq:gsp}
\end{gather}
for all $a \in A, v \in V$.
Suppose that equation~\eqref{eq:gsp} holds for some $\theta_k \neq 0$.
We assume that $\theta_1 \neq 0$ without loss of generality.
Let $a, b \in A$.
Then $\partial_2(a) \cdot b = \theta_1^{-1} \theta_2 \partial_1(a) \cdot b$ and thus we have
\begin{equation*}
 [a, b] = \partial_1(a) \cdot \partial_2(b) - \partial_2(a) \cdot
 \partial_1(b)
 = \theta_1^{-1} \theta_2 \partial_1(a) \cdot \partial_1(b) - \theta_1^{-1} \theta_2 \partial_1(a) \cdot \partial_1(b) = 0.
\end{equation*}
Similarly, the Lie bracket in $(A \ltimes_{-\rho_\mu^*, \mu^*} V^*, [\ ,\ ], \cdot)$ is also trivial.
So in this sense, we only need to consider the case that $\theta_1 = \theta_2 = 0$ as we have done in Proposition~\ref{pro:cond}.
\end{Remark}

\begin{Theorem}[{\cite[Theorem~5.24]{liu2020noncommutative}}]\label{thm:oopsemi}
Let $(V, \rho, \mu)$ be a module of a Poisson algebra $(A, [\ ,\ ], \cdot)$.
 Let $T\colon V \to A$ be a linear map. Then $r = T - \sigma(T)$ is an antisymmetric solution of PYBE in $(A \ltimes_{-\rho^*, \mu^*} V^*, [\ ,\ ], \cdot)$ if and only if
 $T$ is an $\mathcal{O}$-operator of $(A, [\ ,\ ], \cdot)$ associated to $(V, \rho, \mu)$.
\end{Theorem}

By Proposition~\ref{pro:cond}, combining Corollary~\ref{cor:qminp}, Propositions~\ref{pro:aybe_to_qybe} and~\ref{pro:assopp} and Theorem~\ref{thm:oopsemi} together,
we have the following conclusion.

\begin{Corollary}\label{cor:oybe}With the conditions in Proposition~{\rm \ref{pro:cond}}.
Suppose that $T\colon V \to A$ is an $\mathcal{O}$-operator of $(A, \cdot, \Phi)$ associated to $(V, \mu, \Omega)$.
On the one hand, $T$ is an $\mathcal{O}$-operator of the induced Poisson algebra $(A, [\ ,\ ], \cdot)$ associated to $(V, \rho_\mu, \mu)$ by
Proposition~{\rm \ref{pro:assopp}} and hence $r = T - \sigma(T)$ is an antisymmetric solution of PYBE in
$(A \ltimes_{-\rho_\mu^*,\mu^*} V^*,[\ ,\ ], \cdot)$ by Theorem~{\rm \ref{thm:oopsemi}}.
On the other hand, $r = T - \sigma(T)$ is an antisymmetric solution of
$(-\Phi + \Omega^*)$-admissible AYBE in $(A \ltimes_{\mu^*} V^*, \cdot, \Phi - \Omega^*)$
by Corollary~{\rm \ref{cor:qminp}}.
Due to Proposition~{\rm \ref{pro:cond}}, $r$ is an antisymmetric solution of
PYBE in the induced Poisson algebra $(A \ltimes_{-\rho_\mu^*, \mu^*} V^*, [\ ,\ ], \cdot)$
by Proposition~{\rm \ref{pro:aybe_to_qybe}}.
That is, the two approaches to get $r = T - \sigma(T)$ as an antisymmetric solution of PYBE
in $(A \ltimes_{-\rho_\mu^*,\mu^*} V^*,[\ ,\ ], \cdot)$ coincide,
and thus we have the following commutative diagram:
\begin{equation*}
 \xymatrix@C=3cm{
 \txt{$T$ \\ an $\mathcal{O}$-operator of $(A, \cdot, \Phi)$ \\ associated to $(V, \mu, \Omega)$ } \ar[d]^-{{\rm Prop.}~\ref{pro:assopp}} \ar[r]^-{{\rm Cor.}~\ref{cor:qminp}} &
 \txt{$r = T - \sigma(T)$ \\ an antisymmetric solution of \\ $(-\Phi + \Omega^*)$-admissible AYBE \\ in $(A \ltimes_{\mu^*} V^*, \cdot, \Phi - \Omega^*)$} \ar[d]^-{{\rm Prop.}~\ref{pro:aybe_to_qybe}} \\
 \txt{$T$ \\ an $\mathcal{O}$-operator of $(A, [\ ,\ ], \cdot)$ \\ associated to $(V, \rho_\mu, \mu)$ } \ar[r]^-{{\rm Thm.}~\ref{thm:oopsemi}} &
 \txt{$r = T - \sigma(T)$ \\ an antisymmetric solution of \\ PYBE in $(A \ltimes_{-\rho_\mu^*, \mu^*} V^*, [\ ,\ ], \cdot)$}
 }
\end{equation*}
\end{Corollary}

\begin{Definition}[\cite{gerstenhaber1963cohomology}]
 A \textit{$($left$)$ pre-Lie algebra $(A,\diamond)$} is a vector space $A$ together with a~bilinear multiplication $\diamond\colon A \otimes A \to A$ satisfying the following equation:
 \begin{equation*}
 a \diamond (b \diamond c) - (a \diamond b) \diamond c = b \diamond (a \diamond c) - (b \diamond a) \diamond c, \qquad \forall a, b, c \in A.
 \end{equation*}
\end{Definition}

\begin{Definition}[\cite{loday1995cup}]
 A \textit{$($left$)$ Zinbiel algebra $(A,\ast)$} is a vector space $A$ together with a bilinear multiplication $\ast\colon A \otimes A \to A$ satisfying the following equation:
 \begin{equation*}
 a \ast (b \ast c) = (a \ast b) \ast c + (b \ast a) \ast c, \qquad \forall a, b, c \in A.
 \end{equation*}
\end{Definition}

A Zinbiel algebra $(A, \ast)$ is equivalently defined as a dendriform algebra $(A, \succ, \prec)$ in which
\begin{equation*}
 a \succ b = b \prec a = a \ast b, \qquad \forall a, b \in A.
\end{equation*}
Hence for a Zinbiel algebra $(A, \ast)$, the associated algebra $(A, \cdot)$ is commutative,
where
\begin{equation}
 a \cdot b = a \ast b + b \ast a, \qquad \forall a, b \in A. \label{eq:Zin}
\end{equation}
Furthermore, we give the notion of a differential Zinbiel algebra:
$(A, \ast, \Phi)$ is called a \textit{differential Zinbiel algebra}
if $(A, \ast, \ast^{\rm op}, \Phi)$ is a differential dendriform algebra,
where $a \ast^{\rm op} b = b \ast a$ for all $a, b \in A$.
It is obvious that the associated differential algebra of a differential Zinbiel algebra is commutative.

\begin{Definition}[\cite{aguiar2000pre}]
 A \textit{$($left$)$ pre-Poisson algebra} is a triple $(A, \diamond, \ast)$,
 where $(A, \diamond)$ is a pre-Lie algebra and
 $(A, \ast)$ is a Zinbiel algebra
 such that the following conditions hold:
 \begin{gather*}
 (a \diamond b - b \diamond a) \ast c = a \diamond (b \ast c) - b \ast (a \diamond c), \qquad
 (a \ast b + b \ast a) \diamond c = a \ast (b \diamond c) + b \ast (a \diamond c),
 \end{gather*}
	for all $a, b, c \in A$.
\end{Definition}

\begin{Proposition}[{\cite[Proposition~2.2]{aguiar2000pre}}]\label{pro:pptp}
 Let $(A, \diamond, \ast)$ be a pre-Poisson algebra.
 Define two bilinear multiplications $\cdot, [\ ,\ ]\colon A\otimes A\rightarrow A$ respectively by
 \begin{equation}
 a \cdot b := a \ast b + b \ast a \qquad {\rm and} \qquad
 [a, b] := a \diamond b - b \diamond a, \qquad
 \forall a, b \in A. \label{eq:pptp}
 \end{equation}
 Then $(A, [\ ,\ ], \cdot)$ is a Poisson algebra,
 called the \textit{associated Poisson algebra of $(A, \diamond, \ast)$}.
\end{Proposition}

Recall that a \textit{perm-algebra} is a vector space $P$ with a bilinear multiplication $(a,b) \to ab$,
such that $P$ is an algebra which is left-commutative in the following sense:
\begin{equation*}
	(ab)c = (ba)c, \qquad \forall a, b, c \in P,
\end{equation*}
that is, a perm-algebra is associative and left-commutative.

\begin{Lemma}[\cite{gubarev2014operads}]\label{lem:perm}
Let $A$ be a vector space with two bilinear multiplications $\diamond$ and $*$.
Then $(A, \diamond , *)$ is a~pre-Poisson algebra if and only if for every perm-algebra $P$,
$(P \otimes A, [\ ,\ ], \cdot)$ is a~Poisson algebra, where $[\ ,\ ]$ and $\cdot$ are respectively defined by
	\begin{gather*}
		[p \otimes a, q \otimes b] = pq \otimes a \diamond b - qp \otimes b \diamond a,\\
		(p \otimes a) \cdot (q \otimes b) = pq \otimes a*b + qp \otimes b*a, \qquad \forall p, q \in P , a, b \in A.
	\end{gather*}
\end{Lemma}

\begin{Proposition}\label{pro:dzpp} Let $(A, \ast, \Phi = \{\partial_1, \partial_2\})$ be a differential Zinbiel algebra.
 Define a bilinear multiplication $\diamond\colon A \otimes A \to A$ by
 \begin{equation}
 a \diamond b := \partial_1(a) \ast \partial_2 (b) - \partial_2(a) \ast \partial_1(b), \qquad \forall a, b \in A. \label{eq:dzpp}
 \end{equation}
 Then $(A, \diamond, \ast)$ is a pre-Poisson algebra, called the \textit{induced pre-Poisson algebra of $(A, \ast, \Phi)$}.
\end{Proposition}
\begin{proof}The conclusion can be verified by a direct proof or the following approach in terms of perm-algebras. Let $P$ be a perm-algebra. Define a bilinear multiplication $\cdot$ on $P \otimes A$ by
 \begin{equation*}
 (p \otimes a) \cdot (q \otimes b) := pq \otimes a*b + qp \otimes b*a, \qquad \forall p, q \in P, a, b \in A.
 \end{equation*}
 For all $p, q, r \in P$ and $a, b, c \in A$, we have
 \begin{gather*}
 (p \otimes a) \cdot ((q \otimes b) \cdot (r \otimes c)) \\
 \qquad{}= pqr \otimes a \ast (b \ast c) + qrp \otimes (b \ast c) \ast a + prq \otimes a \ast (c \ast b) + rqp \otimes (c \ast b) \ast a \\
 \qquad{}=pqr \otimes a \ast (b \ast c) + rqp \otimes (b \ast c) \ast a + prq \otimes a \ast (c \ast b) + rqp \otimes (c \ast b) \ast a \\
\qquad{} = pqr \otimes a \ast (b \ast c) + prq \otimes a \ast (c \ast b) + rqp \otimes c \ast (b \ast a), \\
 ((p \otimes a) \cdot (q \otimes b)) \cdot (r \otimes c) \\
\qquad{} = pqr \otimes (a \ast b) \ast c + rpq \otimes c \ast (a \ast b) + qpr \otimes (b \ast a) \ast c + rqp \otimes c \ast (b \ast a) \\
\qquad{} = pqr \otimes (a \ast b) \ast c + rpq \otimes c \ast (a \ast b) + pqr \otimes (b \ast a) \ast c + rqp \otimes c \ast (b \ast a) \\
\qquad{} = pqr \otimes a \ast (b \ast c) + rpq \otimes c \ast (a \ast b) + rqp \otimes c \ast (b \ast a).
	\end{gather*}
 Hence $(P \otimes A, \cdot)$ is an algebra. Obviously $(P \otimes A, \cdot)$ is commutative.
 Define linear maps $d_1,d_2\colon P \otimes A \to P \otimes A$ respectively by
 \begin{equation*}
 	d_i: P \otimes A \to P \otimes A, \qquad p \otimes a \mapsto p \otimes \partial_i(a), \qquad \forall a \in A, \quad p \in P, \quad i=1,2.
 \end{equation*}
 It is straightforward to verify that $(P \otimes A, \cdot, \{d_1, d_2\})$ is a commutative differential algebra.
 Let $(P \otimes A, [\ ,\ ], \cdot)$ be the induced Poisson algebra of
 $(P \otimes A, \cdot, \{d_1, d_2\})$, where $[\ ,\ ]$ is defined by equation~\eqref{eq:lie}, i.e.,
 \begin{gather*}
 	[p \otimes a, q \otimes b] = d_1(p \otimes a) \cdot d_2(q \otimes b) - d_2(p \otimes a) \cdot d_1(q \otimes b) \\
 \qquad{}= (p \otimes \partial_1(a)) \cdot (q \otimes \partial_2(b)) - (p \otimes \partial_2(a) ) \cdot (q \otimes \partial_1(b)) \cdot\\
 \qquad{}= (pq \otimes \partial_1(a) \ast \partial_2(b) + qp \otimes \partial_2(b) \ast \partial_1(a)) - ( pq \otimes \partial_2(a) \ast \partial_1(b) + qp \otimes \partial_1(b) \ast \partial_2(a))\\
\qquad{} = pq \otimes a \diamond b-qp \otimes b \diamond a, \qquad \forall p,q \in P, \quad a, b \in A.
	\end{gather*}
	Hence the conclusion follows from Lemma~\ref{lem:perm}.
\end{proof}

\begin{Theorem}
 Let $(A, \ast, \Phi = \{\partial_1, \partial_2\})$ be a differential Zinbiel algebra.
 Let $(A, \cdot, \Phi)$ be the associated commutative differential algebra of $(A, \ast, \Phi)$, where $\cdot$ is defined by equation~\eqref{eq:Zin}.
 Let $(A, \diamond, \ast)$ be the induced pre-Poisson algebra of $(A, \ast, \Phi)$, where $\diamond$ is defined by equation~\eqref{eq:dzpp}.
 Then the associated Poisson algebra $(A, [\ ,\ ], \cdot)$ of the pre-Poisson algebra $(A, \diamond, \ast)$, where $[\ ,\ ]$ and $\cdot$ are respectively defined by equation~\eqref{eq:pptp}, is exactly the induced Poisson algebra $(A, [\ ,\ ], \cdot)$ of the differential algebra $(A, \cdot, \Phi)$, where $[\ ,\ ]$ is defined by equation~\eqref{eq:lie}.
 Hence we have the following commutative diagram:
 \begin{equation*}
 \xymatrix@C=4cm{
 \txt{differential Zinbiel algebra\\ $(A, \ast, \Phi)$} \ar[d]^-{equation~\eqref{eq:dzpp}} \ar[r]^-{equation~\eqref{eq:Zin}} &
 \txt{commutative differential algebra\\ $(A, \cdot, \Phi)$} \ar[d]^-{equation~\eqref{eq:lie}} \\
 \txt{pre-Poisson algebra\\ $(A, \diamond, \ast)$} \ar[r]^-{equation~\eqref{eq:pptp}} &
 \txt{Poisson algebra \\ $(A, [\ ,\ ], \cdot)$}
 }
 \end{equation*}
\end{Theorem}
\begin{proof}Let $P$ be a one-dimensional vector space with a basis $\{e\}$ with a bilinear multiplication whose product is given by $ee=e$.
Obviously $P$ is a perm-algebra which is isomorphic to the field $\mathbb F$. Substituting such $P$ into the proof of Proposition~\ref{pro:dzpp} and noting that in this case, for any vector space $A$, $P\otimes A$ and $A$ are isomorphic as vector spaces, we get the conclusion immediately.
\end{proof}

\begin{Lemma}[{\cite[Theorem~5.18]{liu2020noncommutative}}]\label{lem:otpp}
 Let $(A, [\ ,\ ], \cdot)$ be a Poisson algebra.
 Suppose that $T$ is an $\mathcal{O}$-operator of $(A, [\ ,\ ], \cdot)$ associated to $(V, \rho, \mu)$.
 Define two bilinear multiplications $\ast, \circ\colon V \otimes V \to V$ respectively by
 \begin{equation}
 u \diamond v := \rho(T(u)) v, \qquad
 u \ast v := \mu(T(u))v, \qquad
 \forall u, v \in V. \label{eq:otpp}
 \end{equation}
 Then $(V, \diamond, \ast)$ is a pre-Poisson algebra.
\end{Lemma}

Let $(A, \cdot, \Phi = \{\partial_1, \partial_2\})$ be a commutative differential algebra.
Suppose that $T$ is an $\mathcal{O}$-operator of $(A, \cdot, \Phi)$ associated to a module $(V, \mu, \Omega = \{\alpha_1, \alpha_2\})$.
Then $(V, \ast, \Omega)$ is a differential Zinbiel algebra by Proposition~\ref{pro:opd}, where $\ast$ is defined by equation~\eqref{eq:otd} in the commutative case, that is,
\begin{equation}
 u \ast v := \mu(T(u)) v, \qquad \forall u,v \in V. \label{eq:Zin11}
\end{equation}
Furthermore, there is the induced pre-Poisson algebra $(V,
\diamond, \ast)$ of $(V, \ast, \Omega)$ by
Proposition~\ref{pro:dzpp}, where $\diamond$ is defined by
equation~\eqref{eq:dzpp}. On the other hand, let $(A, [\ ,\ ], \cdot)$
be the induced Poisson algebra of $(A, \cdot, \Phi)$, where $[\ ,\
]$ is defined by equation~\eqref{eq:lie}. Then by
Proposition~\ref{pro:assopp}, $T$ is an $\mathcal O$-operator on
$(A, [\ ,\ ], \cdot)$ associated to $(V, \rho_\mu, \mu)$ which is
the induced module of $(A, [\ ,\ ], \cdot)$ with respect
to $(V, \mu, \Omega)$. Hence there is a pre-Poisson algebra $(V,
\diamond', \ast')$ by Lemma~\ref{lem:otpp}, where $\diamond'$ and
$\ast'$ are respectively defined by equation~\eqref{eq:otpp}.

\begin{Proposition}With the conditions above.
 Then the two pre-Poisson algebras $(V, \diamond, \ast)$ and $(V, \diamond', \ast')$ coincide.
 Hence we have the following commutative diagram:
 \begin{equation*}
 \xymatrix@C=4cm{
 \txt{$T$ \\ an $\mathcal{O}$-operator of $(A, \cdot, \Phi)$ \\ associated to $(V, \mu, \Omega)$ } \ar[r]^-{{\rm Prop.}~\ref{pro:opd}} \ar[d]^-{{\rm Prop.}~\ref{pro:assopp}} &
 \txt{$(V, \ast, \Omega)$ \\ a differential \\ Zinbiel algebra} \ar[d]^-{{\rm Prop.}~\ref{pro:dzpp}} \\
 \txt{$T$ \\ an $\mathcal{O}$-operator of $(A, [\ ,\ ], \cdot)$\\ associated to $(V, \rho_\mu, \mu)$ } \ar[r]^-{{\rm Lem.}~\ref{lem:otpp}} &
 \txt{$(V, \diamond, \ast)$ \\ a pre-Poisson\\ algebra}
 }
 \end{equation*}
\end{Proposition}
\begin{proof}Let $u, v \in V$.
 Then by equations~\eqref{eq:dzpp} and~\eqref{eq:Zin11}, we have
 \begin{gather*}
 u \ast v = \mu(T(u)) v,\\
 u\diamond v=\alpha_1(u) \ast \alpha_2(v) - \alpha_2(u) \ast \alpha_1(v)=\mu(T(\alpha_1(u)))\alpha_2(v)-\mu(T(\alpha_2(u)))\alpha_1(v).
 \end{gather*}
 On the other hand, by equations~\eqref{eq:otpp} and \eqref{eq:def_rep_lie}, we have
 \begin{align*}
 u \ast' v &\overset{\hphantom{(0.00)}}{=} \mu(T(u)) v, \\
 u \diamond' v &\overset{\hphantom{(0.00)}}{=} \rho_\mu(T(u))v = \mu(\partial_1(T(u))) \alpha_2(v) - \mu(\partial_2(T(u))) \alpha_1(v) \\
 &\overset{\eqref{eq:oop2}}{=} \mu(T(\alpha_1(u))) \alpha_2(v) - \mu(T(\alpha_2(u))) \alpha_1(v).
 \end{align*}
 Hence $(V, \diamond, \ast)$ and $(V, \diamond', \ast')$ coincide.
\end{proof}

\begin{Proposition}[{\cite[Example~5.16]{liu2020noncommutative}}]\label{pro:ppto}
Let $(A, \diamond, \ast)$ be a pre-Poisson algebra and $(A, [\ ,\ ], \cdot)$ be the associated Poisson algebra of $(A, \diamond,
 \ast)$,
 where $[\ ,\ ]$ and $\cdot$ are respectively defined by equation~\eqref{eq:pptp}.
 Then $(A, L_\diamond, L_\ast)$ is a module of $(A, [\ ,\ ], \cdot)$,
 where $L_\diamond, L_\ast\colon A \to \End(A)$ are respectively defined by
 \begin{equation*}
 L_\diamond(a)(b) := a \diamond b, \qquad L_\ast(a)(b) := a \ast b, \qquad \forall a, b \in A.
 \end{equation*}
 Furthermore, the identity map $\id\colon A \to A$ is an $\mathcal{O}$-operator of $(A, [\ ,\ ], \cdot)$ associated to $(A, L_\diamond, L_\ast)$.
\end{Proposition}

\begin{Proposition}\label{pro:idop} Let $(A, \ast, \Phi = \{\partial_1, \partial_2\})$ be a differential Zinbiel algebra.
 Let $(A, \cdot, \Phi)$ be the associated commutative differential algebra of $(A, \ast, \Phi)$, where $\cdot$ is defined by equation~\eqref{eq:Zin}, and
 $(A, [\ ,\ ], \cdot)$ be the induced Poisson algebra of $(A, \cdot, \Phi)$, where $[\ ,\ ]$ is defined by equation~\eqref{eq:lie}.
 Let $(A, \diamond, \ast)$ be the induced pre-Poisson algebra of $(A, \ast, \Phi)$ given in Proposition~{\rm \ref{pro:dzpp}}, where $\diamond$ is defined by equation~\eqref{eq:dzpp}.
 Then the module $(A, L_\diamond, L_\ast)$ of the Poisson algebra $(A, [\ ,\ ], \cdot)$ is exactly the induced module $(A, \rho_{L_\ast}, L_\ast)$ of $(A, [\ ,\ ], \cdot)$ with respect to $(A, L_\ast, \Phi)$, that is, ${L_\diamond = \rho_{L_\ast}}$.
 Furthermore, on the one hand, the identity map $\id$ is an $\mathcal{O}$-operator of the
 commutative differential algebra $(A, \cdot, \Phi)$ associated to $(A, L_\ast, \Phi)$ by Proposition~{\rm \ref{pro:assda}} and
 hence $\id$ is an $\mathcal{O}$-operator of the Poisson algebra~$(A, [\ ,\ ], \cdot)$ associated to $(A, \rho_{L_{\ast}}, L_{\ast})$ by Proposition~{\rm \ref{pro:assopp}}.
 On the other hand, $\id$ is an $\mathcal{O}$-operator of $(A, [\ ,\ ], \cdot)$ associated to $(A, L_\diamond, L_\ast)$ by Proposition~{\rm \ref{pro:ppto}}.
 Therefore the two approaches for $\id$ as an $\mathcal O$-operator of the Poisson algebra $(A, [\ ,\ ], \cdot)$ coincide and hence we have the following commutative diagram:
 \begin{equation*}
 \xymatrix@C=4cm{
 \txt{$(A, \ast, \Phi)$ \\ a differential \\ Zinbiel algebra} \ar[d]^-{{\rm Prop.}~\ref{pro:dzpp}} \ar[r]^-{{\rm Prop.}~\ref{pro:assda}} &
 \txt{$\id$ \\ an $\mathcal{O}$-operator of $(A, \cdot, \Phi)$\\ associated to $(A, L_{\ast}, \Phi)$ } \ar[d]^-{{\rm Prop.}~\ref{pro:assopp}} \\
 \txt{$(A, \diamond, \ast)$ \\ a pre-Poisson \\ algebra} \ar[r]^-{{\rm Prop.}~\ref{pro:ppto}} &
 \txt{$\id$ \\ an $\mathcal{O}$-operator of $(A, [\ ,\ ], \cdot)$ \\ associated to $(A, L_{\diamond}, L_{\ast})$}
 }
 \end{equation*}
\end{Proposition}
\begin{proof}
 In fact, we have
 \begin{equation*}
 \rho_{L_{\ast}}(a)(b)= L_{\ast}(\partial_1(a)) \partial_2(b) - L_{\ast}(\partial_2(a))\partial_1(b) \overset{\eqref{eq:dzpp}}{=} a \diamond b = L_{\diamond}(a)(b), \qquad \forall a, b \in A.
 \end{equation*}
 Hence the conclusion follows from Propositions~\ref{pro:assda},~\ref{pro:assopp}, \ref{pro:dzpp} and \ref{pro:ppto}.
\end{proof}

At the end of this paper, we give the following construction of Poisson bialgebras from differential Zinbiel algebras.

\begin{Proposition}Let $(A, \ast, \Phi = \{\partial_1, \partial_2\})$ be a differential Zinbiel algebra.
 Let $(A, \cdot, \Phi)$ be the associated commutative differential algebra of $(A, \ast, \Phi)$,
 where $\cdot$ is defined by equation~\eqref{eq:Zin}, and $(A, [\ ,\ ], \cdot)$ be the induced Poisson algebra of $(A, \cdot, \Phi)$, where $[\ ,\ ]$ is defined by equation~\eqref{eq:lie}.
 Let $(A, \diamond, \ast)$ be the induced pre-Poisson algebra of $(A, \ast, \Phi)$, where $\diamond$ is defined by equation~\eqref{eq:dzpp}.
 Let $\{e_1, e_2,\dots, e_n\}$ be a basis of $A$ and $\{e_1^*, e_2^*, \dots, e_n^*\}$ be the dual basis.
 Set
 \begin{equation*}
 r = \sum_{i=1}^n (e_i \otimes e_i^* - e_i^* \otimes e_i).
 \end{equation*}
 Then there is a Poisson bialgebra $(A \ltimes_{-L^*_{\diamond},L^*_{\ast}} A^*, [\ ,\ ], \cdot, \delta, \Delta)$,
 where $\delta$ and $\Delta$ are defined by equations~\eqref{eq:pcbd1} and~\eqref{eq:pcbd2} with the above $r$ respectively.
 Moreover, $(A \ltimes_{L^*_{\ast}}A^*, \cdot, \Delta, \Phi - \Phi^*, -\Phi + \Phi^*)$
 is a commutative and cocommutative differential ASI bialgebra
 whose induced Poisson bialgebra is exactly $(A \ltimes_{-L^*_{\diamond},L^*_{\ast}} A^*, [\ ,\ ], \cdot, \delta, \Delta)$.
\end{Proposition}
\begin{proof}
 Note that $\id = \sum^n_{i=1} e_i\otimes e_i^*$ and $r = \id - \sigma(\id)$.
 The first conclusion follows from Proposition~\ref{pro:idop}, Theorem~\ref{thm:oopsemi} and Proposition~\ref{pro:pybebi}.
 The second conclusion follows from Proposition~\ref{pro:idop}, Corollaries~\ref{cor:oybe} and \ref{cor:asbi}.
 Note that here equation~\eqref{eq:vip1} holds, where $\eth_k$ is replaced by $-\partial_k + \partial_k^*$ and $\partial_k$ is replaced by $\partial_k - \partial_k^*$ for all $k=1, 2$.
\end{proof}

We give the following example to illustrate the above construction explicitly.
Note that the classification of complex differential Zinbiel algebras in dimension $\leq 4$ is given in \cite{almutari2018derivations}.
\begin{Example}
 Let $(A,\ast)$ be a Zinbiel algebra in dimension 3 with a basis $\{e_1, e_2, e_3\}$ whose non-zero product is given by
 \begin{equation*}
 e_1 \ast e_1 = e_3, \qquad e_2 \ast e_1 = e_3, \qquad e_2 \ast e_2 = e_3.
 \end{equation*}
 Let $\partial_1, \partial_2\colon A \to A$ be two liner maps respectively given by
 \begin{gather*}
 \partial_1(e_1) = e_1, \qquad \partial_1(e_2) = e_2, \qquad \partial_1(e_3) = 2 e_3; \\
 \partial_2(e_1) = e_2, \qquad \partial_2(e_2) = -e_1 + e_2, \qquad \partial_2(e_3) = e_3.
 \end{gather*}
 Then $(A, \ast, \Phi = \{\partial_1, \partial_2\})$ is a differential Zinbiel algebra.
 Let $\{e_1^*, e_2^*, e_3^*\}$ be the dual basis of $\{e_1, e_2, e_3\}$.
 Then we have
 \begin{gather*}
 \partial_1^*(e_1^*) = e_1^*, \qquad \partial_1^*(e_2^*) = e_2^*, \qquad \partial_1^*(e_3^*) = 2 e_3^*; \\
 \partial_2^*(e_1^*) = -e_2^*, \qquad \partial_2^*(e_2^*) = e_1^* + e_2^*, \qquad \partial_2^*(e_3^*) = e_3^*.
 \end{gather*}
 Let $(A, \cdot, \Phi)$ be the associated commutative differential algebra of $(A, \ast, \Phi)$, where $\cdot$ is defined by equation~\eqref{eq:Zin} whose non-zero product is given by
 \begin{equation*}
 e_1 \cdot e_1 = e_2 \cdot e_2 = 2 e_3, \qquad
 e_1 \cdot e_2 = e_2 \cdot e_1 = e_3.
 \end{equation*}
 Let $(A, [\ ,\ ], \cdot)$ be the induced Poisson algebra of $(A, \cdot, \Phi)$, where $[\ ,\ ]$ is defined by equation~\eqref{eq:lie} whose non-zero product is given by
 \begin{equation*}
 [e_1, e_2] = -[e_2, e_1] = -3 e_3.
 \end{equation*}
 Let $(A, \diamond, \ast)$ be the induced pre-Poisson algebra of $(A, \ast, \Phi)$, where $\diamond$ is defined by equation~\eqref{eq:dzpp} whose non-zero product is given by
 \begin{equation*}
 e_1 \diamond e_1 = - e_3, \qquad
 e_1 \diamond e_2 = -2 e_3, \qquad
 e_2 \diamond e_1 = e_3, \qquad
 e_2 \diamond e_2 = - e_3.
 \end{equation*}
 Hence the non-zero product of $[\ ,\ ]$ and $\cdot$ on the Poisson algebra $(A \ltimes_{-L^*_{\diamond},L^*_{\ast}} A^*, [\ ,\ ], \cdot)$ is given respectively by
 \begin{gather*}
 [e_1, e_2] = -[e_2, e_1] = -3e_3, \qquad
 [e_1, e_3^*] = -[e_3^*, e_1] = e_1^* + 2 e_2^*, \\
 [e_2, e_3^*] = -[e_3^*, e_2] = -e_1^* + e_2^*;\\
 e_1 \cdot e_1 = e_2 \cdot e_2 = 2 e_3, \qquad
 e_1 \cdot e_2 = e_2 \cdot e_1 = e_3, \qquad
 e_1 \cdot e_3^* = e_3^* \cdot e_1 = e_1^*, \\
 e_2 \cdot e_3^* = e_3^* \cdot e_2 = e_1^* + e_2^*.
 \end{gather*}
 Set $r = \sum_{i=1}^3 (e_i \otimes e_i^* - e_i^* \otimes e_i)$.
 Define $\delta$ and $\Delta$ by equations~\eqref{eq:pcbd1} and \eqref{eq:pcbd2} respectively.
 Thus we have
 \begin{gather*}
 \delta(e_1) = e_3 \otimes e_1^* - e_1^* \otimes e_3 - e_3 \otimes e_2^* + e_2^* \otimes e_3, \\
 \delta(e_2) = 2 e_3 \otimes e_1^* - 2 e_1^* \otimes e_3^* + e_3 \otimes e_2^* - e_2^* \otimes e_3,\\
 \delta(e_3) = \delta(e_1^*) = \delta(e_2^*) = 0, \qquad \delta(e_3^*) = 3 e_1^* \otimes e_2^* - 3 e_2^* \otimes e_1^*;\\
 \Delta(e_1) = - e_3 \otimes e_1^* - e_1^* \otimes e_3 - e_2^* \otimes e_3 - e_3 \otimes e_2^*, \\
 \Delta(e_2) = - e_3 \otimes e_2^* - e_2^* \otimes e_3, \\
 \Delta(e_3) = \Delta(e_1^*) = \Delta(e_2^*) = 0, \qquad \Delta(e_3^*) = -2 e_1^* \otimes e_1^* -2 e_2^* \otimes e_2^* - e_2^* \otimes e_1^* - e_1^* \otimes e_2^*.
 \end{gather*}
 Therefore $(A \ltimes_{-L^*_{\diamond},L^*_{\ast}} A^*, [\ ,\ ], \cdot, \delta, \Delta)$ is a Poisson bialgebra.
At the same time, we obtain a~commutative and cocommutative differential ASI bialgebra $(A \ltimes_{L^*_{\ast}} A^*, \cdot, \Delta, \Phi - \Phi^*, -\Phi + \Phi^*)$
 whose induced Poisson bialgebra is exactly $(A \ltimes_{-L^*_{\diamond},L^*_{\ast}} A^*, [\ ,\ ], \cdot, \delta,\Delta)$.
\end{Example}

\subsection*{Acknowledgements}
This work is supported by NSFC (11931009, 12271265, 12261131498), the Fundamental Research Funds for the Central Universities and Nankai Zhide Foundation.
The authors thank the referees for valuable suggestions to improve the paper.

\pdfbookmark[1]{References}{ref}
\LastPageEnding

\end{document}